\documentclass[reqno]{amsart}
\usepackage{mathrsfs}
\usepackage{color}
\usepackage{amsmath}
\usepackage{amsfonts}
\usepackage{amssymb}
\usepackage{graphicx}
\usepackage{hyperref}

\newtheorem{Theorem}{Theorem}[section]
\newtheorem{Corollary}[Theorem]{Corollary}
\newtheorem{Lemma}[Theorem]{Lemma}

\newtheorem{Definition}[Theorem]{Definition}

\newtheorem{Conjecture}[Theorem]{Conjecture}
\newtheorem{Problem}[Theorem]{Problem}
\newtheorem{Remark}[Theorem]{Remark}

\numberwithin{equation}{section}

\begin{document}

	\title[concavity property of minimal $L^{2}$ integrals]
	{Concavity property of minimal $L^{2}$ integrals with Lebesgue measurable gain \uppercase\expandafter{\romannumeral2}}

	\author{Qi'an Guan}
	\address{Qi'an Guan: School of
		Mathematical Sciences, Peking University, Beijing 100871, China.}
	\email{guanqian@math.pku.edu.cn}

	\author{Zhitong Mi}
	\address{Zhitong Mi: School of Mathematics and Statistics, Beijing Jiaotong University, Beijing
		100044, China.
	}
	\email{zhitongmi@amss.ac.cn}
	\author{Zheng Yuan}
	\address{Zheng Yuan: Institute of Mathematics, Academy of Mathematics
		and Systems Science, Chinese Academy of Sciences, Beijing 100190, China.}
	\email{yuanzheng@amss.ac.cn}

	\thanks{}
	

	
	\date{\today}
	
	\dedicatory{}
	
	\commby{}
	
	
	\begin{abstract}
		In this article, we present a concavity property of the minimal $L^{2}$ integrals related to multiplier ideal sheaves
		with Lebesgue measurable gain on weakly pseudoconvex K\"ahler manifolds.
		As applications, we give a necessary condition for the concavity degenerating to linearity,
		and a characterization for the equality in  optimal jets $L^2$ extension problem on open Riemann surfaces.
	\end{abstract}
	
	\maketitle
	
	\section{Introduction}
	
	The truth of Demailly's strong openness conjecture (SOC for short), i.e. $\mathcal{I}(\varphi)=\mathcal{I}_+(\varphi):=\mathop{\cup} \limits_{\epsilon>0}\mathcal{I}((1+\epsilon)\varphi)$
	is an important feature of multiplier ideal sheaves and used in the study of several complex variables, complex algebraic geometry and complex differential geometry
	(see e.g. \cite{GZSOC,K16,cao17,cdM17,FoW18,DEL18,ZZ2018,GZ20,berndtsson20,ZZ2019,ZhouZhu20siu's,FoW20,KS20,DEL21}),
	where $\varphi$ is a plurisubharmonic function (see \cite{Demaillybook}),
	and the multiplier ideal sheaf $\mathcal{I}(\varphi)$ is the sheaf of germs of holomorphic functions $f$ such that $|f|^2e^{-\varphi}$ is locally integrable (see e.g. \cite{Tian},\cite{Nadel},\cite{Siu96},\cite{DEL},\cite{DK01},\cite{DemaillySoc},\cite{DP03},\cite{Lazarsfeld},\cite{Siu05},\cite{Siu09},\cite{DemaillyAG},\cite{Guenancia}).

	The truth of Demailly's strong openness conjecture was proved by Guan-Zhou \cite{GZSOC} (the 2-dimensional case was proved by Jonsson-Musta\c{t}$\breve{a}$
	\cite{JonssonMustata}). When $\mathcal{I}(\varphi)=\mathcal{O}$, the SOC degenerates to the openness conjecture (OC for short), which was posed by Demailly-Koll\'ar \cite{DK01} and proved by Berndtsson \cite{Berndtsson2} (the 2-dimensional case was proved by Favre-Jonsson in \cite{FavreJonsson}).
	
	Recall that Berndtsson \cite{Berndtsson2} proved the OC by establishing an effectiveness result of the OC.
	Stimulated by Berndtsson's effectiveness result, and continuing the solution of the SOC \cite{GZSOC},
	Guan-Zhou \cite{GZeff} established an effectiveness result of the SOC by considering (for the first time) the minimal $L^{2}$ integral related to multiplier ideal sheaves on the sublevel set $\{\varphi<0\}$,
	i.e. the pseudoconvex domain $D$.

	In \cite{G16}, by considering  all the minimal $L^{2}$ integrals on the sublevels of the weights $\varphi$,
	Guan established a sharp version of the effectiveness result of the SOC
	and presented a concavity property of the minimal $L^2$ integrals.
	After that, Guan \cite{G2018} generalized the concavity property for smooth gain on Stein manifolds.

	In \cite{GM}, Guan-Mi gave some applications of the concavity property: a necessary condition for the concavity degenerating to linearity,
	a characterization for 1-dimensional case,
	and a characterization for  the equality in optimal $L^2$ extension problem on open Riemann surfaces with subharmonic weights.
	In \cite{GY-concavity}, Guan-Yuan obtained the concavity property with Lebesgue measurable gain on Stein manifolds with applications:
	necessary conditions for the concavity degenerating to linearity,
	characterizations for 1-dimensional case,
	and a characterization for the equality in optimal $L^2$ extension problem  on open Riemann surfaces with weights may not be subharmonic.
	
	In the present article, we generalize the above concavity property to weakly pseudoconvex K\"ahler manifolds with applications:
	a necessary condition for the concavity degenerating to linearity,
	and a characterization for  the equality in optimal jets $L^2$ extension problem on open Riemann surfaces.
	
	\subsection{Concavity property of minimal $L^2$ integrals on weakly pseudoconvex K\"ahler manifolds}
	
	\
	
	Let $M$ be a complex manifold. Let $X$ and $Z$ be closed subsets of $M$. We say that a triple $(M,X,Z)$ satisfies condition $(A)$, if the following statements hold:
	
	$\uppercase\expandafter{\romannumeral1}.$ $X$ is a closed subset of $M$ and $X$ is locally negligible with respect to $L^2$ holomorphic functions; i.e., for any local coordinated neighborhood $U\subset M$ and for any $L^2$ holomorphic function $f$ on $U\backslash X$, there exists an $L^2$ holomorphic function $\tilde{f}$ on $U$ such that $\tilde{f}|_{U\backslash X}=f$ with the same $L^2$ norm;
	
	$\uppercase\expandafter{\romannumeral2}.$ $Z$ is an analytic subset of $M$ and $M\backslash (X\cup Z)$ is a weakly pseudoconvex K\"ahler manifold.
	
	Let $M$ be an $n-$dimensional complex manifold, and let $(M,X,Z)$  satisfy condition $(A)$. Let $K_M$ be the canonical line bundle on $M$. Let $\psi$ be a plurisubharmonic function on $M$ such that $\{\psi<-t\}\backslash (X\cup Z)$ is a weakly pseudoconvex K\"ahler manifold for any $t\in\mathbb{R}$. Let $\varphi$ be a Lebesgue measurable function on $M$, such that $\psi+\varphi$ is a plurisubharmonic function on $M$. Denote $T=-\sup\limits_M \psi$.
	\begin{Definition}
		We say that a positive measurable function $c$ (so-called ``\textbf{gain}") on $(T,+\infty)$ in class $P_{T,M}$ if the following two statements hold:
		\par
		$(1)$ $c(t)e^{-t}$ is decreasing with respect to $t$;
		\par
		$(2)$ there is a closed subset $E$ of $M$ such that $E\subset Z\cap \{\psi(z)=-\infty\}$ and for any compact subset $K\subset M\backslash E$, $e^{-\varphi}c(-\psi)$ has a positive lower bound on $K$.
	\end{Definition}

	Let $Z_0$ be a subset of $\{\psi=-\infty\}$ such that $Z_0 \cap
	Supp(\mathcal{O}/\mathcal{I}(\varphi+\psi))\neq \emptyset$. Let $U \supset Z_0$ be
	an open subset of $M$, and let $f$ be a holomorphic $(n,0)$ form on $U$. Let $\mathcal{F}_{z_0} \supset \mathcal{I}(\varphi+\psi)_{z_0}$ be an ideal of $\mathcal{O}_{z_0}$ for any $z_{0}\in Z_0$.
	
	Denote
	\begin{equation}
		\begin{split}
			\inf\{\int_{ \{ \psi<-t\}}|\tilde{f}|^2e^{-\varphi}c(-\psi): \tilde{f}\in
			H^0(\{\psi<-t\},\mathcal{O} (K_M)  ) \\
			\&\, (\tilde{f}-f)\in
			H^0(Z_0 ,(\mathcal{O} (K_M) \otimes \mathcal{F})|_{Z_0})\}
		\end{split}
	\end{equation}
	by $G(t)$, where $t\in[T,+\infty)$, $c$ is a nonnegative function on $(T,+\infty)$, $|f|^2:=\sqrt{-1}^{n^2}f\wedge \bar{f}$ for any $(n,0)$ form $f$ and $(\tilde{f}-f)\in
	H^0(Z_0 ,(\mathcal{O} (K_M) \otimes \mathcal{F})|_{Z_0} )$ means $(\tilde{f}-f,z_0)\in(\mathcal{O}(K_M)\otimes \mathcal{F})_{z_0}$ for all $z_0\in Z_0$.
	If there is no holomorphic $(n,0)$ form $\tilde{f}$ on $\{\psi< -t\}$ satisfying $(\tilde{f}-f)\in
	H^0(Z_0 ,(\mathcal{O} (K_M) \otimes \mathcal{F})|_{Z_0} )$, we set $G(t)=+\infty$.
	
	In the present article, we obtain the following concavity for $G(t)$ on weakly pseudoconvex K\"ahler manifolds.
	\begin{Theorem}(Main theorem)
		Let $c\in\mathcal{P}_{T,M}$. If there exists $t \in [T,+\infty)$ satisfying that $G(t)<+\infty$, then $G(h^{-1}(r))$ is concave with respect to  $r\in (\int_{T_1}^{T}c(t)e^{-t}dt,\int_{T_1}^{+\infty}c(t)e^{-t}dt)$, $\lim\limits_{t\to T+0}G(t)=G(T)$ and $\lim\limits_{t \to +\infty}G(t)=0$, where $h(t)=\int_{T_1}^{t}c(t_1)e^{-t_1}dt_1$ and $T_1 \in (T,+\infty)$.
		\label{maintheorem}
	\end{Theorem}
	When $M$ is a Stein manifold, Theorem \ref{maintheorem} can be referred to \cite{GY-concavity}. By our recent work on the minimal $L^2$ integrals with boundary points, the assumption  `$\{\psi<-t\}\backslash (X\cup Z)$ is a weakly pseudoconvex K\"ahler manifold for any $t\in\mathbb{R}$' is not necessary for the truth of concavity property of the minimal $L^2$
	integrals, see \cite{GMY-boundary2}.
	
	\begin{Remark}
		\label{infty2}Let $c\in\mathcal{P}_{T,M}$.	If  $\int_{T_1}^{+\infty}c(t)e^{-t}dt=+\infty$ and $f\notin H^{0}(Z_0,
		(\mathcal{O}(K_{M})\otimes\mathcal{F})|_{Z_0})$, then $G(t)=+\infty$ for any $t\geq T$. Thus, when there exists $t \in [T,+\infty)$ satisfying that $G(t)\in(0,+\infty)$, we have $\int_{T_1}^{+\infty}c(t)e^{-t}dt<+\infty$ and $G(\hat{h}^{-1}(r))$ is concave with respect to  $r\in (0,\int_{T}^{+\infty}c(t)e^{-t}dt)$, where $\hat{h}(t)=\int_{t}^{+\infty}c(l)e^{-l}dl$. We give a proof of this remark in section \ref{sec:4.2}.
	\end{Remark}

	Let $c(t)$ be a nonnegative measurable function on $(T,+\infty)$. Set
	\begin{equation}\nonumber
		\begin{split}
			\mathcal{H}^2(c,t)=\{\tilde{f}:\int_{ \{ \psi<-t\}}|\tilde{f}|^2e^{-\varphi}c(-\psi)<+\infty,\  \tilde{f}\in
			H^0(\{\psi<-t\},\mathcal{O} (K_M)  ) \\
			\& (\tilde{f}-f)\in
			H^0(Z_0 ,(\mathcal{O} (K_M) \otimes \mathcal{F})|_{Z_0} )\},
		\end{split}
	\end{equation}
	where $t\in[T,+\infty)$.
	
	As a corollary of Theorem \ref{maintheorem}, we give a necessary condition for the concavity degenerating to linearity (related results can be referred to \cite{GM}, \cite{GY-concavity} and \cite{xu-phd}).
	\begin{Corollary}
		\label{linear case of G}
		Let $c(t)\in \mathcal{P}_{T,M}$, if $G(t)\in(0,+\infty)$ for some $t\ge T$ and $G(\hat{h}^{-1}(r))$ is linear with respect to $r\in[0,\int_T^{+\infty}c(s)e^{-s}ds)$, where $\hat{h}(t)=\int_{t}^{+\infty}c(l)e^{-l}dl$, then there exists a unique holomorphic $(n,0)$ form $F$ on $M$ satisfying $(F-f)\in
		H^0(Z_0 ,(\mathcal{O} (K_M) \otimes \mathcal{F})|_{Z_0})$ and $G(t;c)=\int_{\{\psi<-t\}}|F|^2e^{-\varphi}c(-\psi)$ for any $t\ge T$.
		
		Furthermore
		\begin{equation}
			\begin{split}
				\int_{\{-t_1\le\psi<-t_2\}}|F|^2e^{-\varphi}a(-\psi)=\frac{G(T_1;c)}{\int_{T_1}^{+\infty}c(t)e^{-t}dt}
				\int_{t_2}^{t_1}a(t)e^{-t}dt
				\label{other a also linear}
			\end{split}
		\end{equation}
		for any nonnegative measurable function $a$ on $(T,+\infty)$, where $T\le t_2<t_1\le+\infty$ and $T_1\in (T,+\infty)$.
	\end{Corollary}
	
	\begin{Remark}
		\label{rem:linear}
		If $\mathcal{H}^2(\tilde{c},t_0)\subset\mathcal{H}^2(c,t_0)$ for some $t_0\ge T$, we have
		\begin{equation}
			\begin{split}
				G(t_0;\tilde{c})=\int_{\{\psi<-t_0\}}|F|^2e^{-\varphi}\tilde{c}(-\psi)=\frac{G(T_1;c)}{\int_{T_1}^{+\infty}c(t)e^{-t}dt}
				\int_{t_0}^{+\infty}\tilde{c}(s)e^{-s}ds,
				\label{other c also linear}
			\end{split}
		\end{equation}
		where $\tilde{c}$ is a nonnegative measurable function on $(T,+\infty)$.
	\end{Remark}

	\subsection{Equality in optimal jets $L^2$ extension on open Riemann surfaces}
	\
	
	The celebrated $L^2$ extension theorem which was firstly established by Ohsawa-Takegoshi \cite{OT87} has been widely studied and used in the research of  several complex variables and complex geometry by many mathematicians, e.g., Berndtsson, Demailly, Ohsawa, Siu, et al (see \cite{berndtsson1996,berndtsson annals,Demaillyshm,DemaillyManivel, Demailly15,Ohsawa2,Ohsawa3,Ohsawa4,Ohsawa5,berndtsson paun,DHP,HPS}).
	
	Finding optimal constant in  the $L^2$ extension theorem is a natural question and is highly related to the so-called Suita conjecture (see \cite{suita72} and see Ohsawa \cite{Ohsawa5} for the relation between optimal constant in the $L^2$ extension theorem and the Suita conjecture).  Guan-Zhou-Zhu (see \cite{ZGZ}, see also \cite{GZZCRMATH}) firstly used a method of undetermined functions to study optimal constant in  the $L^2$ extension theorem. For bounded
	pseudoconvex domains in $\mathbb{C}^n$, Blocki \cite{Blocki-inv} developed the equation of undetermined functions, and got the optimal version of Ohsawa-Takegoshi's $L^2$ extension theorem in \cite{OT87} which deduced the inequality part of Suita conjecture for planar domains. Using undetermined functions method, Guan-Zhou (see \cite{GZsci}, see also \cite{guan-zhou CRMATH2012}) proved the optimal $L^2$ extension theorem with negligible weight on Stein manifolds, and obtained the inequality part of Suita conjecture for open Riemann surfaces, which is the original form of the inequality part of Suita conjecture in \cite{suita72}. In \cite{guan-zhou13ap}, Guan-Zhou established an $L^2$ extension theorem with optimal estimate in a general setting on Stein manifolds, which gave various optimal versions of $L^2$ extension theorem. As an application, Guan-Zhou \cite{guan-zhou13ap} proved the equality part of Suita conjecture, which finished the proof of Suita conjecture. In \cite{ZZ2018}, Zhou-Zhu proved an optimal $L^2$ extension theorem on weakly pseudoconvex K\"ahler manifolds.
	
	Let us introduce the Suita conjecture.
	Let $M=\Omega$ be an open Riemann surface admitting a nontrivial Green function $G_{\Omega}$.  Let $w$ be a local coordinate on a neighborhood $V_{z_0}$ of $z_0\in\Omega$ satisfying $w(z_0)=0$.
	Let $c_{\beta}(z)$ be the logarithmic capacity (see \cite{S-O69}) on $\Omega$ which is defined by
	$$c_{\beta}(z_0):=\exp \lim\limits_{z\to z_0}(G_{\Omega}(z,z_0)-\log|w(z)|).$$
	Let $B_{\Omega}(z_0)$ be the Bergman kernel function on $\Omega$.
	\begin{Conjecture}
		[Suita conjecture \cite{suita72}]
		$\pi B_{\Omega}(z_0)\ge(c_{\beta}(z_0))^2$ holds, and the equality holds if and only if $\Omega$ is
		conformally equivalent to the unit disc less a (possible) closed set of inner capacity zero.
	\end{Conjecture}

	Let us focus on the equality part of Suita conjecture.
Following from Ohsawa's \cite{Ohsawa5} observation, the equality part of Suita conjecture is a characterization for the equality in optimal $L^2$ extension problem:
	
	\emph{$\inf\{\int_{\Omega}|F|^2: F$ is a holomorphic  $(1,0)$ form on $\Omega$ with $F(z_0)=dw\}= \frac{2\pi}{(c_{\beta}(z_0))^2}$ holds if and only if $\Omega$ is
		conformally equivalent to the unit disc less a (possible) closed set of inner capacity zero.}

	In \cite{popovici05}, Popovici  established a $k$-jets $L^2$ extension theorem, which is a generalization of Ohsawa-Takegoshi's $L^2$ extension theorem. In \cite{Demailly15}, Demailly  proved an $L^2$ extension theorem from a possibly non-reduced subvariety. Hosono \cite{hosono} proved an optimal $k$-jets $L^2$ extension theorem for bounded pseudoconvex domain with Green type  function and trivial gain ($c\equiv 1$).  Rao-Zhang \cite{RaoZhang 21} proved an $L^2$ extension theorem for jets with variable denominators. Zhou-Zhu \cite{ZhouZhuJAG} proved an optimal $L^2$ extension theorem from  non-reduced subvariety on  weakly pseudoconvex K\"ahler manifolds for smooth gain which optimized Demailly's result in \cite{Demailly15} (for the Stein case see also Li-Xu-Zhou \cite{LiXuZhou23}).
	
The equality part of Suita conjecture corresponds to the equality in the optimal 0-jet $L^2$ extension problem on open Riemann surfaces with trivial weights $\varphi\equiv0$ and trivial gain $c\equiv1$.
	Then it is natural to ask
	
	\begin{Problem}\label{pb:equality}
		Can one characterize  the equality in optimal $k$-jets $L^2$ extension problem on open Riemann surfaces,
		where $k$ is a nonnegative integer?
	\end{Problem}

	Assume that $k=0$. When the weights $\varphi$ are harmonic and gain is trivial ($c\equiv1$),
	Guan-Zhou \cite{guan-zhou13ap} gave an affirmative answer to 0-jet version of Problem \ref{pb:equality}, which was called extended Suita conjecture conjectured by Yamada \cite{Yamada}.
	When the weights $\varphi$ are subharmonic and gain is smooth, using the concavity of minimal $L^2$ integrals, Guan-Mi \cite{GM} gave an affirmative answer to 0-jet version of Problem \ref{pb:equality}.
	When the weights $\varphi$ may not be subharmonic and gain is Lebesgue measurable, using the concavity of minimal $L^2$ integrals, Guan-Yuan \cite{GY-concavity} gave an affirmative answer to 0-jet version of Problem \ref{pb:equality}.

	In the present article, we give a characterization for the equality in  optimal k-jets $L^2$ extension problem on open Riemann surfaces \emph{by using the concavity property of minimal $L^2$ integrals,} which gives an affirmative answer to Problem \ref{pb:equality}.

	To state our results, we recall some notations (see \cite{OF81}, see also \cite{guan-zhou13ap,GY-concavity}).
	Let $\Omega$ be an open Riemann surface admitting a nontrivial Green function $G_{\Omega}$, and let $z_0\in\Omega$.
	Let $p:\Delta\rightarrow\Omega$ be the universal covering from unit disc $\Delta$ to $\Omega$.
	we call the holomorphic function $f$ (resp. holomorphic $(1,0)$ form $F$) on $\Delta$ a multiplicative function (resp. multiplicative differential (Prym differential)),
	if there is a character $\chi$, which is the representation of the fundamental group of $\Omega$, such that $g^{\star}f=\chi(g)f$ (resp. $g^{\star}(F)=\chi(g)F$),
	where $|\chi|=1$ and $g$ is an element of the fundamental group of $\Omega$. Denote the set of such kinds of $f$ (resp. $F$) by $\mathcal{O}^{\chi}(\Omega)$ (resp. $\Gamma^{\chi}(\Omega)$).
	
	It is known that for any harmonic function $u$ on $\Omega$,
	there exists a $\chi_{u}$ and a multiplicative function $f_u\in\mathcal{O}^{\chi_u}(\Omega)$,
	such that $|f_u|=p^{\star}e^{u}$.
	If $u_1-u_2=\log|f|$, then $\chi_{u_1}=\chi_{u_2}$,
	where $u_1$ and $u_2$ are harmonic functions on $\Omega$ and $f$ is a holomorphic function on $\Omega$.
	Recall that for the Green function $G_{\Omega}(z,z_0)$,
	there exists a $\chi_{z_0}$ and a multiplicative function $f_{z_0}\in\mathcal{O}^{\chi_{z_0}}(\Omega)$, such that $|f_{z_0}(z)|=p^{\star}e^{G_{\Omega}(z,z_0)}$ (see \cite{suita72}).
	
	We present the affirmative answer to Problem \ref{pb:equality} as follows.
	
	\begin{Theorem}\label{thm:1d-extension}
		Let $k$ be a nonnegative integer. Let $\psi$ be a negative  subharmonic function on $\Omega$ satisfying that   $a=\frac{1}{2}v(dd^{c}\psi,z_0)>0$. Let $\varphi$ be a Lebesgue measurable function on $\Omega$  such that $\varphi+\psi$ is subharmonic on $\Omega$, $\frac{1}{2}v(dd^{c}(\varphi+\psi),z_0)=k+1$ and $\alpha:=(\varphi+\psi-2(k+1)G_{\Omega}(\cdot,z_0))(z_0)>-\infty$. Let $c(t)$ be a positive measurable function on $(0,+\infty)$ satisfying $c(t)e^{-t}$ is decreasing on $(0,+\infty)$ and $\int_{0}^{+\infty}c(t)e^{-t}dt<+\infty$.
		
		Let $w$ be a local coordinate on a neighborhood $V_{z_0}$ of $z_0\in\Omega$ satisfying $w(z_0)=0$, and let $f=w^kdw$ be a holomorphic $(1,0)$ form on $V_{z_0}$. Let $I$ be an ideal of $\mathcal{O}_{z_0}$, which is generated by $w$. Then there exists a holomorphic $(1,0)$ form $F$ on $\Omega$ such that $(F-f,z_0)\in I^{k+1}\otimes \mathcal{O}(K_{\Omega})_{z_0}$ and
		\begin{equation}
			\label{eq:210902a}
			\int_{\Omega}|F|^2e^{-\varphi}c(-\psi)\leq(\int_0^{+\infty}c(t)e^{-t}dt)\frac{2\pi e^{-\alpha}}{a(c_{\beta}(z_0))^{2(k+1)}}.
		\end{equation}
		
		Moreover, equality $(\int_0^{+\infty}c(t)e^{-t}dt)\frac{2\pi e^{-\alpha}}{a(c_{\beta}(z_0))^{2(k+1)}}=\inf\{\int_{\Omega}|\tilde{F}|^2e^{-\varphi}c(-\psi):\tilde{F}$ is a holomorphic $(1,0)$ form such that $(\tilde{F}-f,z_0)\in I^{k+1}\otimes \mathcal{O}(K_{\Omega})_{z_0}$$\}$ holds if and only if the following statements hold:
		
		$(1)$ $\varphi+\psi=2\log|g|+2(k+1)G_{\Omega}(\cdot,z_0)+2u$, where $g$ is a holomorphic function on $\Omega$ such that $g(z_0)\not=0$ and $u$ is a harmonic function on $\Omega$;
		
		$(2)$  $\psi=2aG_{\Omega}(\cdot,z_0)$ on $\Omega$;
		
		$(3)$ $\chi_{-u}=(\chi_{z_0})^{k+1}$, where $\chi_{-u}$ and $\chi_{z_0}$ are the  characters associated to the functions $-u$ and $G_{\Omega}(\cdot,z_0)$ respectively.
	\end{Theorem}
	
	In fact, for any $z_0\in\Omega$ and any $k\geq0$ be a integer, there exists a harmonic function $u$ satisfying $\chi_{-u}=(\chi_{z_0})^{k+1}$ (see Lemma \ref{l:chi} for details).
	
	\begin{Remark}
		\label{r:min}	
		When the  three statements in Theorem \ref{thm:1d-extension} hold, $p_{\star}(f_uf_{z_0}^kdf_{z_0})$ is a holomorphic $(1,0)$ form on $\Omega$ and there exists a holomorphic function $h_1(z)$ on $V_{z_0}$ such that $gp_{\star}(f_uf_{z_0}^kdf_{z_0})=h_1w^{k}dw$, where $g$ is the holomorphic function in statement $(1)$. Note that $|h_1(z_0)|=e^{\frac{\alpha}{2}}c_{\beta}^{k+1}(z_0)$, and denote $c_0=\frac{1}{h_1(z_0)}$. Then  $c_0gp_{\star}(f_uf_{z_0}^kdf_{z_0})$ is the unique holomorphic $(1,0)$ form $F$ on $\Omega$ satisfying $({F}-f,z_0)\in I^{k+1}\otimes \mathcal{O}(K_{\Omega})_{z_0}$ and $\int_{\Omega}|F|^2e^{-\varphi}c(-\psi)\leq(\int_0^{+\infty}c(t)e^{-t}dt)\frac{2\pi e^{-\alpha}}{a(c_{\beta}(z_0))^{2(k+1)}}$. We prove the present remark in Section \ref{sec:3.3}.\end{Remark}

	Let $h$ be a Lebesgue measurable function on $\Omega$ such that $h+a_1G_{\Omega}(\cdot,z_0)$ is subharmonic on $\Omega$ and $\alpha_1:=(h+a_1G_{\Omega}(\cdot,z_0))(z_0)>-\infty$, where $a_1>-k-1$ is a real number.
	Denote that $\rho:=e^{-2h}$.
	
	Denote that
	$$K^{(k)}_{\Omega,\rho}(z_0):=\frac{2}{\inf\{\int_{\Omega}|\tilde{F}|^2\rho:\tilde{F}\in H^0(\Omega,\mathcal{O}(K_{\Omega}))\,\&\,(\tilde{F}-f,z_0)\in I^{k+1}\otimes \mathcal{O}(K_{\Omega})_{z_0}\}}.$$
	Especially, when $\rho=1$ and $\Omega$ is a planar domain, $K^{(k)}_{\Omega,\rho}(z_0)$ degenerates to the Bergman kernel for higher derivatives (see \cite{Berg70,Blo18}).

	Choosing $c\equiv1$ and $\psi=2(a_1+k+1)G_{\Omega}(\cdot,z_0)$, Theorem \ref{thm:1d-extension} degenerates to the following
	
	\begin{Theorem}
		\label{thm:k-suita}
		Let $\Omega$ be an open Riemann surface admitting a nontrivial Green function $G_{\Omega}$, and let $z_0\in\Omega$ and $k$ is a nonnegative integer. Then we have
		\begin{equation}
			\label{eq:210902b}
			(c_{\beta}(z_0))^{2(k+1)}\leq\frac{\pi}{a_1+k+1}e^{-2\alpha_1}K^{(k)}_{\Omega,\rho}(z_0).
		\end{equation}
		Moreover,  equality $(c_{\beta}(z_0))^{2(k+1)}=\frac{\pi}{a_1+k+1}e^{-2\alpha_1}K^{(k)}_{\Omega,\rho}(z_0)$ holds if and only if the following statements hold:
		
		$(1)$ $h=\log|g|-a_1G_{\Omega}(\cdot,z_0)+u$, where $g$ is a holomorphic function on $\Omega$ such that $g(z_0)\not=0$ and $u$ is a harmonic function on $\Omega$;
		
		$(2)$ $\chi_{-u}=(\chi_{z_0})^{k+1}$, where $\chi_{-u}$ and $\chi_{z_0}$ are the  characters associated to the functions $-u$ and $G_{\Omega}(\cdot,z_0)$ respectively.
	\end{Theorem}
	
	Note that $\chi_{z_0}=1$ if and only if $\Omega$ is
	conformally equivalent to the unit disc less a (possible) closed set of inner capacity zero (see \cite{suita72}).
When $k=0$, $a_1=0,$ and $h\equiv0$, Theorem \ref{thm:k-suita} is the solution of Suita conjecture (see \cite{suita72,Blocki-inv,guan-zhou CRMATH2012,guan-zhou13ap}). Theorem \ref{thm:k-suita} implies the following result.
	\begin{Corollary}
		If $(c_{\beta}(z_0))^{2(k+1)}=\frac{\pi}{a_1+k+1}e^{-2\alpha_1}K^{(k)}_{\Omega,\rho}(z_0)$ holds for some $k\geq0$, then for any $n\geq1$,
		$$(c_{\beta}(z_0))^{2n(k+1)}=\frac{\pi}{na_1+nk+n}e^{-2n\alpha_1}K^{(nk+n-1)}_{\Omega,\rho^n}(z_0)$$
		holds.
	\end{Corollary}
	
	\section{Preparation}
	
	In this section, we present some preparations.
	
	\subsection{$L^2$ method on weakly pseudoconvex K\"ahler manifold}
	\
	
	We call a positive measurable function $c$ on $(S,+\infty)$ is in class $\tilde{\mathcal{P}}_S$ if $\int_S^s c(l)e^{-l}dl<+\infty$ for some $s>S$ and $c(t)e^{-t}$ is decreasing with respect to $t$.
	
	In this section, we present the following two lemmas.
	\begin{Lemma}
		Let $B \in (0, +\infty)$ and $t_0 \ge S$ be arbitrarily given. Let $(M,\omega)$ be
		an $n-$dimensional weakly pseudoconvex K\"ahler manifold. Let $\psi < -S$ be a
		plurisubharmonic function on $M$. Let $\varphi$ be a plurisubharmonic function on $M$.
		Let F be a holomorphic $(n,0)$ form on $\{\psi< -t_0\}$ such that
		\begin{equation}\nonumber
			\int_{K\cap \{\psi<-t_0\}} {|F|}^2<+\infty,
		\end{equation}
		for any compact subset $K$ of $M,$ and
		\begin{equation}\nonumber
			\int_M \frac{1}{B} \mathbb{I}_{\{-t_0-B< \psi < -t_0\}}  {|F|}^2
			e^{{-}\varphi}\le C <+\infty.
		\end{equation}
		Then there exists a holomorphic $(n,0)$ form $\widetilde F$ on $M$, such that
		\begin{equation}
			\int_M {|\widetilde F-(1-b_{t_0,B}(\psi))F|}^2
			e^{{-}\varphi+v_{t_0,B}(\psi)}c(-v_{t_0,B}(\psi))\le C\int^{t_0+B}_{S}c(t)e^{{-}t}dt.
		\end{equation}
		where $b_{t_0,B}(t)=\int^{t}_{-\infty}\frac{1}{B} \mathbb{I}_{\{-t_0-B< s < -t_0\}}ds$,
		$v_{t_0,B}(t)=\int^{t}_{-t_0}b_{t_0,B}(s)ds-t_0$ and $c(t)\in \tilde{\mathcal{P}}_S$.
		\label{lemma2.1}
	\end{Lemma}
	
	\par
	We will prove Lemma \ref{lemma2.1} in  appendix. Lemma \ref{lemma2.1} implies the following lemma, which will be used in the proof of the Theorem \ref{maintheorem}.
	\begin{Lemma}
		Let $(M,X,Z)$  satisfy condition $(A)$ and $c(t)\in \mathcal{P}_{T,M}$. Let $B \in (0, +\infty)$ and $t_0>t_1 > T$ be arbitrarily given. Let $\psi < -T$ be a
		plurisubharmonic function on $M$ such that $\{\psi<-t\}\backslash (X\cup Z)$ is a weakly pseudoconvex K\"ahler manifold for any $t\ge T$. Let $\varphi$ be a Lebesgue measurable function on $M$ such that $\varphi+\psi$ is plurisubharmonic on $M$ .
		Let F be a holomorphic $(n,0)$ form on $\{\psi< -t_0\}$ such that
		\begin{equation}
			\int_{\{\psi<-t_0\}} {|F|}^2e^{-\varphi}c(-\psi)<+\infty,
			\label{condition of lemma 2.2}
		\end{equation}
		Then there exists a holomorphic $(n,0)$ form $\tilde{F}$ on $\{\psi<-t_1\}$ such that
		\begin{equation}
			\int_{\{\psi<-t_1\}}|\tilde{F}-(1-b_{t_0,B}(\psi))F|^2e^{-\varphi-\psi+v_{t_0,B}(\psi)}c(-v_{t_0,B}(\psi))
			\le C\int_{t_1}^{t_0+B}c(t)e^{-t}dt,
		\end{equation}
		where $C=\int_M \frac{1}{B} \mathbb{I}_{\{-t_0-B< \psi < -t_0\}}  {|F|}^2
		e^{{-}\varphi-\psi}$, $b_{t_0,B}(t)=\int^{t}_{-\infty}\frac{1}{B} \mathbb{I}_{\{-t_0-B< s < -t_0\}}ds$ and
		$v_{t_0,B}(t)=\int^{t}_{-t_0}b_{t_0,B}(s)ds-t_0$.
		\label{lemma 2.2}
	\end{Lemma}
	\begin{proof} It follows from inequality \eqref{condition of lemma 2.2} and $\inf_{t\in(t_0,t_0+B)}c(t)>0$ that
		$$C=\int_M \frac{1}{B} \mathbb{I}_{\{-t_0-B< \psi < -t_0\}}  {|F|}^2
		e^{{-}\varphi-\psi}<+\infty.$$
		As $(M,X,Z)$  satisfies condition $(A)$ and $c(t)\in \mathcal{P}_{T,M}$, then  $M\backslash (Z\cup X)$ is a weakly pseudoconvex K\"ahler manifold and there exists a closed subset $E\subset Z\cap \{\psi=-\infty\}$ such that  $e^{-\varphi}c(-\psi)$ has locally positive lower bound on $M\backslash E$.
		
		Combining inequality \eqref{condition of lemma 2.2} and $e^{-\varphi}c(-\psi)$ has locally positive lower bound on $M\backslash E$, we obtain that
		$$\int_{K\cap\{\psi<-t_0\}}|F|^2<+\infty$$
		holds for any compact subset $K$ of $M\backslash (Z\cup X)$. Then Lemma \ref{lemma2.1} shows that there exists a holomorphic $(n,0)$ form $\tilde{F}_Z$ on $\{\psi<-t_1\}\backslash (Z\cup X)$, such that
		\begin{equation}\nonumber
			\int_{\{\psi<-t_1\}\backslash (Z\cup X)}|\tilde{F}_Z-(1-b_{t_0,B}(\psi))F|^2e^{-\varphi-\psi+v_{t_0,B}(\psi)}c(-v_{t_0,B}(\psi))
			\le C\int_{t_1}^{t_0+B}c(t)e^{-t}dt.
		\end{equation}
		For any $z\in \{\psi<-t_1\}\cap((Z\cup X)\backslash E)$, there exists an open neighborhood $V_z$ of $z$, such that $V_z\subset\subset \{\psi<-t_1\}\backslash E$. Note that $c(t)e^{-t}$ is decreasing on $(t,+\infty)$ and $v_{t_0,B}(\psi)\ge \psi$, then we have
		\begin{equation}\nonumber
			\begin{split}
				&\int_{V_z\backslash (Z\cup X)}
				|\tilde{F}_Z-(1-b_{t_0,B}(\psi))F|^2e^{-\varphi}c(-\psi)\\
				\le& \int_{V_z\backslash (Z\cup X)}
				|\tilde{F}_Z-(1-b_{t_0,B}(\psi))F|^2e^{-\varphi-\psi+v_{t_0,B}(\psi)}c(-v_{t_0,B}(\psi))\\
				<&+\infty.
			\end{split}
		\end{equation}
		Note that there exists a positive number $C_1>0$ such that $e^{-\varphi}c(-\psi)>C_1$ on $V_z$ and
		$\int_{V_z\backslash (Z\cup X)}
		|(1-b_{t_0,B}(\psi))F|^2e^{-\varphi}c(-\psi)\le\int_{\{\psi<-t_0\}}
		|(1-b_{t_0,B}(\psi))F|^2e^{-\varphi}c(-\psi)<+\infty$, then we have
		\begin{equation}\nonumber
			\begin{split}
				&\int_{V_z\backslash (Z\cup X)}|\tilde{F}_Z|^2\\
				\le&C_1 \int_{V_z\backslash (Z\cup X)}|\tilde{F}_Z|^2e^{-\varphi}c(-\psi)\\
				\le&2C_1
				(\int_{V_z\backslash (Z\cup X)}|\tilde{F}_Z-(1-b_{t_0,B}(\psi))F|^2e^{-\varphi}c(-\psi)+\int_{V_z\backslash (Z\cup X)}
				|(1-b_{t_0,B}(\psi))F|^2e^{-\varphi}c(-\psi))\\
				<&+\infty.
			\end{split}
		\end{equation}
		As $Z\cup X$ is locally negligible with respect to $L^2$ holomorphic function, hence we can find a holomorphic extension $\tilde{F}_E$ of $\tilde{F}_Z$ from $\{\psi<-t_1\}\backslash (Z\cup X)$ to $\{\psi<-t_1\}\backslash E$ such that
		\begin{equation}\nonumber
			\int_{\{\psi<-t_1\}\backslash E}|\tilde{F}_E-(1-b_{t_0,B}(\psi))F|^2e^{-\varphi-\psi+v_{t_0,B}(\psi)}c(-v_{t_0,B}(\psi))
			\le C\int_{t_1}^{t_0+B}c(t)e^{-t}dt.
		\end{equation}
		Note that $E\subset\{\psi<-t_0\}\subset \{\psi<-t_1\}$, for any $z\in E$, there exists an open neighborhood $U_z$ of $z$ such that $U_z\subset\subset \{\psi<-t_0\}$. As $\varphi+\psi$ is plurisubharmonic on $M$ and $e^{v_{t_0,B}(\psi)}c(-v_{t_0,B}(\psi))$ has positive lower bound on $\{\psi<-t_1\}$, then we have
		\begin{equation}\nonumber
			\begin{split}
				&\int_{U_z\backslash E}
				|\tilde{F}_E-(1-b_{t_0,B}(\psi))F|^2\\
				\le& C_2\int_{\{\psi<-t_1\}\backslash E}
				|\tilde{F}_E-(1-b_{t_0,B}(\psi))F|^2e^{-\varphi-\psi+v_{t_0,B}(\psi)}c(-v_{t_0,B}(\psi))\\
				<&+\infty.
			\end{split}
		\end{equation}
		where $C_2$ is some positive number. As $U_z\subset\subset \{\psi<-t_0\}$, we have
		$$\int_{U_z}|(1-b_{t_0,B}(\psi))F|^2<\int_{U_z}|F|^2<+\infty.$$
		Combining the two inequality above, we obtain that $\int_{U_z\backslash E} |\tilde{F}_E|^2<+\infty.$
		
		As $E$ is contained in some analytic subset of $M$, we can find a holomorphic extension $\tilde{F}$ of $\tilde{F}_E$ from $\{\psi<-t_1\}\backslash E$ to $\{\psi<-t_1\}$ such that
		\begin{equation}\nonumber
			\int_{\{\psi<-t_1\}}|\tilde{F}-(1-b_{t_0,B}(\psi))F|^2e^{-\varphi-\psi+v_{t_0,B}(\psi)}c(-v_{t_0,B}(\psi))
			\le C\int_{t_1}^{t_0+B}c(t)e^{-t}dt.
		\end{equation}
		Lemma \ref{lemma 2.2} is proved.
		
	\end{proof}

	\subsection{Some properties of $G(t)$}
	\
	
	In the present section, we prove some properties related to $G(t)$.
	
	We firstly introduce a property of coherent analytic sheaves and a convergence property of holomorphic $(n,0)$ form.
	\begin{Lemma}(see \cite{G-R})
		\label{closedness}
		Let $N$ be a submodule of $\mathcal O_{\mathbb C^n,o}^q$, $1\leq q<+\infty$, let $f_j\in\mathcal O_{\mathbb C^n}(U)^q$ be a sequence of $q-$tuples holomorphic in an open neighborhood $U$ of the origin $o$. Assume that the $f_j$ converge uniformly in $U$ towards  a $q-$tuples $f\in\mathcal O_{\mathbb C^n}(U)^q$, assume furthermore that all germs $(f_{j},o)$ belong to $N$. Then $(f,o)\in N$.	
	\end{Lemma}
	
	\begin{Lemma}
		\label{l:converge}(see \cite{GY-concavity})
		Let $M$ be a complex manifold. Let $S$ be an analytic subset of $M$.  	
		Let $\{g_j\}_{j=1,2,...}$ be a sequence of nonnegative Lebesgue measurable functions on $M$, which satisfies that $g_j$ are almost everywhere convergent to $g$ on  $M$ when $j\rightarrow+\infty$,  where $g$ is a nonnegative Lebesgue measurable function on $M$. Assume that for any compact subset $K$ of $M\backslash S$, there exist $s_K\in(0,+\infty)$ and $C_K\in(0,+\infty)$ such that
		$$\int_{K}{g_j}^{-s_K}dV_M\leq C_K$$
		for any $j$, where $dV_M$ is a continuous volume form on $M$.
		
		Let $\{F_j\}_{j=1,2,...}$ be a sequence of holomorphic $(n,0)$ form on $M$. Assume that $\liminf_{j\rightarrow+\infty}\int_{M}|F_j|^2g_j\leq C$, where $C$ is a positive constant. Then there exists a subsequence $\{F_{j_l}\}_{l=1,2,...}$, which satisfies that $\{F_{j_l}\}$ is uniformly convergent to a holomorphic $(n,0)$ form $F$ on $M$ on any compact subset of $M$ when $l\rightarrow+\infty$, such that
		$$\int_{M}|F|^2g\leq C.$$
	\end{Lemma}

	Let $(M,X,Z)$ satisfy condition $(A)$, and let $c\in\mathcal{P}_{T,M}$.
	Let $Z_0$ be a subset of $\{\psi=-\infty\}$ such that $Z_0 \cap
	Supp(\mathcal{O}/\mathcal{I}(\varphi+\psi))\neq \emptyset$. Let $U \supset Z_0$ be
	an open subset of $M$. Let $\mathcal{F} \supset \mathcal{I}(\varphi+\psi)|_U$ be a coherent subsheaf of $\mathcal{O}$ on $U$ and let $f$ be a holomorphic $(1,0)$ form on $U$.
	\par
	The following lemma is a characterization of $G(t)= 0$, where $t\geq T$.

	\begin{Lemma}
		The following two statements are equivalent:
		\par
		$(1)$ $f \in
		H^0(Z_0,(\mathcal{O} (K_M) \otimes \mathcal{F})|_{Z_0} )$.
		\par
		$(2)$ $G(t) = 0$.
		\label{G equal to 0}
	\end{Lemma}
	\begin{proof}
		$(1)$ obviously implies $(2)$.
		
		Now we prove $(2)$ implies $(1)$. As $G(t)=0$, then there exists holomorphic $(n,0)$ form $\{\tilde{f}_j\}_{j\in \mathbb{N}^+}$ on $\{\psi<-t\}$ such that $\lim\limits_{j \to +\infty}\int_{\{\psi<-t\}}|\tilde{f}_j|^2e^{-\varphi}c(-\psi)=0$ and $(\tilde{f}_j-f)\in H^0(Z_0,(\mathcal{O} (K_M) \otimes \mathcal{F})|_{Z_0})$ for any $j$. As $e^{-\varphi}c(-\psi)$ has positive lower bound on any compact subset of $M\backslash Z$, where $Z$ is some analytic subset of $M$, it follows from Lemma \ref{l:converge} that there exists a subsequence of $\{\tilde{f}_j\}_{j\in \mathbb{N}^+}$ also denoted by $\{\tilde{f}_j\}_{j\in \mathbb{N}^+}$ that compactly convergent to $0$. It is clear that $\tilde{f}_j-f$ is compactly convergent to $0-f=-f$ on $U\cap \{\psi<-t\}$. It follows from Lemma \ref{closedness} that $f\in H^0(Z_0,(\mathcal{O} (K_M) \otimes \mathcal{F})|_{Z_0} )$. This prove Lemma \ref{G equal to 0}.
	\end{proof}

	The following lemma shows the existence and uniqueness of the holomorphic $(n,0)$ form related to $G(t)$.
	\begin{Lemma}
		\label{existence of F}
		Assume that $G(t)<+\infty$ for some $t\in [T,+\infty)$. Then there exists a unique
		holomorphic $(n,0)$ form $F_t$ on $\{\psi<-t\}$ satisfying
		$$\ \int_{\{\psi<-t\}}|F_t|^2e^{-\varphi}c(-\psi)=G(t)$$  and
		$\ (F_t-f)\in
		H^0(Z_0,(\mathcal{O} (K_M) \otimes \mathcal{F})|_{Z_0} ).$
		\par
		Furthermore, for any holomorphic $(n,0)$ form $\hat{F}$ on $\{\psi<-t\}$ satisfying
		$$\int_{\{\psi<-t\}}|\hat{F}|^2e^{-\varphi}c(-\psi)<+\infty$$ and $(\hat{F}-f)\in
		H^0(Z_0,(\mathcal{O} (K_M) \otimes \mathcal{F})|_{Z_0} ).$ We have the following equality
		\begin{equation}
			\begin{split}
				&\int_{\{\psi<-t\}}|F_t|^2e^{-\varphi}c(-\psi)+
				\int_{\{\psi<-t\}}|\hat{F}-F_t|^2e^{-\varphi}c(-\psi)\\
				=&\int_{\{\psi<-t\}}|\hat{F}|^2e^{-\varphi}c(-\psi).
				\label{orhnormal F}
			\end{split}
		\end{equation}
	\end{Lemma}
	
	\begin{proof}
		Firstly, we prove the existence of $F_t$. As $G(t)<+\infty$, then there exists a sequence of holomorphic $(n,0)$ form $\{f_j\}_{j\in \mathbb{N}^+}$ on $\{\psi<-t\}$ such that $\lim\limits_{j \to +\infty}\int_{\{\psi<-t\}}|f_j|^2e^{-\varphi}c(-\psi)=G(t)$ and $(f_j-f)\in H^0(Z_0,(\mathcal{O} (K_M) \otimes \mathcal{F})|_{Z_0})$ for any $j$.  As $e^{-\varphi}c(-\psi)$ has a positive lower bound on any compact subset of $M\backslash Z$, where $Z$ is a analytic subset of $M$, it follows from Lemma \ref{l:converge} that there exists a subsequence of $\{f_j\}$ compactly converging to a holomorphic $(n,0)$ form $F$ on $\{\psi<-t\}$ satisfying $\int_{\{\psi<-t\}}|F|^2e^{-\varphi}c(-\psi)\le G(t)$. It follows from Lemma \ref{closedness} that $(F-f)\in H^0(Z_0,(\mathcal{O} (K_M) \otimes \mathcal{F})|_{Z_0})$. Then we obtain the existence of $F_t(=F)$.
		
		We prove the uniqueness of $F_t$ by contradiction: if not, there exists
		two different holomorphic $(n,0)$ forms $f_1$ and $f_2$ on $\{\psi<-t\}$
		satisfying $\int_{\{\psi<-t\}}|f_1|^2e^{-\varphi}$  $c(-\psi)=
		\int_{\{\psi<-t\}}|f_2|^2e^{-\varphi}c(-\psi)=G(t)$, $(f_1-f)\in
		H^0(Z_0,(\mathcal{O} (K_M) \otimes \mathcal{F})|_{Z_0})$ and $(f_2-f)\in
		H^0(Z_0,(\mathcal{O} (K_M) \otimes \mathcal{F})|_{Z_0})$. Note that
		\begin{equation}
			\begin{split}
				\int_{\{\psi<-t\}}|\frac{f_1+f_2}{2}|^2e^{-\varphi}c(-\psi)+
				\int_{\{\psi<-t\}}|\frac{f_1-f_2}{2}|^2e^{-\varphi}c(-\psi)\\
				=\frac{1}{2}(\int_{\{\psi<-t\}}|f_1|^2e^{-\varphi}c(-\psi)+
				\int_{\{\psi<-t\}}|f_1|^2e^{-\varphi}c(-\psi))=G(t),
			\end{split}
		\end{equation}
		then we obtain that
		\begin{equation}\nonumber
			\begin{split}
				\int_{\{\psi<-t\}}|\frac{f_1+f_2}{2}|^2e^{-\varphi}c(-\psi)
				\le G(t)
			\end{split}
		\end{equation}
		and $(\frac{f_1+f_2}{2}-f)\in H^0(Z_0,(\mathcal{O} (K_M) \otimes \mathcal{F})|_{Z_0})$, which contradicts  the definition of $G(t)$.
		\par
		Now we prove the equality \eqref{orhnormal F}. For any holomorphic $h$ on $\{\psi<-t\}$
		satisfying $\int_{\{\psi<-t\}}|h|^2e^{-\varphi}c(-\psi)<+\infty$ and $h \in
		H^0(Z_0,(\mathcal{O} (K_M) \otimes \mathcal{F})|_{Z_0})$.  It is clear that for any complex
		number $\alpha$, $F_t+\alpha h$ satisfying $((F_t+\alpha h)-f) \in
		H^0(Z_0,(\mathcal{O} (K_M) \otimes \mathcal{F})|_{Z_0})$ and
		$\int_{\{\psi<-t\}}|F_t|^2e^{-\varphi}c(-\psi) \leq \int_{\{\psi<-t\}}|F_t+\alpha
		h|^2e^{-\varphi}c(-\psi)$. Note that
		\begin{equation}\nonumber
			\begin{split}
				\int_{\{\psi<-t\}}|F_t+\alpha
				h|^2e^{-\varphi}c(-\psi)-\int_{\{\psi<-t\}}|F_t|^2e^{-\varphi}c(-\psi)\geq 0
			\end{split}
		\end{equation}
		(By considering $\alpha \to 0$) implies
		\begin{equation}\nonumber
			\begin{split}
				\mathfrak{R} \int_{\{\psi<-t\}}F_t\bar{h}e^{-\varphi}c(-\psi)=0,
			\end{split}
		\end{equation}
		then we have
		\begin{equation}\nonumber
			\begin{split}
				\int_{\{\psi<-t\}}|F_t+h|^2e^{-\varphi}c(-\psi)=
				\int_{\{\psi<-t\}}(|F_t|^2+|h|^2)e^{-\varphi}c(-\psi).
			\end{split}
		\end{equation}
		\par
		Letting $h=\hat{F}-F_t$, we obtain equality \eqref{orhnormal F}.
	\end{proof}

	The following lemma shows the  lower semicontinuity property of $G(t)$.
	\begin{Lemma}$G(t)$ is decreasing with respect to $t\in
		[T,+\infty)$, such that $\lim \limits_{t \to t_0+0}G(t)=G(t_0)$ for any $t_0\in
		[T,+\infty)$, and if $G(t)<+\infty$ for some $t>T$, then $\lim \limits_{t \to +\infty}G(t)=0$. Especially, $G(t)$ is lower semicontinuous on $[T,+\infty)$.
		\label{semicontinuous}
	\end{Lemma}
	\begin{proof}
		By the definition of $G(t)$, it is clear that $G(t)$ is decreasing on
		$[T,+\infty)$. If $G(t)<+\infty$ for some $t>T$, by the dominated convergence theorem, we know $\lim\limits_{t\to +\infty}G(t)=0$. It suffices
		to prove $\lim \limits_{t \to t_0+0}G(t)=G(t_0)$ . We prove it by
		contradiction: if not, then $\lim \limits_{t \to t_0+0}G(t)<
		G(t_0)<+\infty$.
		\par
		By Lemma \ref{existence of F}, for any $t>t_0$, there exists a unique holomorphic $(n,0)$ form
		$F_t$ on $\{\psi<-t\}$ satisfying
		$\int_{\{\psi<-t\}}|F_t|^2e^{-\varphi}c(-\psi)=G(t)$ and $(F_t-f) \in
		H^0(Z_0,(\mathcal{O} (K_M) \otimes \mathcal{F})|_{Z_0})$. Note that $G(t)$ is decreasing
		implies that $\int_{\{\psi<-t\}}|F_t|^2e^{-\varphi}c(-\psi)\leq \lim
		\limits_{t \to t_0+0}G(t)$ for any $t>t_0$. If $\lim\limits_{t \to t_0+0}G(t)=+\infty$, the equality $\lim \limits_{t \to t_0+0}G(t)=G(t_0)$ is clear, thus it suffices to prove the case $\lim\limits_{t \to t_0+0}G(t)<+\infty$. As $e^{-\varphi}c(-\psi)$ has positive lower bound on any compact subset of $M\backslash Z$, where $Z$ is some analytic subset of $M$, and $\int_{\{\psi<-t\}}|F_t|^2e^{-\varphi}c(-\psi)\le \lim\limits_{t \to t_0+0}G(t)<+\infty$ holds for any $t\in (t_0,t_1]$, where $t_1>t_0$ is a fixed number, it follows from Lemma \ref{l:converge} that there exists $F_{t_j}$ $(t_j\to t_0+0,$ as $j \to +\infty)$ uniformly convergent on any compact subset of $\{\psi<-t_1\}$. Using diagonal method, we obtain a subsequence of $\{F_t\}$ (also denoted by $\{F_{t_j}\}$) convergent on any compact subset of $\{\psi<-t_0\}$.
		
		Let $\hat{F}_{t_0}:=\lim \limits_{j \to +\infty}F_{t_j}$, which is a
		holomorphic $(n,0)$ form on $\{\psi<-t_0\}$. Then it follows from the decreasing property of $G(t)$ that
		\begin{equation}\nonumber
			\begin{split}
				\int_{K}|\hat{F}_{t_0}|^2e^{-\varphi}c(-\psi)
				&\leq
				\liminf\limits_{j \to +\infty}\int_{K}|F_{t_j}|^2e^{-\varphi}c(-\psi)\\
				&\leq
				\liminf \limits_{j \to +\infty}G(t_j)\\
				&\leq \lim \limits_{t\to t_0+0} G(t)
			\end{split}
		\end{equation}
		for any compact set $K \subset \{\psi<-t_0\}$. It follows from Levi's theorem that
		\begin{equation}\nonumber
			\begin{split}
				\int_{\{\psi<-t_0\}}|\hat{F}_{t_0}|^2e^{-\varphi}c(-\psi)
				\leq \lim \limits_{t\to t_0+0} G(t).
			\end{split}
		\end{equation}
		It follows from Lemma \ref{closedness} that $(\hat{F}_{t_0}-f)\in H^0(Z_0,(\mathcal{O} (K_M) \otimes \mathcal{F})|_{Z_0}).$ Then we obtain that $G(t_0)\leq
		\int_{\{\psi<-t_0\}}|\hat{F}_{t_0}|^2e^{-\varphi}c(-\psi)
		\leq \lim \limits_{t\to t_0+0} G(t)$
		which contradicts $\lim \limits_{t\to t_0+0} G(t) <G(t_0)$.
	\end{proof}

	We consider the derivatives of $G(t)$ in the following lemma.
	
	\begin{Lemma}
		Assume that $G(t_1)<+\infty$, where $t_1\in (T,+\infty)$. Then for any $t_0>t_1$, we have
		\begin{equation}\nonumber
			\begin{split}
				\frac{G(t_1)-G(t_0)}{\int^{t_0}_{t_1} c(t)e^{-t}dt}\leq
				\liminf\limits_{B \to
					0+0}\frac{G(t_0)-G(t_0+B)}{\int_{t_0}^{t_0+B}c(t)e^{-t}dt},
			\end{split}
		\end{equation}
		i.e.
		\begin{equation}\nonumber
			\frac{G(t_0)-G(t_1)}{\int_{T_1}^{t_0}
				c(t)e^{-t}dt-\int_{T_1}^{t_1} c(t)e^{-t}dt} \geq
			\limsup \limits_{B \to 0+0}
			\frac{G(t_0+B)-G(t_0)}{\int_{T_1}^{t_0+B}
				c(t)e^{-t}dt-\int_{T_1}^{t_0} c(t)e^{-t}dt}.
		\end{equation}
		\label{derivatives of G}
	\end{Lemma}
	\begin{proof}
		It follows from Lemma \ref{semicontinuous} that $G(t)<+\infty$ for any $t>t_1$. By Lemma \ref{existence of F}, there exists a holomorphic $(n,0)$ form $F_{t_0}$ on $\{\psi<-t_0\}$, such that $(F_{t_0}-f)\in H^0(Z_0,(\mathcal{O} (K_M) \otimes \mathcal{F})|_{Z_0}) $ and $G(t_0)=\int_{\{\psi<-t_0\}}|F_{t_0}|^2e^{-\varphi}c(-\psi)$.
		
		It suffices to consider that $\liminf\limits_{B\to 0+0} \frac{G(t_0)-G(t_0+B)}{\int_{t_0}^{t_0+B}c(t)e^{-t}dt}\in [0,+\infty)$ because of the decreasing property of $G(t)$. Then there exists $B_j\to 0+0$ (as $j\to+\infty$) such that
		\begin{equation}
			\label{eq:211106e}\lim\limits_{j\to +\infty} \frac{G(t_0)-G(t_0+B_j)}{\int_{t_0}^{t_0+B_j}c(t)e^{-t}dt}=\liminf\limits_{B\to 0+0} \frac{G(t_0)-G(t_0+B)}{\int_{t_0}^{t_0+B}c(t)e^{-t}dt}\end{equation}
		and $\{\frac{G(t_0)-G(t_0+B_j)}{\int_{t_0}^{t_0+B_j}c(t)e^{-t}dt}\}_{j\in\mathbb{N}^{+}}$ is bounded. As $c(t)e^{-t}$ is decreasing and positive on $(t,+\infty)$, then
		\begin{equation}\label{eq:210908a}
			\begin{split}
				\lim\limits_{j\to +\infty} \frac{G(t_0)-G(t_0+B_j)}{\int_{t_0}^{t_0+B_j}c(t)e^{-t}dt}
				=&(\lim\limits_{j\to +\infty} \frac{G(t_0)-G(t_0+B_j)}{B_j})(\frac{1}{\lim\limits_{t\to t_0+0}c(t)e^{-t}})\\
				=&(\lim\limits_{j\to +\infty} \frac{G(t_0)-G(t_0+B_j)}{B_j})(\frac{e^{t_0}}{\lim\limits_{t\to t_0+0}c(t)}).
			\end{split}
		\end{equation}
		Hence $\{\frac{G(t_0)-G(t_0+B_j)}{B_j}\}_{j\in\mathbb{N}^+}$ is bounded with respect to $j$.
		
		As $t \leq v_j(t)$, the decreasing property of $c(t)e^{-t}$ shows that
		\begin{equation}\nonumber
			e^{-\psi+v_j(\psi)}c(-v_j(\psi))\geq c(-\psi).
		\end{equation}
		\par
		By Lemma \ref{lemma 2.2}, for any $B_j$, there exists holomorphic
		$(n,0)$ form $\tilde{F}_j$ on $\{\psi<-t_1\}$ such that $(\tilde{F}_j-F_{t_0})\in H^0(Z_0,(\mathcal{O} (K_M) \otimes \mathcal{I}(\varphi+\psi))|_{Z_0})\subset H^0(Z_0,(\mathcal{O} (K_M) \otimes \mathcal{F})|_{Z_0})$
		and
		\begin{flalign}
			&\int_{\{\psi<-t_1\}}|\tilde{F}_j-(1-b_{t_0,B_j}(\psi))F_{t_0}|^2e^{-\varphi}c(-\psi)\nonumber\\
			\leq &
			\int_{\{\psi<-t_1\}}|\tilde{F}_j-(1-b_{t_0,B_j}(\psi))F_{t_0}|^2e^{-\varphi}e^{-\psi+v_j(\psi)}c(-v_j(\psi))\nonumber\\
			\leq &
			\int^{t_0+B_j}_Tc(t)e^{-t}dt\int_{\{\psi<-t_1\}}\frac{1}{B_j}
			\mathbb{I}_{\{-t_0-B_j<\psi<-t_0\}}|F_{t_0}|^2e^{-\varphi-\psi}\nonumber\\
			\leq &
			\frac{e^{t_0+B_j}\int^{t_0+B_j}_Tc(t)e^{-t}dt}{\inf
				\limits_{t\in(t_0,t_0+B_j)}c(t)}\int_{\{\psi<-t_1\}}\frac{1}{B_j}
			\mathbb{I}_{\{-t_0-B_j<\psi<-t_0\}}|F_{t_0}|^2e^{-\varphi}c(-\psi)\nonumber\\
			= &
			\frac{e^{t_0+B_j}\int^{t_0+B}_Tc(t)e^{-t}dt}{\inf
				\limits_{t\in(t_0,t_0+B_j)}c(t)}\times
			(\int_{\{\psi<-t_1\}}\frac{1}{B_j}\mathbb{I}_{\{\psi<-t_0\}}|F_{t_0}|^2e^{-\varphi}c(-\psi)\nonumber\\
			&-\int_{\{\psi<-t_1\}}\frac{1}{B_j}\mathbb{I}_{\{\psi<-t_0-B_j\}}|F_{t_0}|^2e^{-\varphi}c(-\psi))\nonumber\\
			\leq &
			\frac{e^{t_0+B_j}\int^{t_0+B}_Tc(t)e^{-t}dt}{\inf
				\limits_{t\in(t_0,t_0+B_j)}c(t)} \times
			\frac{G(t_0)-G(t_0+B_j)}{B_j}.
			\label{219}
		\end{flalign}
		Firstly, we will prove that $\int_{\{\psi<-t_1\}} |\tilde{F}_j|^2e^{-\varphi}c(-\psi)$ is uniformly bounded
		with respect to $j$.
		Note that
		\begin{equation}
			\begin{split}
				&(\int_{\{\psi<-t_1\}}|\tilde{F}_j-(1-b_{t_0,B_j}(\psi))F_{t_0}|^2e^{-\varphi}c(-\psi))^{1/2}\\
				\geq &
				(\int_{\{\psi<-t_1\}}|\tilde{F}_j|^2e^{-\varphi}c(-\psi))^{1/2}-
				(\int_{\{\psi<-t_1\}}|(1-b_{t_0,B_j}(\psi))F_{t_0}|^2e^{-\varphi}c(-\psi))^{1/2},
			\end{split}
		\end{equation}
		then it follows from inequality (\ref{219}) that
		\begin{equation}
			\begin{split}
				&(\int_{\{\psi<-t_1\}}|\tilde{F}_j|^2e^{-\varphi}c(-\psi))^{1/2} \\
				\leq
				&(\frac{e^{t_0+B_j}\int^{t_0+B}_Tc(t)e^{-t}dt}{\inf
					\limits_{t\in(t_0,t_0+B_j)}c(t)})^{1/2} \times
				(\frac{G(t_0)-G(t_0+B_j)}{B_j})^{1/2}\\
				&+
				\int_{\{\psi<-t_1\}}|(1-b_{t_0,B_j}(\psi))F_{t_0}|^2e^{-\varphi}c(-\psi))^{1/2}.
			\end{split}
		\end{equation}
		Since $\{\frac{G(t_0)-G(t_0+B_j)}{B_j}\}_{j\in \mathbb{N}^+}$  is bounded and
		$0\leq b_{t_0,B_j}(\psi) \leq 1$, then $\int_{\{\psi<-t_1\}}|\tilde{F}_j|^2e^{-\varphi}$ $c(-\psi)dV_X$
		is uniformly bounded with respect to $j$.\par
		Now we prove the main result.
		
		It follows from $\int_{\{\psi<-t_1\}}|\tilde{F}_{j}|^{2}e^{-\varphi}c(-\psi)$ is bounded with respect to $j$ and Lemma \ref{l:converge} that there exists a subsequence of $\{\tilde{F}_j\}$, denoted by $\{\tilde{F}_{j_k}\}_{k\in\mathbb{N}^+}$, which is uniformly convergent to a holomorphic $(n,0)$ form $F_1$ on $\{\psi<-t_1\}$ on any compact subset of $\{\psi<-t_1\}$ when $k\rightarrow+\infty$,  such that
		$$\int_{\{\psi<-t_1\}}|F_1|^2e^{-\varphi}c(-\psi)\le\liminf_{j\rightarrow+\infty}\int_{\{\psi<-t_1\}}|\tilde{F}_{j}|^{2}e^{-\varphi}c(-\psi)<+\infty.$$
		As $(\tilde{F}_j-F_{t_{0}})\in H^{0}(Z_0,(\mathcal{O}(K_{M})\otimes\mathcal{F})|_{Z_0})$ for any $j$, we have $(F_1-F_{t_{0}})\in H^{0}(Z_0,(\mathcal{O}(K_{M})\otimes\mathcal{F})|_{Z_0})$.
		By direct calculation, we have
		\begin{equation*}
			\lim_{j\rightarrow+\infty}b_{t_0,B_j}(t)=\lim_{j\rightarrow+\infty}\int_{-\infty}^{t}\frac{1}{B_j}\mathbb{I}_{\{-t_0-B_j<s<-t_0\}}ds=\left\{ \begin{array}{lcl}
				0 & \mbox{if}& x\in(-\infty,-t_0)\\
				1 & \mbox{if}& x\in[-t_0,+\infty)
			\end{array} \right.
		\end{equation*}
		and
		\begin{equation*}
			\lim_{j\rightarrow+\infty}v_{t_0,B_j}(t)=\lim_{j\rightarrow+\infty}\int_{-t_0}^{t}b_{t_0,B_j}ds-t_0=\left\{ \begin{array}{lcl}
				-t_0 & \mbox{if}& x\in(-\infty,-t_0)\\
				t & \mbox{if}& x\in[-t_0,+\infty)
			\end{array}. \right.
		\end{equation*}
		Following from equality \eqref{eq:210908a}, inequality \eqref{219} and the Fatou's Lemma, we have
		\begin{equation}
			\label{eq:211106b}\begin{split}
				&\int_{\{\psi<-t_0\}}|F_1-F_{t_0}|^2e^{-\varphi-\psi-t_0}c(t_0)+\int_{\{-t_0\le\psi<-t_1\}}|F_1|^2e^{-\varphi}c(-\psi)\\
				=&\int_{\{\psi<-t_1\}}\lim_{k\rightarrow+\infty}|\tilde{F}_{j_k}-(1-b_{t_{0},B_{j_k}}(\psi))F_{t_{0}}|^{2}e^{-\varphi}e^{-\psi+v_{t_0,B_{j_k}}(\psi)}c(-v_{t_0,B_{j_k}}(\psi))
				\\
				\le&\liminf_{k\rightarrow+\infty}\int_{\{\psi<-t_1\}}|\tilde{F}_{j_k}-(1-b_{t_{0},B_{j_k}}(\psi))F_{t_{0}}|^{2}e^{-\varphi}e^{-\psi+v_{t_0,B_{j_k}}(\psi)}c(-v_{t_0,B_{j_k}}(\psi))
				\\
				\leq&
				\liminf_{k\rightarrow+\infty}(\frac{e^{t_{0}+B_{j_k}}\int_{t_1}^{t_{0}+B_{j_k}}c(t)e^{-t}dt}{\inf_{t\in(t_{0},t_{0}+B_{j_k})}c(t)}
				\times\frac{G(t_{0})-G(t_{0}+B_{j_k})}{B_{j_k}})\\
				=&\frac{e^{t_0}\int_{t_1}^{t_0}c(t)e^{-t}dt}{\lim_{t\rightarrow t_0+0}c(t)}\lim_{j\rightarrow +\infty}\frac{G(t_0)-G(t_0+B_j)}{B_j}\\
				=&\int_{t_1}^{t_0}c(t)e^{-t}dt\lim_{j\rightarrow +\infty}\frac{G(t_0)-G(t_0+B_j)}{\int_{t_0}^{t_0+B_j}c(t)e^{-t}dt}.
			\end{split}
		\end{equation}
		
		Since $c(t)e^{-t}$ is decreasing on $(T,+\infty)$, we have $e^{\psi}c(-\psi)\le e^{-t_0}c(t_0)$ on $\{\psi<-t_0\}$. It follows Lemma \ref{existence of F}, equality \eqref{eq:211106e} and inequality \eqref{eq:211106b} that
		\begin{equation}
			\label{eq:211106c}
			\begin{split}&\int_{t_1}^{t_0}c(t)e^{-t}dt\liminf_{B\rightarrow 0+0}\frac{G(t_0)-G(t_0+B)}{\int_{t_0}^{t_0+B}c(t)e^{-t}dt}\\
				=&\int_{t_1}^{t_0}c(t)e^{-t}dt\lim_{j\rightarrow +\infty}\frac{G(t_0)-G(t_0+B_j)}{\int_{t_0}^{t_0+B_j}c(t)e^{-t}dt}\\
				\ge&	\int_{\{\psi<-t_0\}}|F_1-F_{t_0}|^2e^{-\varphi-\psi-t_0}c(t_0)+\int_{\{-t_0\le\psi<-t_1\}}|F_1|^2e^{-\varphi}c(-\psi)\\
				\ge&\int_{\{\psi<-t_0\}}|F_1-F_{t_0}|^2e^{-\varphi}c(-\psi)+\int_{\{-t_0\le\psi<-t_1\}}|F_1|^2e^{-\varphi}c(-\psi)\\
				=&\int_{\{\psi<-t_1\}}|F_1|^2e^{-\varphi}c(-\psi)-\int_{\{\psi<-t_0\}}|F_{t_0}|^2e^{-\varphi}c(-\psi)\\
				\ge& G(t_1)-G(t_0).
			\end{split}
		\end{equation}
		This proves Lemma
		\ref{derivatives of G}.
	\end{proof}
	
	The following property of concave
	functions will be used in the proof of Theorem \ref{maintheorem}.
	\begin{Lemma}(see \cite{G16})
		Let $H(r)$ be a lower semicontinuous function on $(0,R]$. Then $H(r)$ is concave
		if and only if
		\begin{equation}\nonumber
			\begin{split}
				\frac{H(r_1)-H(r_2)}{r_1-r_2} \leq
				\liminf\limits_{r_3 \to r_2-0}
				\frac{H(r_3)-H(r_2)}{r_3-r_2}
			\end{split}
		\end{equation}
		holds for any $0<r_2<r_1 \leq R$.
		\label{characterization of concave function}
	\end{Lemma}
	
	\subsection{Some results used in the proof of Theorem \ref{thm:1d-extension}}
	\
	
	In this section, we give some results which will be used in the proof of Theorem \ref{thm:1d-extension}.

	Let $M=\Omega$ be an open Riemann surface admitting a nontrivial Green function $G_{\Omega}$.  Let $\psi$ be a subharmonic function on $\Omega$ satisfying $T=-\sup_{\Omega}\psi=0$, and let $\varphi$ be a Lebesgue measurable function on $\Omega$, such that $\varphi+\psi$ is subharmonic on $\Omega$. Let $Z_0=z_0$ be a point in $\Omega$. Let $c(t)$ is a positive measurable function on $(0,+\infty)$,  satisfying that $c(t)e^{-t}$ is decreasing on $(0,+\infty)$ and $e^{-\varphi}c(-\psi)$ has locally positive lower bound on $\Omega\backslash Z_1$, where $Z_1\subset\{\psi=-\infty\}$ is a discrete point subset of $\Omega.$

	Let $w$ be a local coordinate on a neighborhood $V_{z_0}$ of $z_0\in\Omega$ satisfying $w(z_0)=0$. Set $f=w^kdw$ on $V_{z_0}$, where $f$ is the holomorphic $(1,0)$ form in the definition of $G(t)$.

	The following  theorem (see \cite{GY-concavity}) gives a characterization of $G(\hat{h}^{-1}(r))$ is linear with respect to $r\in(0,\int_{0}^{+\infty}c(l)e^{-l}dl)$. Set $d^c=\frac{1}{2\pi i}(\partial-\bar\partial)$.

	\begin{Theorem}
		\label{thm:1-d-linear}(see \cite{GY-concavity})
		Assume that $(\psi-2aG_{\Omega}(\cdot,z_0))(z_0)>-\infty$, where $a=\frac{1}{2}v(dd^{c}\psi,z_0)$, and  $G(0)\in(0,+\infty)$. Then $G(\hat{h}^{-1}(r))$ is linear with respect to $r$ if and only if the following statements hold:

		$(1)$ $\varphi+\psi=2\log|g|+2G_{\Omega}(\cdot,z_0)+2u$, $ord_{z_0}(g)=k$ and $\mathcal{F}_{z_0}=\mathcal{I}(\varphi+\psi)_{z_0}$, where $g$ is a holomorphic function on $\Omega$ and $u$ is a harmonic function on $\Omega$;
		
		$(2)$ $a>0$ and $\psi=2aG_{\Omega}(\cdot,z_0)$ on $\Omega$;
		
		$(3)$ $\chi_{-u}=\chi_{z_0}.$
	\end{Theorem}

	\begin{Lemma}(see \cite{S-O69}, see also \cite{Tsuji}).
		\label{l:green}
		$G_{\Omega}(z,z_0)=\sup_{v\in\Delta_0(z_0)}v(z)$, where $\Delta_0(z_0)$ is the set of negative subharmonic functions on $\Omega$ satisfying that $v-\log|w|$ has a locally finite upper bound near $z_0$.
	\end{Lemma}
	
	\begin{Lemma}\label{l:G-compact}(see \cite{GY-concavity})
		Let $\Omega$ be an open Riemann surface which
		admits a nontrivial Green function $G_{\Omega}$. Let $z_0\in \Omega$, and let $U$ be  a open neighborhood of $z_0$. Then there exists $t>0$ such that $\{G_{\Omega}(z,z_0)<-t\}$ is a relatively compact subset of $U$.
	\end{Lemma}
	
	\begin{Lemma}
		\label{l:psi<G}
		Let $\Omega$ be an open Riemann surface which
		admits a nontrivial Green function $G_{\Omega}$. Let $z_0\in\Omega$, and let $\psi<0$ be a subharmonic function on $\Omega$ satisfying $\frac{1}{2}v(dd^c\psi,z_0)\geq k+1$. Let $l(t)$ is a positive Lebesgue measurable function on $(0,+\infty)$ satisfying $l$ is decreasing on $(0,+\infty)$ and $\int_0^{+\infty}l(t)dt<+\infty$. If $\psi\not=2(k+1)G_{\Omega}(\cdot,z_0)$, then there exists a Lebesgue measurable subset $V$ of $\Omega$, such that  $l(-\psi(z))<l(-2(k+1)G_{\Omega}(z,z_0))$ for any  $z\in V$ and $\mu(I)>0$, where $\mu$ is the Lebesgue measure on $\Omega$.
	\end{Lemma}
	\begin{proof}
		It follows from Lemma \ref{l:G-compact} that there exists $t_0>0$ such that $\{z\in\Omega:2(k+1)G_{\Omega}(z,z_0)<-t_0\}\subset\subset\Omega$. As $l$ is decreasing and $\int_0^{+\infty}l(t)dt<+\infty$, then there exists $t_1>t_0$ such that $l(t)<l(t_1)$ holds for any $t>t_1.$
		
		Following from Siu's Decomposition Theorem, $\psi\not=2(k+1)G_{\Omega}(\cdot,z_0)$ and Lemma \ref{l:green}, we know $\psi-2(k+1)G_{\Omega}(z,z_0)$ is a negative subharmonic function on $\Omega$. By $\psi$ is upper semicontinuous, then we have $\sup_{z\in\Omega:\{2(k+1)G_{\Omega}(z,z_0)\leq-t_1\}}\psi(z)<-t_1$. Thus there exists  $t_2\in(t_0,t_1)$ such that $\sup_{\{z\in\Omega:2(k+1)G_{\Omega}(z,z_0)\leq-t_2\}}\psi(z)<-t_1$. Denote $t_3:=-\sup_{z\in\Omega:\{2(k+1)G_{\Omega}(z,z_0)\leq-t_2\}}\psi(z)$. Let $V=\{z\in\Omega:-t_1<2(k+1)G_{\Omega}(z,z_0)<-t_2\}$, then $\mu(V)>0$. As $l(t)$ is decreasing on $(0,+\infty)$, for any $z\in V$, we have
		$$l(-\psi(z))\leq l(t_3)<l(t_1)\leq l(-2(k+1)G_{\Omega}(z,z_0)).$$
		Thus, Lemma \ref{l:psi<G} holds.
	\end{proof}

	\begin{Lemma}
		\label{l:chi}
		Let $\Omega$ be an open Riemann surface which
		admits a nontrivial Green function $G_{\Omega}$. Then there exists a harmonic function $u_{z_0}$ on $\Omega$ such that $\chi_{z_0}=\chi_{u_{z_0}}$ for any $z_0\in\Omega$.
	\end{Lemma}
	\begin{proof}
		Let $f_{z_0}$ be a holomorphic function on $\Delta$ such that $|f_{z_0}(z)|=p^{\star}e^{G_{\Omega}(z,z_0)}$. As $\Omega$ be an open Riemann surface, it follows from Weierstrass Theorem on open Riemann surfaces (see \cite{OF81}) that there exists a holomorphic function $\tilde{f}$ on $\Omega$ satisfying $\{z\in\Omega:\tilde{f}(z)=0\}=\{z_0\}$ and $d\tilde{f}(z_0)\not=0$. Let $u_{z_0}=\frac{p_{\star}(\log|f_{z_0}|)}{\log|\tilde{f}|}$. Note that
		$$e^{G_{\Omega}(\cdot,z_0)-u}=|\tilde{f}|,$$
		then we have $\chi_{z_0}=\chi_{u_{z_0}}$.
	\end{proof}

	Let $\Omega$ be an open Riemann surface with a Green function. Let $p:\Delta\rightarrow\Omega$ be the universal covering of $\Omega$. We can choose $V_{z_0}$ small enough, such that $p$ restricted on any component of $p^{-1}(V_{z_0})$ is biholomorphic. Let $\rho_1=e^{-2h_1}$, where $h_1$ is harmonic on $\Omega$. There exists a multiplicative holomorphic function $f_{h_1}$ on $\Delta$, such that $|f_{h_1}|=p^{\star}e^{h_1}$. Let $f_{-h_1}:=f_{h_1}^{-1}$. Let $f_{-h_1,j}:=f_{-h_1}|_{U_j}$ and $p_j:=p|_{U_j},$ where $U_j$ is a component of $p^{-1}(V_{z_0})$ for any fixed $j.$
	
	\begin{Lemma}\label{l:d}
		(see \cite{guan-zhou13ap}) Let $\Omega$ be an open Riemann surface which
		admits a nontrivial Green function $G_{\Omega}$. Let $z_0\in\Omega$, and let $V_{z_0}$ be a neighborhood of $z_0$ with local coordinate $w$ such that $w(z_0)=0.$ Assume that there is a negative subharmonic function $\Psi$ on $\Omega$, such that $\Psi|_{V_{z_0}}=\log|w|^2$ and $\Psi|_{\Omega\backslash V_{z_0}}\geq\sup_{z\in V_{z_0}}\Psi.$ Let $d_1(t)$ and $d_2(t)$ be two positive continuous functions on $(0,+\infty)$, which satisfy
		$$\int_0^{+\infty}d_1(t)e^{-t}dt=\int_0^{+\infty}d_2(t)e^{-t}dt<+\infty,$$
		$$d_1(t)|_{\{t>r_1\}\cup\{t<r_3\}}=d_2(t)|_{\{t>r_1\}\cup\{t<r_3\}},$$
		$$d_1(t)|_{\{r_2<t<r_1\}}>d_2(t)|_{\{r_2<t<r_1\}},$$
		and
		$$d_1(t)|_{\{r_3<t<r_2\}}>d_2(t)|_{\{r_3<t<r_2\}},$$
		where $0<r_3<r_2<r_1<+\infty.$ Assume that $\{\psi<-r_3+1\}\subset\subset V_{z_0}$, which is a disc with the coordinate $w$. Let $F$ be a holomorphic $(1,0)$ form, which satisfies $((p_j)_{\star}(f_{-h_1,j}))F|_{z_0}=dw.$ Then we have
		$$\int_{\Omega}|F|^2\rho_1d_1(-\Psi)\leq\int_{\Omega}|F|^2\rho_1d_2(-\Psi).$$
		Moreover, the equality holds if and only if $((p_j)_{\star}(f_{-h_1,j}))F|_{V_{z_0}}=dw.$
	\end{Lemma}

	\section{Proofs of Theorem \ref{maintheorem}, Remark \ref{infty2}  and Corollary \ref{linear case of G}}\label{sec:Proof-main}
	In this section, we prove Theorem \ref{maintheorem}, Remark \ref{infty2} and Corollary \ref{linear case of G}.
	\subsection{A proof of Theorem \ref{maintheorem}}
	\
	
	Firstly, we prove that if $G(t_0)<+\infty$ for some $t_0>T$, then $G(t_1)<+\infty$ for any $t_1\in(T,t_0)$. It follows from Lemma \ref{existence of F} that there exists a holomorphic $(n,0)$ form $F_{t_0}$ on $\{\psi<-t_0\}$ satisfying $(F_{t_0}-f)\in H^{0}(Z_0,(\mathcal{O}(K_M)\otimes\mathcal{F})|_{Z_0})$ and $\int_{\{\psi<-t_0\}}|F_{t_0}|^2e^{-\varphi}c(-\psi)=G(t_0)<+\infty$. It follows from Lemma \ref{lemma 2.2} that we have a holomorphic $(n,0)$ form $\tilde{F}$ on $\{\psi<-t_1\}$, such that
	$$(\tilde{F}-F_{t_0})\in H^{0}(Z_0,(\mathcal{O}(K_M)\otimes\mathcal{I}(\varphi+\psi)|_{Z_0}))\subset H^{0}(Z_0,(\mathcal{O}(K_M)\otimes\mathcal{F})|_{Z_0})$$
	and
	\begin{equation}\
		\begin{split}
			& \int_{\{\psi<-t_1\}}|\tilde{F}-(1-b_{t_0,B}(\psi))F_{t_0}|^2e^{-\varphi}c(-\psi) \\
			\le&\int_{\{\psi<-t_1\}}|\tilde{F}-(1-b_{t_0,B}(\psi))F_{t_0}|^2
			e^{-\varphi-\psi+v_{t_0,B}(\psi)}c(-v_{t_0,B}(\psi))\\
			\le &\int_{t_1}^{t_0+B}c(t)e^{-t}dt \int_{\{\psi<-t_1\}}\frac{1}{B}\mathbb{I}_{\{-t_0-B<\psi<-t_0\}}|F_{t_0}|^2e^{-\varphi-\psi}.
			\label{proof of main theorem 1}
		\end{split}
	\end{equation}
	Note that
	\begin{equation}\nonumber
		\begin{split}
			& (\int_{\{\psi<-t_1\}}|\tilde{F}|^2e^{-\varphi}c(-\psi))^{\frac{1}{2}}
			-(\int_{\{\psi<-t_1\}}|(1-b_{t_0,B}(\psi))F_{t_0}|^2e^{-\varphi}c(-\psi))^{\frac{1}{2}}\\
			\le& \int_{\{\psi<-t_1\}}|\tilde{F}-(1-b_{t_0,B}(\psi))F_{t_0}|^2e^{-\varphi}c(-\psi).
		\end{split}
	\end{equation}
	Combining with inequality \eqref{proof of main theorem 1}, we obtain
	\begin{equation}\
		\begin{split}
			(\int_{\{\psi<-t_1\}}|\tilde{F}|^2e^{-\varphi}c(-\psi))^{\frac{1}{2}}
			\le& (\int_{t_1}^{t_0+B}c(t)e^{-t}dt \int_{\{\psi<-t_1\}}\frac{1}{B}\mathbb{I}_{\{-t_0-B<\psi<-t_0\}}|F_{t_0}|^2e^{-\varphi-\psi})^{\frac{1}{2}}\\
			+& (\int_{\{\psi<-t_1\}}|(1-b_{t_0,B}(\psi))F_{t_0}|^2e^{-\varphi}c(-\psi))^{\frac{1}{2}}.
		\end{split}
	\end{equation}
	Since $b_{t_0,B}(\psi)=1$ on $\{\psi\ge t_0\}$, $0\le b_{t_0,B}(\psi)\le 1$, $\int_{\{\psi<-t_0\}}|F_{t_0}|^2e^{-\varphi}c(-\psi)<+\infty$, and $c(t)$ has positive lower bound on any compact subset of $(T,+\infty)$, then
	$$ (\int_{\{\psi<-t_1\}}|(1-b_{t_0,B}(\psi))F_{t_0}|^2e^{-\varphi}c(-\psi))^{\frac{1}{2}}<+\infty$$
	and
	\begin{equation}\nonumber
		\begin{split}
			&\int_{t_1}^{t_0+B}c(t)e^{-t}dt \int_{\{\psi<-t_1\}}\frac{1}{B}\mathbb{I}_{\{-t_0-B<\psi<-t_0\}}|F_{t_0}|^2e^{-\varphi-\psi}\\
			\le& \frac{e^{t_0+B}\int_{t_1}^{t_0+B}c(t)e^{-t}dt}{\inf\limits_{t\in(t_0,t_0+B)}c(t)}
			\int_{\{\psi<-t_1\}}\frac{1}{B}\mathbb{I}_{\{-t_0-B<\psi<-t_0\}}|F_{t_0}|^2e^{-\varphi-\psi}
			<+\infty.
		\end{split}
	\end{equation}
	Hence we have
	$$\int_{\{\psi<-t_1\}}|\tilde{F}|^2e^{-\varphi}c(-\psi)<+\infty.$$
	Then we obtain that $G(t_1)\le \int_{\{\psi<-t_1\}}|\tilde{F}|^2e^{-\varphi}c(-\psi)<+\infty$.
	
	Now, assume that $G(t_0)<+\infty$ for some $t_0\ge T$. As $G(h^{-1}(r))$ is lower semicontinuous (Lemma \ref{semicontinuous}), then Lemma \ref{derivatives of G} and Lemma \ref{characterization of concave function} imply the concavity of $G(h^{-1}(r))$. It follows from Lemma \ref{semicontinuous} that $\lim\limits_{t\to T+0}G(t)=G(T)$ and $\lim\limits_{t\to+\infty}G(t)=0$.
	
	Theorem \ref{maintheorem} is proved.

	\subsection{A proof of Remark \ref{infty2}}\label{sec:4.2}
	\
	
	Note that if there exists a positive decreasing concave function $g(t)$ on $(a,b)\subset\mathbb{R}$ and $g(t)$ is not a constant function, then $b<+\infty$.
	
	Assume that $G(t_0)<+\infty$ for some $t_0\geq T$, as $f\not\in H^0(Z_0,(\mathcal{O}(K_M)\otimes\mathcal{F})|_{Z_0})$, Lemma \ref{G equal to 0} shows that $G(t_0)\in(0,+\infty)$. Following from Theorem \ref{maintheorem} we know $G({h}^{-1}(r))$ is concave with respect to $r\in(\int_{T_1}^{T}c(t)e^{-t}dt,\int_{T_1}^{+\infty}c(t)e^{-t}dt)$ and $G({h}^{-1}(r))$ is not a constant function, therefore we obtain $\int_{T_1}^{+\infty}c(t)e^{-t}dt<+\infty$, which contradicts to $\int_{T_1}^{+\infty}c(t)e^{-t}dt=+\infty$. Thus we have $G(t)\equiv+\infty$.
	
	When $G(t_2)\in(0,+\infty)$ for some $t_2\in[T,+\infty)$, Lemma \ref{G equal to 0} shows that $f\not\in H^0(Z_0,(\mathcal{O}(K_M)\otimes\mathcal{F})|_{Z_0})$. Combining the above discussion, we know $\int_{T_1}^{+\infty}c(t)e^{-t}dt<+\infty$. Using Theorem \ref{maintheorem}, we obtain that $G(\hat{h}^{-1}(r))$ is concave with respect to  $r\in (0,\int_{T}^{+\infty}c(t)e^{-t}dt)$, where $\hat{h}(t)=\int_{t}^{+\infty}c(l)e^{-l}dl$.
	
	Thus, Remark \ref{infty2} holds.
	
	\subsection{A proof of Corollary \ref{linear case of G}}
	\
	
	As $G(\hat{h}^{-1}(r))$ is linear with respect to $r$, then $G(t)=\frac{G(T)}{\int_{T}^{+\infty}c(s)e^{-s}ds}\int_{t}^{+\infty}c(s)e^{-s}ds$ for any $t\in[T,+\infty)$. We firstly prove the existence and uniqueness of $F$.
	
	Following the notations in Lemma \ref{derivatives of G}, as $G(t)=\frac{G(T_1)}{\int_{T}^{+\infty}c(s)e^{-s}ds}\int_t^{+\infty}c(s)e^{-s}ds\in(0,+\infty)$ for any  $t\in(T,+\infty)$,  by choosing $t_1\in(T,+\infty)$ and $t_0>t_1$, we know that the inequality  \eqref{eq:211106c} must be equality, which implies that
	\begin{equation}
		\label{eq:20210412c}
		\int_{\{\psi<-t_0\}}|F_1-F_{t_0}|^2e^{-\varphi}(e^{-\psi-t_0}c(t_0)-c(-\psi))=0,
	\end{equation}
	where $F_1$ is a holomorphic $(n,0)$ form on $\{\psi<-t_1\}$ such that $(F_1-f)\in H^{0}(Z_0,(\mathcal{O}(K_{M})\otimes\mathcal{F})|_{Z_0})$  and $F_{t_0}$ is a holomorphic $(n,0)$ form on $\{\psi<-t_0\}$ such that $(F_{t_0}-f)\in H^{0}(Z_0,(\mathcal{O}(K_{M})\otimes\mathcal{F})|_{Z_0})$.
	As $\int_{T_1}^{+\infty}c(t)e^{-t}<+\infty$ and $c(t)e^{-t}$ is decreasing, then there exists $t_2>t_0$ such that $c(t)e^{-t}<c(t_0)e^{-t_0}-\delta$ for any $t\geq t_2$, where $\delta$ is a positive constant. Then equality \eqref{eq:20210412c} implies that
	\begin{displaymath}
		\begin{split}
			&\delta\int_{\{\psi<-t_2\}}|F_1-F_{t_0}|^2e^{-\varphi}e^{-\psi}\\
			\le&\int_{\{\psi<-t_2\}}|F_1-F_{t_0}|^2e^{-\varphi}(e^{-\psi-t_0}c(t_0)-c(-\psi))\\
			\le&\int_{\{\psi<-t_0\}}|F_1-F_{t_0}|^2e^{-\varphi}(e^{-\psi-t_0}c(t_0)-c(-\psi))\\
			=&0
		\end{split}
	\end{displaymath}
	It follows from $\varphi+\psi$ is plurisubharmonic function and $F_1$ and $F_{t_0}$ are holomorphic $(n,0)$ forms that $F_1=F_{t_0}$ on $\{\psi<-t_0\}$. As $\int_{\{\psi<-t_0\}}|F_{t_0}|^2e^{-\varphi}c(-\psi)=G(t_0)$ and inequality  \eqref{eq:211106c} becomes equality, we have
	$$\int_{\{\psi<-t_1\}}|F_1|^2e^{-\varphi}c(-\psi)=G(t_1).$$
	Following from Lemma \ref{existence of F},  there exists a unique holomorphic $(n,0)$ form $F_t$ on $\{\psi<-t\}$ satisfying $(F_t-f)\in H^{0}(Z_0,(\mathcal{O}(K_{M})\otimes\mathcal{F})|_{Z_0})$ and
	$\int_{\{\psi<-t\}}|F_t|^{2}e^{-\varphi}c(-\psi)=G(t)$ for any $t>T$. By discussion in the above, we know $F_{t}=F_{t'}$ on $\{\psi<-\max{\{t,t'\}}\}$ for any $t\in(T,+\infty)$ and $t'\in(T,+\infty)$. Hence combining $\lim_{t\rightarrow T+0}G(t)=G(T)$, we obtain that there  exists a unique holomorphic $(n,0)$ form $F$ on $M$ satisfying $(F-f)\in H^{0}(Z_0,(\mathcal{O}(K_{M})\otimes\mathcal{F})|_{Z_0})$ and
	$\int_{\{\psi<-t\}}|F|^{2}e^{-\varphi}c(-\psi)=G(t)$ for any $t\geq T$.

	Secondly, we prove equality \eqref{other a also linear}. As $a(t)$ is nonnegative measurable function on $(T,+\infty)$, then there exists a sequence of functions $\{\sum\limits_{j=1}^{n_i}a_{ij}\mathbb{I}_{E_{ij}}\}_{i\in\mathbb{N}^+}$ $(n_i<+\infty$ for any $i\in\mathbb{N}^+)$ satisfying $\sum\limits_{j=1}^{n_i}a_{ij}\mathbb{I}_{E_{ij}}$ is increasing with respect to $i$ and $\lim\limits_{i\to +\infty}\sum\limits_{j=1}^{n_i}a_{ij}\mathbb{I}_{E_{ij}}=a(t)$ for any $t\in(T,+\infty)$, where $E_{ij}$ is a Lebesgue measurable subset of $(T,+\infty)$ and $a_{ij}\ge 0$ is a constant.  It follows from Levi's Theorem that it suffices to prove the case that $a(t)=\mathbb{I}_{E}(t)$, where $E\subset\subset (T,+\infty)$ is a Lebesgue measurable set.
	
	Note that $G(t)=\int_{\{\psi<-t\}}|F|^2e^{-\varphi}c(-\psi)=\frac{G(T_1)}{\int_{T_1}^{+\infty}c(s)e^{-s}ds}
	\int_{t}^{+\infty}c(s)e^{-s}ds$, then
	\begin{equation}\label{linear 3.4}
		\int_{\{-t_1\le\psi<-t_2\}}|F|^2e^{-\varphi}c(-\psi)=\frac{G(T_1)}{\int_{T_1}^{+\infty}c(s)e^{-s}ds}
		\int_{t_2}^{t_1}c(s)e^{-s}ds
	\end{equation}
	holds for any $T\le t_2<t_1<+\infty$. It follows from the dominated convergence theorem and inequality \eqref{linear 3.4} that
	\begin{equation}\label{linear 3.5}
		\int_{\{z\in M:-\psi(z)\in N\}}|F|^2e^{-\varphi}=0
	\end{equation}
	holds for any $N\subset\subset (T,+\infty)$ such that $\mu(N)=0$, where $\mu$ is the Lebesgue measure on $\mathbb{R}$.
	
	As $c(t)e^{-t}$ is decreasing on $(T,+\infty)$, there are at most countable points denoted by $\{s_j\}_{j\in \mathbb{N}^+}$ such that $c(t)$ is not continuous at $s_j$. Then there is a decreasing sequence open set $\{U_k\}$, such that $\{s_j\}_{j\in \mathbb{N}^+}$, such that
	$\{s_j\}_{j\in \mathbb{N}^+}\subset U_k\subset (T,+\infty)$ for any $j$, and $\lim\limits_{k \to +\infty}\mu(U_k)=0$. Choosing any closed interval $[t'_2,t'_1]\subset (T,+\infty)$. Then we have
	\begin{equation}\label{linear 3.6}
		\begin{split}
			&\int_{\{-t'_1\le\psi<-t'_2\}}|F|^2e^{-\varphi}\\
			=&\int_{\{z\in M:-\psi(z)\in(t'_2,t'_1]\backslash U_k\}}|F|^2e^{-\varphi}+
			\int_{\{z\in M:-\psi(z)\in[t'_2,t'_1]\cap U_k\}}|F|^2e^{-\varphi}\\
			=&\lim_{n\to+\infty}\sum_{i=0}^{n-1}\int_{\{z\in M:-\psi(z)\in I_{i,n}\backslash U_k\}}|F|^2e^{-\varphi}+
			\int_{\{z\in M:-\psi(z)\in[t'_2,t'_1]\cap U_k\}}|F|^2e^{-\varphi},
		\end{split}
	\end{equation}
	where $I_{i,n}=(t'_1-(i+1)\alpha_n,t'_1-i\alpha_n]$ and $\alpha_n=\frac{t'_1-t'_2}{n}$. Note that
	\begin{equation}\label{linear 3.7}
		\begin{split}
			&\lim_{n\to+\infty}\sum_{i=0}^{n-1}\int_{\{z\in M:-\psi(z)\in I_{i,n}\backslash U_k\}}|F|^2e^{-\varphi}\\
			\le&\limsup_{n\to+\infty}\sum_{i=0}^{n-1}\frac{1}{\inf_{I_{i,n}\backslash U_k}c(t)}\int_{\{z\in M:-\psi(z)\in I_{i,n}\backslash U_k\}}|F|^2e^{-\varphi}c(-\psi).
		\end{split}
	\end{equation}
	It follows from equality \eqref{linear 3.4} that inequality \eqref{linear 3.7} becomes
	\begin{equation}\label{linear 3.8}
		\begin{split}
			&\lim_{n\to+\infty}\sum_{i=0}^{n-1}\int_{\{z\in M:-\psi(z)\in I_{i,n}\backslash U_k\}}|F|^2e^{-\varphi}\\
			\le&\frac{G(T_1)}{\int_{T_1}^{+\infty}c(s)e^{-s}ds}
			\limsup_{n\to+\infty}\sum_{i=0}^{n-1}\frac{1}{\inf_{I_{i,n}\backslash U_k}c(t)}\int_{I_{i,n}\backslash U_k}c(s)e^{-s}ds.
		\end{split}
	\end{equation}
	It is clear that $c(t)$ is uniformly continuous and has positive lower bound and upper bound on $[t'_2,t'_1]\backslash U_k$. Then we have
	\begin{equation}\label{linear 3.9}
		\begin{split}
			&\limsup_{n\to+\infty}\sum_{i=0}^{n-1}\frac{1}{\inf_{I_{i,n}\backslash U_k}c(t)}\int_{I_{i,n}\backslash U_k}c(s)e^{-s}ds \\
			\le&\limsup_{n\to+\infty}\sum_{i=0}^{n-1}\frac{\sup_{I_{i,n}\backslash U_k}c(t)}{\inf_{I_{i,n}\backslash U_k}c(t)}\int_{I_{i,n}\backslash U_k}e^{-s}ds\\
			=&\int_{(t'_2,t'_1]\backslash U_k}e^{-s}ds.
		\end{split}
	\end{equation}
	Combining inequality \eqref{linear 3.6}, \eqref{linear 3.8} and \eqref{linear 3.9}, we have
	\begin{equation}\label{linear 3.10}
		\begin{split}
			&\int_{\{-t'_1\le\psi<-t'_2\}}|F|^2e^{-\varphi}\\
			=&\int_{\{z\in M:-\psi(z)\in(t'_2,t'_1]\backslash U_k\}}|F|^2e^{-\varphi}+
			\int_{\{z\in M:-\psi(z)\in[t'_2,t'_1]\cap U_k\}}|F|^2e^{-\varphi}\\
			\le&\frac{G(T_1)}{\int_{T_1}^{+\infty}c(s)e^{-s}ds}\int_{(t'_2,t'_1]\backslash U_k}e^{-s}ds+
			\int_{\{z\in M:-\psi(z)\in[t'_2,t'_1]\cap U_k\}}|F|^2e^{-\varphi}.
		\end{split}
	\end{equation}
	Let $k\to +\infty$, following inequality \eqref{linear 3.5} and inequality \eqref{linear 3.10}, we obtain that
	\begin{equation}\label{linear 3.11}
		\begin{split}
			\int_{\{-t'_1\le\psi<-t'_2\}}|F|^2e^{-\varphi}
			\le\frac{G(T_1)}{\int_{T_1}^{+\infty}c(s)e^{-s}ds}\int_{t'_2}^{t'_1}e^{-s}ds.
		\end{split}
	\end{equation}
	Following from a similar discussion we can obtain that
	\begin{equation}\nonumber
		\begin{split}
			\int_{\{-t'_1\le\psi<-t'_2\}}|F|^2e^{-\varphi}
			\ge\frac{G(T_1)}{\int_{T_1}^{+\infty}c(s)e^{-s}ds}\int_{t'_2}^{t'_1}e^{-s}ds.
		\end{split}
	\end{equation}
	then combining inequality \eqref{linear 3.11}, we know
	\begin{equation}\label{linear 3.12}
		\begin{split}
			\int_{\{-t'_1\le\psi<-t'_2\}}|F|^2e^{-\varphi}
			=\frac{G(T_1)}{\int_{T_1}^{+\infty}c(s)e^{-s}ds}\int_{t'_2}^{t'_1}e^{-s}ds.
		\end{split}
	\end{equation}
	Then it is clear that for any open set $U\subset (T,+\infty)$ and compact set $V\subset (T,+\infty)$
	$$
	\int_{\{z\in M;-\psi(z)\in U\}}|F|^2e^{-\varphi}
	=\frac{G(T_1)}{\int_{T_1}^{+\infty}c(s)e^{-s}ds}\int_{U}e^{-s}ds.
	$$
	and
	$$
	\int_{\{z\in M;-\psi(z)\in V\}}|F|^2e^{-\varphi}
	=\frac{G(T_1)}{\int_{T_1}^{+\infty}c(s)e^{-s}ds}\int_{V}e^{-s}ds.
	$$
	As $E\subset\subset (T,+\infty)$. then $E\cap(t_2,t_1]$ is a Lebesgue measurable subset of $(T+\frac{1}{n},n)$ for some large $n$, where $T\le t_2<t_1\le+\infty$. Then there exists a sequence of compact sets $\{V_j\}$ and a sequence of open subsets $\{V'_j\}$ satisfying $V_1\subset \ldots \subset V_j\subset V_{j+1}\subset\ldots \subset E\cap(t_2,t_1]\subset \ldots \subset V'_{j+1}\subset V'_j\subset \ldots\subset V'_1\subset\subset (T,+\infty)$ and $\lim\limits_{j\to +\infty}\mu(V'_j-V_j)=0$, where $\mu$ is the Lebesgue measure on $\mathbb{R}$. Then we have
	\begin{equation}\nonumber
		\begin{split}
			\int_{\{-t'_1\le\psi<-t'_2\}}|F|^2e^{-\varphi}\mathbb{I}_E(-\psi)
			=&\int_{z\in M:-\psi(z)\in E\cap (t_2,t_1]}|F|^2e^{-\varphi}\\
			\le&\liminf_{j\to+\infty}\int_{\{z\in M:-\psi(z)\in V'_j\}}|F|^2e^{-\varphi}\\
			\le&\liminf_{j\to+\infty}\frac{G(T_1)}{\int_{T_1}^{+\infty}c(s)e^{-s}ds}\int_{V'_j}e^{-s}ds\\
			\le&\frac{G(T_1)}{\int_{T_1}^{+\infty}c(s)e^{-s}ds}\int_{E\cap(t_2,t_1]}e^{-s}ds\\
			=&\frac{G(T_1)}{\int_{T_1}^{+\infty}c(s)e^{-s}ds}\int_{t_2}^{t_1}e^{-s}\mathbb{I}_E(s)ds,
		\end{split}
	\end{equation}
	and
	\begin{equation}\nonumber
		\begin{split}
			\int_{\{-t'_1\le\psi<-t'_2\}}|F|^2e^{-\varphi}\mathbb{I}_E(-\psi)
			\ge&\liminf_{j\to+\infty}\int_{\{z\in M:-\psi(z)\in V_j\}}|F|^2e^{-\varphi}\\
			\ge&\liminf_{j\to+\infty}\frac{G(T_1)}{\int_{T_1}^{+\infty}c(s)e^{-s}ds}\int_{V_j}e^{-s}ds\\
			=&\frac{G(T_1)}{\int_{T_1}^{+\infty}c(s)e^{-s}ds}\int_{t_2}^{t_1}e^{-s}\mathbb{I}_E(s)ds,
		\end{split}
	\end{equation}
	which implies that
	$$\int_{\{-t'_1\le\psi<-t'_2\}}|F|^2e^{-\varphi}\mathbb{I}_E(-\psi)=
	\frac{G(T_1)}{\int_{T_1}^{+\infty}c(s)e^{-s}ds}\int_{t_2}^{t_1}e^{-s}\mathbb{I}_E(s)ds.$$
	Hence we obtain that the equality \eqref{other a also linear} holds.
	
	\subsection{A proof of Remark \ref{rem:linear}}
	
	By the definition of $G(t;\tilde{c})$, we have $G(t_0;\tilde{c})\le\int_{\{\psi<-t_0\}}|F|^2e^{-\varphi}\tilde{c}(-\psi)$, then we only consider the case $G(t_0;\tilde{c})<+\infty$.
	
	By the definition of $G(t;\tilde{c})$, we can choose a holomorphic $(n,0)$ form $F_{t_0,\tilde{c}}$ on $\{\psi<-t_0\}$ satisfying $(F_{t_0,\tilde{c}}-f)\in H^0(Z_0 ,(\mathcal{O} (K_X) \otimes \mathcal{F})|_{Z_0}$ and $\int_{ \{ \psi<-t_0\}}|F_{t_0,\tilde{c}}|^2e^{-\varphi}\tilde{c}(-\psi)<+\infty.$ As $\mathcal{H}^2(\tilde{c},t_0)\subset \mathcal{H}^2(c,t_0)$, we have $\int_{ \{ \psi<-t_0\}}|F_{t_0,\tilde{c}}|^2e^{-\varphi}c(-\psi)<+\infty.$ By using
	Lemma \ref{existence of F}, we obtain that
	\begin{equation}\nonumber
		\begin{split}
			\int_{ \{ \psi<-t\}}|F_{t_0,\tilde{c}}|^2e^{-\varphi}c(-\psi)
			=&\int_{ \{ \psi<-t\}}|F|^2e^{-\varphi}c(-\psi)\\
			+&\int_{ \{ \psi<-t\}}|F_{t_0,\tilde{c}}-F|^2e^{-\varphi}c(-\psi)
		\end{split}
	\end{equation}
	for any $t\ge t_0,$ then
	\begin{equation}\label{linear 3.13}
		\begin{split}
			\int_{ \{-t_3\le \psi<-t_4\}}|F_{t_0,\tilde{c}}|^2e^{-\varphi}c(-\psi)
			=&\int_{ \{-t_3\le \psi<-t_4\}}|F|^2e^{-\varphi}c(-\psi)\\
			+&\int_{\{-t_3\le \psi<-t_4\}}|F_{t_0,\tilde{c}}-F|^2e^{-\varphi}c(-\psi)
		\end{split}
	\end{equation}
	holds for any $t_3>t_4\ge t_0$. It follows from the dominant convergence theorem, equality \eqref{linear 3.13}, \eqref{linear 3.5} and $c(t)>0$ for any $t>T$, that
	\begin{equation}\label{linear 3.14}
		\begin{split}
			\int_{ \{z\in M:-\psi(z)=t\}}|F_{t_0,\tilde{c}}|^2e^{-\varphi}
			=\int_{\{z\in M:-\psi(z)=t\}}|F_{t_0,\tilde{c}}-F|^2e^{-\varphi}
		\end{split}
	\end{equation}
	holds for any $t>t_0$.
	
	Choosing any closed interval $[t'_4,t'_3]\subset (t_0,+\infty)\subset (T,+\infty)$. Note that $c(t)$ is uniformly continuous and have positive lower bound and upper bound on $[t'_4,t'_3]\backslash U_k$, where $\{U_k\}$ is the decreasing sequence of open subsets of $(T,+\infty)$, such that $c$ is continuous on $(T,+\infty)\backslash U_k$ and $\lim\limits_{k \to +\infty}\mu(U_k)=0$. Take $N=\cap_{k=1}^{+\infty}U_k.$ Note that
	\begin{equation}\label{linear 3.15}
		\begin{split}
			&\int_{ \{-t'_3\le\psi<-t'_4\}}|F_{t_0,\tilde{c}}|^2e^{-\varphi}\\
			=&\lim_{n\to+\infty}\sum_{i=0}^{n-1}\int_{\{z\in M:-\psi(z)\in S_{i,n}\backslash U_k\}}|F_{t_0,\tilde{c}}|^2e^{-\varphi}
			+\int_{\{z\in M:-\psi(z)\in(t'_4,t'_3]\cap U_k\}}|F_{t_0,\tilde{c}}|^2e^{-\varphi}\\
			\le&\limsup_{n\to+\infty}\sum_{i=0}^{n-1}\frac{1}{\inf_{S_{i,n}}c(t)}\int_{\{z\in M:-\psi(z)\in S_{i,n}\backslash U_k\}}|F_{t_0,\tilde{c}}|^2e^{-\varphi}c(-\psi)\\
			&+\int_{\{z\in M:-\psi(z)\in(t'_4,t'_3]\cap U_k\}}|F_{t_0,\tilde{c}}|^2e^{-\varphi},
		\end{split}
	\end{equation}
	where $S_{i,n}=(t'_4-(i+1)\alpha_n,t'_3-i\alpha_n]$ and $\alpha_n=\frac{t'_3-t'_4}{n}$.
	It follows from equality \eqref{linear 3.13}, \eqref{linear 3.14}, \eqref{linear 3.5} and the dominated theorem that
	\begin{equation}\label{linear 3.16}
		\begin{split}
			&\int_{\{z\in M:-\Psi(z)\in S_{i,n}\backslash U_k\}}|F_{t_0,\tilde{c}}|^2e^{-\varphi_\alpha}c(-\Psi)\\
			=&\int_{\{z\in M:-\Psi(z)\in S_{i,n}\backslash U_k\}}|F|^2e^{-\varphi_\alpha}c(-\Psi)
			+\int_{\{z\in M:-\Psi(z)\in S_{i,n}\backslash U_k\}}|F_{t_0,\tilde{c}}-F|^2e^{-\varphi_\alpha}c(-\Psi).
		\end{split}
	\end{equation}
	As $c(t)$ is uniformly continuous and has positive lower bound and upper bound on $[t'_3,t'_4]\backslash U_k$, combing equality \eqref{linear 3.16}, we have
	\begin{equation}\label{linear 3.17}
		\begin{split}
			&\limsup_{n\to+\infty}\sum_{i=0}^{n-1}\frac{1}{\inf_{S_{i,n}\backslash U_k}c(t)}\int_{\{z\in M:-\Psi(z)\in S_{i,n}\backslash U_k\}}|F_{t_0,\tilde{c}}|^2e^{-\varphi_\alpha}c(-\Psi)\\
			=&\limsup_{n\to+\infty}\sum_{i=0}^{n-1}\frac{1}{\inf_{S_{i,n}\backslash U_k}c(t)}\big(\int_{\{z\in M:-\Psi(z)\in S_{i,n}\backslash U_k\}}|F|^2e^{-\varphi_\alpha}c(-\Psi)\\
			&+\int_{\{z\in M:-\Psi(z)\in S_{i,n}\backslash U_k\}}|F_{t_0,\tilde{c}}-F|^2e^{-\varphi_\alpha}c(-\Psi)\big)\\
			\le & \limsup_{n\to+\infty}\sum_{i=0}^{n-1}\frac{\sup_{S_{i,n}\backslash U_k}c(t)}{\inf_{S_{i,n}\backslash U_k}c(t)}(\int_{\{z\in M:-\Psi(z)\in S_{i,n}\backslash U_k\}}|F|^2e^{-\varphi_\alpha}\\
			&+\int_{\{z\in M:-\Psi(z)\in S_{i,n}\backslash U_k\}}|F_{t_0,\tilde{c}}-F|^2e^{-\varphi_\alpha})\\
			=&\int_{\{z\in M:-\Psi(z)\in (t'_4,t'_3]\backslash U_k\}}|F|^2e^{-\varphi}
			+\int_{\{z\in M:-\Psi(z)\in (t'_4,t'_3]\backslash U_k\}}|F_{t_0,\tilde{c}}-F|^2e^{-\varphi}.
		\end{split}
	\end{equation}
	If follows from inequality \eqref{linear 3.15} and \eqref{linear 3.17} that
	\begin{equation}\label{linear 3.18}
		\begin{split}
			&\int_{ \{-t'_3\le\psi<-t'_4\}}|F_{t_0,\tilde{c}}|^2e^{-\varphi}\\
			\le & \int_{\{z\in M:-\psi(z)\in (t'_4,t'_3]\backslash U_k\}}|F|^2e^{-\varphi}
			+\int_{\{z\in M:-\psi(z)\in (t'_4,t'_3]\backslash U_k\}}|F_{t_0,\tilde{c}}-F|^2e^{-\varphi}\\
			&+\int_{ \{z\in M: -\psi(z)\in(t'_4,t'_3]\cap U_k\}}|F_{t_0,\tilde{c}}|^2e^{-\varphi}.
		\end{split}
	\end{equation}
	It follows from $F_{t_0,\tilde{c}}\in\mathcal{H}^2(c,t_0)$ that $\int_{ \{-t'_3\le\psi<-t'_4\}}|F_{t_0,\tilde{c}}|^2e^{-\varphi}<+\infty$. Let $k\to+\infty,$ by equality \eqref{linear 3.5}, inequality \eqref{linear 3.18} and the dominated theorem, we have
	\begin{equation}\label{linear 3.19}
		\begin{split}
			&\int_{ \{-t'_3\le\psi<-t'_4\}}|F_{t_0,\tilde{c}}|^2e^{-\varphi}\\
			\le & \int_{\{z\in M:-\psi(z)\in (t'_4,t'_3]\}}|F|^2e^{-\varphi}
			+\int_{\{z\in M:-\psi(z)\in (t'_4,t'_3]\backslash N\}}|F_{t_0,\tilde{c}}-F|^2e^{-\varphi}\\
			&+\int_{ \{z\in M: -\psi(z)\in(t'_4,t'_3]\cap N\}}|F_{t_0,\tilde{c}}|^2e^{-\varphi}.
		\end{split}
	\end{equation}
	By similar discussions, we also have that
	\begin{equation}\nonumber
		\begin{split}
			&\int_{ \{-t'_3\le\psi<-t'_4\}}|F_{t_0,\tilde{c}}|^2e^{-\varphi}\\
			\ge & \int_{\{z\in M:-\psi(z)\in (t'_4,t'_3]\}}|F|^2e^{-\varphi}
			+\int_{\{z\in M:-\psi(z)\in (t'_4,t'_3]\backslash N\}}|F_{t_0,\tilde{c}}-F|^2e^{-\varphi}\\
			&+\int_{ \{z\in M: -\psi(z)\in(t'_4,t'_3]\cap N\}}|F_{t_0,\tilde{c}}|^2e^{-\varphi}.
		\end{split}
	\end{equation}
	Combining with inequality \eqref{linear 3.19}, we have
	\begin{equation}\label{linear 3.20}
		\begin{split}
			&\int_{ \{-t'_3\le\psi<-t'_4\}}|F_{t_0,\tilde{c}}|^2e^{-\varphi}\\
			= & \int_{\{z\in M:-\psi(z)\in (t'_4,t'_3]\}}|F|^2e^{-\varphi}
			+\int_{\{z\in M:-\psi(z)\in (t'_4,t'_3]\backslash N\}}|F_{t_0,\tilde{c}}-F|^2e^{-\varphi}\\
			&+\int_{ \{z\in M: -\psi(z)\in(t'_4,t'_3]\cap N\}}|F_{t_0,\tilde{c}}|^2e^{-\varphi}.
		\end{split}
	\end{equation}
	Using equality \eqref{linear 3.5},\eqref{linear 3.14} and Levi's Theorem, we have
	\begin{equation}\label{linear 3.21}
		\begin{split}
			&\int_{ \{z\in M:-\psi(z)\in U\}}|F_{t_0,\tilde{c}}|^2e^{-\varphi}\\
			= & \int_{\{z\in M:-\psi(z)\in U\}}|F|^2e^{-\varphi}
			+\int_{\{z\in M:-\psi(z)\in U\backslash N\}}|F_{t_0,\tilde{c}}-F|^2e^{-\varphi}\\
			&+\int_{ \{z\in M: -\psi(z)\in U\cap N\}}|F_{t_0,\tilde{c}}|^2e^{-\varphi}
		\end{split}
	\end{equation}
	holds for any open set $U\subset\subset (t_0,+\infty)$, and
	\begin{equation}\label{linear 3.22}
		\begin{split}
			&\int_{ \{z\in M:-\psi(z)\in V\}}|F_{t_0,\tilde{c}}|^2e^{-\varphi}\\
			= & \int_{\{z\in M:-\psi(z)\in V\}}|F|^2e^{-\varphi}
			+\int_{\{z\in M:-\psi(z)\in V\backslash N\}}|F_{t_0,\tilde{c}}-F|^2e^{-\varphi}\\
			&+\int_{ \{z\in M: -\psi(z)\in V\cap N\}}|F_{t_0,\tilde{c}}|^2e^{-\varphi}
		\end{split}
	\end{equation}
	holds for any compact set $V\subset (t_0,+\infty)$. For any measurable set $E\subset\subset (t_0,+\infty)$, there exists a sequence of compact set $\{V_l\}$, such that $V_l\subset V_{l+1}\subset E$ for any $l$ and $\lim\limits_{l\to +\infty}\mu(V_l)=\mu(E)$, hence by equality \eqref{linear 3.22}, we have
	\begin{equation}\label{linear 3.23}
		\begin{split}
			\int_{ \{\psi<-t_0\}}|F_{t_0,\tilde{c}}|^2e^{-\varphi}\mathbb{I}_E(-\psi)
			\ge&\lim_{l \to +\infty} \int_{ \{\psi<-t_0\}}|F_{t_0,\tilde{c}}|^2e^{-\varphi}\mathbb{I}_{V_j}(-\psi)\\
			\ge&\lim_{l \to +\infty} \int_{ \{\psi<-t_0\}}|F|^2e^{-\varphi}\mathbb{I}_{V_j}(-\psi)\\
			=& \int_{ \{\psi<-t_0\}}|F|^2e^{-\varphi}\mathbb{I}_{V_j}(-\psi).
		\end{split}
	\end{equation}
	It is clear that for any $t>t_0$, there exists a sequence of functions $\{\sum_{j=1}^{n_i}\mathbb{I}_{E_{i,j}}\}_{i=1}^{+\infty}$ defined on $(t,+\infty)$, satisfying $E_{i,j}\subset\subset (t,+\infty)$, $\sum_{j=1}^{n_{i+1}}\mathbb{I}_{E_{i+1,j}}(s)\ge \sum_{j=1}^{n_{i}}\mathbb{I}_{E_{i,j}}(s)$ and $\lim\limits_{i\to+\infty}\sum_{j=1}^{n_i}\mathbb{I}_{E_{i,j}}(s)=\tilde{c}(s)$ for any $s>t$. Combing Levi's Theorem and inequality \eqref{linear 3.23}, we have
	\begin{equation}\label{linear 3.24}
		\begin{split}
			\int_{ \{\psi<-t_0\}}|F_{t_0,\tilde{c}}|^2e^{-\varphi}\tilde{c}(-\psi)
			\ge\int_{ \{\psi<-t_0\}}|F|^2e^{-\varphi}\tilde{c}(-\psi).
		\end{split}
	\end{equation}
	By the definition of $G(t_0,\tilde{c})$, we have $G(t_0,\tilde{c})=\int_{ \{\psi<-t_0\}}|F|^2e^{-\varphi}\tilde{c}(-\psi).$ Equality \eqref{other c also linear} is proved.

	\section{Proof of Theorem \ref{thm:1d-extension}}\label{sec:proof-1d}
	
	Let $\tilde\varphi=\varphi+a\psi$, $\tilde c(t)=c(\frac{t}{1-a})e^{-\frac{at}{1-a}}$ and $\tilde\psi=(1-a)\psi$ for some $a\in(-\infty,1)$. It is clear that $e^{-\tilde\varphi}\tilde c(-\tilde\psi)=e^{-\varphi}c(-\psi)$, $(1-a)\int_{0}^{+\infty}c(l)e^{-l}dl=\int_{0}^{+\infty}\tilde c(l)e^{-l}dl$. Then we can assume that $\frac{1}{2}v(dd^c\psi,z_0)=k+1$ without loss of generality.

	\subsection{Optimal jets $L^2$ extension}\label{sec:3.1}
	\
	
	In this section, we prove the existence  of holomorphic $(1,0)$ form $F$  satisfying $(F-f,z_0)\in I^{k+1}\otimes \mathcal{O}(K_{\Omega})_{z_0}$ and  inequality \eqref{eq:210902a}.
	
	It follows from Lemma \ref{l:green} that $\psi(z)\leq2(k+1)G_{\Omega}(z,z_0)$ for any $z\in\Omega$. As $c(t)e^{-t}$ is decreasing on $(0,+\infty)$, then we have $e^{-\varphi}c(-\psi)\leq e^{-(\varphi+\psi-2(k+1)G_{\Omega}(\cdot,z_0))}c(-2(k+1)G_{\Omega}(\cdot,z_0))$. Thus we can assume that $\psi=2(k+1)G_{\Omega}(\cdot,z_0)$.
	
	The following remark  shows that it suffices to prove the existence  of holomorphic $(1,0)$ form $F$  satisfying $(F-f,z_0)\in I^{k+1}\otimes \mathcal{O}(K_{\Omega})_{z_0}$ and  inequality \eqref{eq:210902a} for the case $c(t)$ has positive lower bound and upper bound on $(t',+\infty)$ for any $t'>0$.
	
	\begin{Remark}
		Let $c_j$ be a positive measurable function on $(0,+\infty)$, such that $c_{j}(t)=c(t)$ when $t<j$, $c_j(t)=\min\{c(j),\frac{1}{j}\}$ when $t\geq j$. It is clear that $c_j$ is decreasing on $(0,+\infty)$ and $\int_{0}^{+\infty}c_j(t)e^{-t}<+\infty$. As
		$$\lim_{j\rightarrow+\infty}\int_{j}^{+\infty}c_j(t)e^{-t}=0,$$
		we have
		$$\lim_{j\rightarrow+\infty}\int_{0}^{+\infty}c_j(t)e^{-t}=\int_{0}^{+\infty}c(t)e^{-t}.$$
		
		Then there exists a holomorphic $(1,0)$ form $F_j$ on $\Omega$ such that $(F_j-f,z_0)\in I^{k+1}\otimes \mathcal{O}(K_{\Omega})_{z_0}$ and
		$$\int_{\Omega}|F_j|^2e^{-\varphi}c_j(-\psi)\leq(\int_0^{+\infty}c_j(t)e^{-t}dt)\frac{2\pi e^{-\alpha}}{(k+1)(c_{\beta}(z_0))^{2(k+1)}}.$$
		
		Note that $\psi=2(k+1)G_{\Omega}(\cdot,z_0)$ is smooth on $\Omega\backslash\{z_0\}$. For any compact subset $K$ of $\Omega\backslash\{z_0\}$, there exists $s_K>0$ such that $\int_{K}e^{-s_K\psi}dV_{\Omega}<+\infty$, where $dV_{\Omega}$ is a continuous volume form on $\Omega$.  Then  we have
		$$\int_{K}(\frac{e^{\varphi}}{c_j(-\psi)})^{s_K}dV_{\Omega}=\int_{K}(\frac{e^{\varphi+\psi}}{c_j(-\psi)})^{s_K}e^{-s_K\psi}dV_{\Omega}\leq C\int_{K}e^{-s_K\psi}dV_{\Omega}<+\infty,$$
		where $C$ is a constant independent of $j$.
		It follows from Lemma \ref{l:converge} ($g_j=e^{-\varphi}c_j(-\psi)$) that there exists a subsequence of $\{F_j\}$, denoted still by $\{F_j\}$, which is uniformly convergent to a holomorphic $(1,0)$ form $F$ on any compact subset of $\Omega$ and
		\begin{displaymath}
			\begin{split}
				\int_{\Omega}|F|^2e^{-\varphi}c(-\psi)&\leq\lim_{j\rightarrow+\infty}(\int_0^{+\infty}c_j(t)e^{-t}dt)\frac{2\pi e^{-\alpha}}{(k+1)(c_{\beta}(z_0))^{2(k+1)}}\\
				&=(\int_0^{+\infty}c(t)e^{-t}dt)\frac{2\pi e^{-\alpha}}{(k+1)(c_{\beta}(z_0))^{2(k+1)}}.	\end{split}
		\end{displaymath}
		Since $\{F_j\}$ is uniformly convergent to   $F$ on any compact subset of $\Omega$ and $(F_j-f,z_0)\in I^{k+1}\otimes \mathcal{O}(K_{\Omega})_{z_0}$ for any $j$, we have $(F-f,z_0)\in I^{k+1}\otimes \mathcal{O}(K_{\Omega})_{z_0}$.
	\end{Remark}

	As $\frac{1}{2}v(dd^c(\varphi+\psi),z_0)=k+1$, $\varphi+\psi$ is subharmonic on $\Omega$  and $\psi=2(k+1)G_{\Omega}(\cdot,z_0)$, it follows from Siu's Decomposition Theorem  that  $\varphi$ is subharmonic on $\Omega$.
	As $\Omega$ is a Stein manifold, then there exist smooth subharmonic functions $\Phi_l$ on $\Omega$, which are decreasingly convergent to $\varphi$.
	We can find a sequence of open Riemann surfaces $\{D_m\}_{m=1}^{+\infty}$ satisfying $z_0\in D_m\subset\subset D_{m+1}$ for any $m$ and $\cup_{m=1}^{+\infty}D_m=\Omega$, and there is a holomorphic $(1,0)$ form $\tilde F$ on $\Omega$ such that $(\tilde F-f,z_0)\in I^{k+1}\otimes\mathcal{O}(K_{\Omega})_{z_0}$.
	
	Note that $\int_{D_m}|\tilde F|^2<+\infty$ for any $m$ and
	$$\int_{D_m}\mathbb{I}_{\{-t_0-1<\psi<-t_0\}}|\tilde{F}|^2e^{-\Phi_l-\psi}\leq e^{t_0+1}\int_{D_m}|\tilde{F}|^2e^{-\Phi_l}<+\infty$$
	for any $m$, $l$ and $t_0>T$.
	Using Lemma \ref{lemma2.1} ($\varphi\sim\Phi_l+\psi$), for any $D_m$, $l\in\mathbb{N}^+$ and $t_0>T$, there exists a holomorphic $(1,0)$ form $F_{l,m,t_0}$ on $D_m$, such that
	\begin{equation}
		\label{eq:210903a}
		\begin{split}
			&\int_{D_m}|F_{l,m,t_0}-(1-b_{t_0,1}(\psi))\tilde{F}|^{2}e^{-\Phi_l-\psi+v_{t_0,1}(\psi)}c(-v_{t_0,1}(\psi))\\
			\leq& (\int_{0}^{t_{0}+1}c(t)e^{-t}dt) \int_{D_m}\mathbb{I}_{\{-t_0-1<\psi<-t_0\}}|\tilde{F}|^2e^{-\Phi_l-\psi}
		\end{split}
	\end{equation}
	where
	$b_{t_0,1}(t)=\int_{-\infty}^{t}\mathbb{I}_{\{-t_{0}-1< s<-t_{0}\}}ds$,
	$v_{t_0,1}(t)=\int_{0}^{t}b_{t_0,1}(s)ds$. Note that $\psi(z)=2(k+1)G_{\Omega}(z,z_0)$, and $b_{t_0,1}(t)=0$ when $-t$ is large enough, then $(F_{l,m,t_0}-\tilde{F},z_0)\in I^{k+1}\otimes\mathcal{O}(K_{\Omega})_{z_0}$, and therefore $(F_{l,m,t_0}-f,z_0)\in I^{k+1}\otimes\mathcal{O}(K_{\Omega})_{z_0}$.

	Note that $v_{t_0,1}(\psi)\geq\psi$ and $c(t)e^{-t}$ is decreasing, then the inequality \eqref{eq:210903a} becomes
	\begin{equation}
		\label{eq:210903b}
		\begin{split}
			&\int_{D_m}|F_{l,m,t_0}-(1-b_{t_0,1}(\psi))\tilde{F}|^{2}e^{-\Phi_l}c(-\psi)\\
			\leq& (\int_{0}^{t_{0}+1}c(t)e^{-t}dt) \int_{D_m}\mathbb{I}_{\{-t_0-1<\psi<-t_0\}}|\tilde{F}|^2e^{-\Phi_l-\psi}.
		\end{split}
	\end{equation}
	Using Lemma \ref{l:G-compact}, there exists a neighborhood $V'_{z_0}$  of $z_0$ with local coordinate $\tilde{w}$, such that $\tilde{w}(z_0)=0$, $G_{\Omega}(\cdot,z_0)|_{V'_{z_0}}=\log|\tilde{w}|$ and $V'_{z_0}=\{z\in\Omega:G_{\Omega}(z,z_0)<-t_0\}\subset\subset D_m$. By direct calculation,  we have $\lim_{z\rightarrow z_0}|\frac{\tilde{F}}{\tilde{w}^{k}d\tilde{w}}|=c_{\beta}^{-k-1}(z_0)$ and
	\begin{equation}
		\label{eq:210903c}
		\begin{split}
			&\limsup_{t_0\rightarrow+\infty}\int_{D_m}\mathbb{I}_{\{-t_0-1<\psi<-t_0\}}|\tilde{F}|^2e^{-\Phi_l-\psi}\\
			=&\limsup_{t_0\rightarrow+\infty}\int_{D_m}\mathbb{I}_{\{-t_0-1<2(k+1)\log|\tilde{w}|<-t_0\}}|\frac{\tilde{F}}{\tilde{w}^{k}d\tilde{w}}|^2 |\tilde{w}|^{2k}e^{-\Phi_l-2(k+1)G_{\Omega}(\cdot,z_0)}d\tilde{w}d\overline{\tilde{w}}\\
			=&\limsup_{t_0\rightarrow+\infty}\int_{D_m}\mathbb{I}_{\{-t_0-1<2(k+1)\log|\tilde{w}|<-t_0\}}|\frac{\tilde{F}}{\tilde{w}^{k}d\tilde{w}}|^2e^{-\Phi_l} |\tilde{w}|^{-2}d\tilde{w}d\overline{\tilde{w}}\\
			=&4\pi c_{\beta}^{-2k-2}(z_0)e^{-\Phi_l(z_0)}\limsup_{t_0\rightarrow+\infty}\int_{\{-t_0-1<2(k+1)\log|s|<-t_0\}}s^{-1}ds\\
			=&\frac{2\pi e^{-\Phi_l(z_0)}}{(k+1)(c_{\beta}(z_0))^{2(k+1)}}\\
			<&+\infty,	
		\end{split}
	\end{equation}
	where the forth equality holds for $\lim_{z\rightarrow z_0}|\frac{\tilde{F}}{\tilde{w}^{k}d\tilde{w}}|^2e^{-\Phi_l}=c_{\beta}^{-2k-2}(z_0)e^{-\Phi_l(z_0)}.$
	Combining inequality \eqref{eq:210903b} and \eqref{eq:210903c}, letting $t_0\rightarrow+\infty$, we have
	\begin{equation}
		\label{eq:210903d}
		\begin{split}
			&\limsup_{t_0\rightarrow+\infty}\int_{D_m}|F_{l,m,t_0}-(1-b_{t_0,1}(\psi))\tilde{F}|^{2}e^{-\Phi_l}c(-\psi)\\
			\leq& \limsup_{t_0\rightarrow+\infty}(\int_{0}^{t_{0}+1}c(t)e^{-t}dt) \int_{D_m}\mathbb{I}_{\{-t_0-1<\psi<-t_0\}}|\tilde{F}|^2e^{-\Phi_l-\psi}\\
			\leq&(\int_{0}^{+\infty}c(t)e^{-t}dt) \frac{2\pi e^{-\Phi_l(z_0)}}{(k+1)(c_{\beta}(z_0))^{2(k+1)}}.
		\end{split}
	\end{equation}
	
	Note that
	$$\limsup_{t_0\rightarrow+\infty}\int_{D_m}|(1-b_{t_0,1}(\psi))\tilde{F}|^2e^{-\Phi_l}c(-\psi)<+\infty,$$
	then we have
	$$\limsup_{t_0\rightarrow+\infty}\int_{D_m}|F_{l,m,t_0}|^2e^{-\Phi_l}c(-\psi)<+\infty.$$
	Using Lemma \ref{l:converge},
	we obtain that
	there exists a subsequence of $\{F_{l,m,t_0}\}_{t_0\rightarrow+\infty}$ (also denoted by $\{F_{l,m,t_0}\}_{t_0\rightarrow+\infty}$)
	compactly convergent to a holomorphic (1,0) form on $D_m$ denoted by $F_{l,m}$. Then it follows from inequality \eqref{eq:210903d} and Fatou's Lemma that
	\begin{displaymath}
		\begin{split}
			\int_{D_m}|F_{l,m}|^{2}e^{-\Phi_l}c(-\psi)
			=&\int_{D_m}\liminf_{t_0\rightarrow+\infty}|F_{l,m,t_0}-(1-b_{t_0,1}(\psi))\tilde{F}|^{2}e^{-\Phi_l}c(-\psi)\\
			\leq&\liminf_{t_0\rightarrow+\infty}\int_{D_m}|F_{l,m,t_0}-(1-b_{t_0,1}(\psi))\tilde{F}|^{2}e^{-\Phi_l}c(-\psi)\\
			\leq&(\int_{0}^{+\infty}c(t)e^{-t}dt) \frac{2\pi e^{-\Phi_l(z_0)}}{(k+1)(c_{\beta}(z_0))^{2(k+1)}}.
		\end{split}	
	\end{displaymath}
	Note that $\lim_{l\rightarrow+\infty}\Phi_l(z_0)=\varphi(z_0)=\alpha$, then we have
	\begin{equation}
		\label{eq:210903e}\limsup_{l\rightarrow+\infty}\int_{D_m}|F_{l,m}|^{2}e^{-\Phi_l}c(-\psi)\leq(\int_{0}^{+\infty}c(t)e^{-t}dt)\frac{2\pi e^{-\alpha}}{(k+1)(c_{\beta}(z_0))^{2(k+1)}}<+\infty.
	\end{equation}
	Using Lemma \ref{l:converge} ($g_l=e^{-\Phi_l}c(-\psi)$),
	we obtain that
	there exists a subsequence of $\{F_{l,m}\}_{l\rightarrow+\infty}$ (also denoted by $\{F_{l,m}\}_{l\rightarrow+\infty}$)
	compactly convergent to a holomorphic (1,0) form on $D_m$ denoted by $F_{m}$ and
	\begin{equation}
		\label{eq:210903f}
		\int_{D_m}|F_{m}|^{2}e^{-\varphi}c(-\psi)
		\leq(\int_{0}^{+\infty}c(t)e^{-t}dt) \frac{2\pi e^{-\alpha}}{(k+1)(c_{\beta}(z_0))^{2(k+1)}}.
	\end{equation}
	Inequality \eqref{eq:210903f} implies that
	$$\int_{D_m}|F_{m'}|^{2}e^{-\varphi}c(-\psi)\leq (\int_{0}^{+\infty}c(t)e^{-t}dt) \frac{2\pi e^{-\alpha}}{(k+1)(c_{\beta}(z_0))^{2(k+1)}}$$
	holds for any $m'\geq m$.
	Using Lemma \ref{l:converge}, diagonal method and Levi's Theorem, we obtain a subsequence of $\{F_m\}$, denoted also by $\{F_m\}$, which is uniformly convergent to a holomorphic $(1,0)$ form $F$ on $\Omega$ satisfying that $(F-f,z_0)\in I^{k+1}\otimes\mathcal{O}(K_{\Omega})_{z_0}$ and
	$$\int_{\Omega}|F|^{2}e^{-\varphi}c(-\psi)\leq (\int_{0}^{+\infty}c(t)e^{-t}dt)\frac{2\pi e^{-\alpha}}{(k+1)(c_{\beta}(z_0))^{2(k+1)}}.$$
	Thus the existence  of holomorphic $(1,0)$ form $F$  satisfying $(F-f,z_0)\in I^{k+1}\otimes \mathcal{O}(K_{\Omega})_{z_0}$ and  inequality \eqref{eq:210902a} holds.
	
	\subsection{Characterization for the equality}\label{sec:3.2}
	\
	
	In this section, we prove the  characterization for the equality \\ $(\int_0^{+\infty}c(t)e^{-t}dt)\frac{2\pi e^{-\alpha}}{(k+1)(c_{\beta}(z_0))^{2(k+1)}}=\inf\{\int_{\Omega}|\tilde{F}|^2e^{-\varphi}c(-\psi):\tilde{F}$ is a holomorphic $(1,0)$ for such that $(\tilde{F}-f,z_0)\in I^{k+1}\otimes \mathcal{O}(K_{\Omega})_{z_0}$$\}$.

	By the  discussion in Section \ref{sec:3.1} ($\psi\sim\psi+t$, $c(\cdot)\sim c(\cdot+t)$ and $\Omega\sim\{\psi<-t\}$), for any $t>0$, there exists a holomorphic $(1,0)$ form $F_t$ on $\{\psi<-t\}$ such that $(F_t-f,z_0)\in I^{k+1}\otimes \mathcal{O}(K_{\Omega})_{z_0}$ and
	$$\int_{\{\psi<-t\}}|F_t|^2e^{-\varphi}c(-\psi)\leq(\int_t^{+\infty}c(l)e^{-l}dl)\frac{2\pi e^{-\alpha}}{(k+1)(c_{\beta}(z_0))^{2(k+1)}}.$$
	
	Firstly, we prove the necessity. By the discussion in Section \ref{sec:3.1}, we know there exists a holomorphic $(1,0)$ form $F\not\equiv0$ on $\Omega$ such that $(F-f,z_0)\in I^{k+1}\otimes \mathcal{O}(K_{\Omega})_{z_0}$ and
	$$\int_{\Omega}|F|^2e^{-(\varphi+\psi-2(k+1)G_{\Omega}(\cdot,z_0))}c(-2(k+1)G_{\Omega}(\cdot,z_0))\leq(\int_0^{+\infty}c(t)e^{-t}dt)\frac{2\pi e^{-\alpha}}{(k+1)(c_{\beta}(z_0))^{2(k+1)}}.$$
	As $c(t)e^{-t}$ is decreasing, $\psi\leq2(k+1)G_{\Omega}(\cdot,z_0)$ and  $(\int_0^{+\infty}c(t)e^{-t}dt)\frac{2\pi e^{-\alpha}}{(k+1)(c_{\beta}(z_0))^{2(k+1)}}=\inf\{\int_{\Omega}|\tilde{F}|^2e^{-\varphi}c(-\psi):\tilde{F}$ is a holomorphic $(1,0)$ for such that $(\tilde{F}-f,z_0)\in I^{k+1}\otimes \mathcal{O}(K_{\Omega})_{z_0}$$\}$, then we have
	$$\int_{\Omega}|F|^2e^{-\varphi}c(-\psi)=\int_{\Omega}|F|^2e^{-(\varphi+\psi-2(k+1)G_{\Omega}(\cdot,z_0))}c(-2(k+1)G_{\Omega}(\cdot,z_0)).$$
	As $\frac{1}{2}v(dd^c\psi,z_0)=k+1$, $c(t)e^{-t}$ is decreasing and $\int_0^{+\infty}c(t)e^{-t}dt<+\infty$, it follows from Lemma \ref{l:psi<G} that  $\psi=2(k+1)G_{\Omega}(\cdot,z_0)$. Thus $\varphi$ is subharmonic on $\Omega.$
	
	Note that $e^{-\varphi}c(-2(k+1)G_{\Omega}(\cdot,z_0))$ has locally positive lower bound on $\Omega\backslash{z_0}$. Taking $\mathcal{F}|_{Z_0}=\mathcal{I}(2(k+1)G_{\Omega}(\cdot,z_0))_{z_0}$, by the definition of $G(t)$, then we obtain that inequality
	\begin{equation}
		\label{eq:210809m}
		\frac{G(t)}{\int_t^{+\infty}c(l)e^{-l}dl}\leq	\frac{2\pi e^{-\alpha}}{(k+1)(c_{\beta}(z_0))^{2(k+1)}}=\frac{G(0)}{\int_0^{+\infty}c(t)e^{-t}dt}\end{equation}
	holds for any $t\geq0$.
	Theorem \ref{maintheorem} tells us  $G(\hat{h}^{-1}(r))$ is linear with respect to $r$. Then using Theorem \ref{thm:1-d-linear} shows that  $\varphi+\psi=2\log|\tilde{g}|+2G_{\Omega}(\cdot,z_0)+2\tilde{u}$, where $\tilde g$ is a holomorphic function on $\Omega$ with $ord_{z_0}(g)=k$ and $\tilde u$ is a harmonic function on $\Omega$ with $\chi_{-\tilde{u}}=\chi_{z_0}$. Then there exists a holomorphic function $f_1$ on $\Omega$ such that $|f_1|=e^{\tilde{u}+G_{\Omega}(\cdot,z_0)}$. Note that $\frac{\tilde{g}}{f_1^k}$ is a holomorphic function on $\Omega$ and denote it by $g$. It's clear that $g(z_0)\not=0$. Then
	\begin{displaymath}
		\begin{split}
			\varphi+\psi=&2\log|\tilde{g}|+2G_{\Omega}(\cdot,z_0)+2\tilde{u}\\
			=&2\log|g|+2k\log|f_1|+2G_{\Omega}(\cdot,z_0)+2\tilde{u}\\
			=&2\log|g|+2(k+1)G_{\Omega}(\cdot,z_0)+2(k+1)\tilde{u}
		\end{split}
	\end{displaymath}
	Denote that $u=(k+1)\tilde{u}$. Then the three statements in Theorem \ref{thm:1d-extension} hold.
	
	Then, we prove the sufficiency.
	If the three statements in Theorem \ref{thm:1d-extension} hold, following from Lemma \ref{l:chi}, there exists a holomorphic function $f_2$ on $\Omega$ such that $|f_2|=e^{u+kG_{\Omega}(\cdot,z_0)+u_{z_0}}$, where $u_{z_0}$ is the harmonic function on $\Omega$ with $\chi_{u_{z_0}}=\chi_{z_0}$. Then we have
	\begin{displaymath}
		\begin{split}
			\varphi+\psi=&2\log|g|+2(k+1)G_{\Omega}(\cdot,z_0)+2u\\
			=&2\log|g|+2G_{\Omega}(\cdot,z_0)+(2kG_{\Omega}(\cdot,z_0)+2u_{z_0}+2u)-2u_{z_0}\\
			=&2\log|gf_2|+2G_{\Omega}(\cdot,z_0)-2u_{z_0}.
		\end{split}
	\end{displaymath}
	Note that $ord_{z_0}(gf_2)=k$, taking $\mathcal{F}|_{Z_0}=\mathcal{I}(2(k+1)G_{\Omega}(z,z_0))_{z_0}$, then Theorem \ref{thm:1-d-linear} shows that $G(\hat{h}^{-1}(r))$ is linear with respect to $r$. It follows from Lemma \ref{linear case of G} that there exists a holomorphic  $(1,0)$ form $\tilde{F}$ such that $(\tilde{F}-f,z_0)\in I^{k+1}\otimes\mathcal{O}(K_{\Omega})_{z_0}$ and $\int_{\{\psi<-t\}}|\tilde{F}|^2e^{-\varphi}c(\psi)=G(t)$. Using Lemma \ref{l:G-compact}, there exists a neighborhood $V'_{z_0}$  of $z_0$ with local coordinate $\tilde{w}$, such that $\tilde{w}(z_0)=0$, $G_{\Omega}(\cdot,z_0)|_{V'_{z_0}}=\log|\tilde{w}|$ and $V'_{z_0}=\{z\in\Omega:G_{\Omega}(z,z_0)<-t_0\}\subset\subset\Omega$. By direct calculation, we have $\lim_{z\rightarrow z_0}|\frac{\tilde{F}}{\tilde{w}^{k}d\tilde{w}}|=c_{\beta}^{-k-1}(z_0)$ and
	\begin{displaymath}
		\begin{split}
			&\lim_{t\rightarrow+\infty}\frac{\int_{\{2(k+1)G_{\Omega}(\cdot,z_0)<-t\}}|\tilde{F}|^2e^{-\varphi}c(-\psi)}{\int_{t}^{+\infty}c(l)e^{-l}dl}\\
			=&\lim_{t\rightarrow+\infty}\frac{\int_{\{2(k+1)\log|\tilde{w}|<-t\}}|\tilde{F}|^2e^{-2\log|g|-2u}c(-2(k+1)\log|\tilde{w}|)}{\int_{t}^{+\infty}c(l)e^{-l}dl}\\
			=&\lim_{t\rightarrow+\infty}\frac{\int_{\{2(k+1)\log|\tilde{w}|<-t\}}|\frac{\tilde{F}}{\tilde{w}^{k}d\tilde{w}}|^2e^{-2\log|g|-2u}|\tilde{w}|^{2k}c(-2(k+1)\log|\tilde{w}|)d\tilde{w}d\overline{\tilde{w}}}{\int_{t}^{+\infty}c(l)e^{-l}dl}\\
			=&\frac{e^{-2\log|g(z_0)|-2u(z_0)}}{c_{\beta}^{2(k+1)}(z_0)}\lim_{t\rightarrow+\infty}\frac{\int_{\{2(k+1)\log|\tilde{w}|<-t\}}|\tilde{w}|^{2k}c(-2(k+1)\log|\tilde{w}|)d\tilde{w}d\overline{\tilde{w}}}{\int_{t}^{+\infty}c(l)e^{-l}dl}\\
			=&\frac{2\pi e^{-\alpha}}{(k+1)(c_{\beta}(z_0))^{2(k+1)}},
		\end{split}
	\end{displaymath}
	where the third equality holds for $\lim_{z\rightarrow z_0}|\frac{\tilde{F}}{\tilde{w}^{k}d\tilde{w}}|^2e^{-2\log|g|-2u}=\frac{e^{-2\log|g(z_0)|-2u(z_0)}}{c_{\beta}^{2(k+1)}(z_0)}$.
	Thus the equality $(\int_0^{+\infty}c(t)e^{-t}dt)\frac{2\pi e^{-\alpha}}{(k+1)(c_{\beta}(z_0))^{2(k+1)}}=G(0)=\inf\{\int_{\Omega}|\tilde{F}|^2e^{-\varphi}c(-\psi):\tilde{F}$ is a holomorphic $(1,0)$ for such that $(\tilde{F}-f,z_0)\in I^{k+1}\otimes \mathcal{O}(K_{\Omega})_{z_0}$$\}$ holds.
	
	\subsection{Another proof of the sufficiency---proof of Remark \ref{r:min}}\label{sec:3.3}
	\
	
	In this section, we  prove the sufficiency in another way, which gives the proof of Remark \ref{r:min}.
	
	As $g(z_0)\not=0$, without loss of generality, we can assume that $g=1$ (when $g\not=1$, we can replace all jets extension $\tilde{F}$ with $\frac{1}{g(z_0)}\tilde{F}g$), thus $\varphi=2u$ and $\psi=2(k+1)G_{\Omega}(\cdot,z_0)$ such that $\chi_{-u}=(\chi_{z_0})^{k+1}$. Let $u_{z_0}$ be the harmonic function in Lemma \ref{l:chi} such that $\chi_{u_{z_0}}=\chi_{z_0}$. It follows from the solution of equality part of Suita conjecture and Lemma \ref{existence of F}, there exists a unique holomorphic $(1,0)$ form $\tilde{F}_0$ on $\Omega$, satisfying that $\tilde{F}_0(z_0)=dw$ and $\int_{\Omega}|\tilde{F}_0|^2e^{2u_{z_0}}=\inf\{\int_{\Omega}|\tilde{F}|^2e^{2u_{z_0}}:\tilde{F}(z_0)=dw\,\&\,\tilde{F}\in H^0(\Omega,\mathcal{O}(K_{\Omega}))\}=2\pi\frac{e^{2u_{z_0}(z_0)}}{c_{\beta}^2(z_0)}$.

	By the  discussion in Section \ref{sec:3.1} ($\psi\sim\psi+t$, $c(\cdot)\sim c(\cdot+t)$, $k=0$ and $\Omega\sim\{\psi<-t\}$), for any $t>0$, there exists a holomorphic $(1,0)$ form $F_t$ on $\{\psi<-t\}$ such that $F_t(z_0)=dw$ and
	$$\int_{\{2G_{\Omega}(\cdot,z_0)<-t\}}|F_t|^2e^{2u_{z_0}}\leq e^{-t}2\pi\frac{e^{2u_{z_0}(z_0)}}{c_{\beta}^2(z_0)}.$$
	Then we have
	$$\frac{\tilde{G}(t)}{e^{-t}}\leq	2\pi\frac{e^{2u_{z_0}(z_0)}}{c_{\beta}^2(z_0)}=\tilde{G}(0)$$
	holds for any $t\geq0$, where $\tilde{G}(t)=\inf\{\int_{\{2G_{\Omega}(\cdot,z_0)<-t\}}|\tilde{F}|^2e^{2u_{z_0}}:\tilde{F}(z_0)=dw\,\&\,\tilde{F}\in H^0(\Omega,\mathcal{O}(K_{\Omega}))\}$.
	Theorem \ref{maintheorem} tells us   $\tilde{G}(-\log r)$ is linear with respect to $r$.
	
	Let $p:\Delta\rightarrow\Omega$ be the universal covering of $\Omega$. Let $V'_{z_0}$ be a neighborhood of $z_0$ with local coordinate $\tilde{w}$, such that $\tilde{w}(z_0)=0$, $G_{\Omega}(\cdot,z_0)|_{V'_{z_0}}=\log|\tilde{w}|$ and $V'_{z_0}=\{z\in\Omega:G_{\Omega}(z,z_0)<-t_0\}$. We can assume that $p$ restricted on any component of $p^{-1}(V'_{z_0})$ is biholomorphic. Let $p_j:=p|_{U_j},$ where $U_j$ is a component of $p^{-1}(V_{z_0})$ for any fixed $j.$ Following from Lemma \ref{l:G-compact}, we can choose $0<r_3<r_2<r_1<+\infty$ such that $\{z\in\Omega:2G_{\Omega}(z,z_0)<-r_3+1\}\subset\subset V'_{z_0}$. Choosing $d_1(t)$ and $d_2(t)$ be  two positive continuous function on $(0,+\infty)$ as in Lemma \ref{l:d}, it follows from Lemma \ref{linear case of G} that
	\begin{displaymath}
		\begin{split}
			\int_{\Omega}|\tilde{F}_0|^2e^{2u_{z_0}}d_1(-2G_{\Omega}(\cdot,z_0))=&(\int_{0}^{+\infty}d_1(t)e^{-t}dt)2\pi\frac{e^{2u_{z_0}(z_0)}}{c_{\beta}^2(z_0)}\\
			=&(\int_{0}^{+\infty}d_2(t)e^{-t}dt)2\pi\frac{e^{2u_{z_0}(z_0)}}{c_{\beta}^2(z_0)}\\
			=&\int_{\Omega}|\tilde{F}_0|^2e^{2u_{z_0}}d_2(-2G_{\Omega}(\cdot,z_0)).
		\end{split}
	\end{displaymath}
	Using Lemma \ref{l:d}, we have $((p_j)_{\star}(f_{u_{z_0,j}}))\tilde{F}_0|_{V_{z_0}}=c_0d\tilde{w}$, where $((p_j)_{\star}(f_{u_{z_0,j}}))\tilde{F}_0|_{z_0}=c_0d\tilde{w}$.
	As $G_{\Omega}(\cdot,z_0)|_{V'_{z_0}}=\log|\tilde{w}|$, there exists a holomorphic function $f_{z_0}$ on $\Delta$ such that $f_{z_0}|_{U_j}=p^{\star}\tilde{w}$. Then we have
	$f_{u_{z_0}}p^{\star}\tilde{F}_0|_{U_j}=c_0df_{z_0}$, which implies that
	$$\tilde{F}_0=c_0p_{\star}(f_{-u_{z_0}}df_{z_0}),$$
	where $f_{-u_{z_0}}=\frac{1}{f_{u_{z_0}}}$ and $|f_{-u_{z_0}}|=p^{\star}e^{-u_{z_0}}$. As $\tilde{F}_0(z_0)=dw$, then we have $|c_0|=e^{u_{z_0}(z_0)}c_{\beta}^{-1}(z_0)$.

	Denote that $\tilde c(t):=c((k+1)t)e^{-kt}$.
	For any holomorphic $(1,0)$ form  $\tilde{F}$ on $\Omega$ satisfying $\int_{\Omega}|\tilde{F}|^2e^{2u_{z_0}}\tilde{c}(-2G_{\Omega}(\cdot,z_0))<+\infty$, there exists $t_0>0$ such that $\{2G_{\Omega}(\cdot,z_0)<-t_0\}\subset\subset\Omega$ and $\tilde{c}(t)\geq\tilde{c}(t_0)e^{t-t_0}\geq\tilde{c}(t_0)e^{-t_0}>0$ for any $t\leq t_0$, thus
	\begin{displaymath}
		\begin{split}
			\int_{\Omega}|\tilde{F}|^2e^{2u_{z_0}}=&\int_{\{2G_{\Omega}(\cdot,z_0)<-t_0\}}|\tilde{F}|^2e^{2u_{z_0}}+\int_{\{2G_{\Omega}(\cdot,z_0)\geq -t_0\}}|\tilde{F}|^2e^{2u_{z_0}}\\
			\leq&\int_{\{2G_{\Omega}(\cdot,z_0)<-t_0\}}|\tilde{F}|^2e^{2u_{z_0}}+\frac{e^{t_0}}{\tilde{c}(t_0)}\int_{\{2G_{\Omega}(\cdot,z_0)\geq -t_0\}}|\tilde{F}|^2e^{2u_{z_0}}\tilde{c}(-2G_{\Omega}(\cdot,z_0))\\
			<&+\infty.
		\end{split}
	\end{displaymath} Using Lemma \ref{linear case of G} and Remark \ref{rem:linear}, we have that $\int_{\Omega}|\tilde{F}_0|^2e^{2u_{z_0}}\tilde{c}(-2G_{\Omega}(\cdot,z_0))=\inf\{\int_{\Omega}|\tilde{F}|^2e^{2u_{z_0}}\tilde{c}(-2G_{\Omega}(\cdot,z_0)):\tilde{F}(z_0)=dw\,\&\,\tilde{F}\in H^0(\Omega,\mathcal{O}(K_{\Omega}))\}=2\pi\frac{e^{2u_{z_0}(z_0)}}{(k+1)c_{\beta}^2(z_0)}$.
	Then for any holomorphic $(1,0)$ form $L$ on $\Omega$ satisfying $(L,z_0)\in I\otimes\mathcal{O}(K_{\Omega})_{z_0}$, we have
	\begin{equation}
		\label{eq:210906a}
		\int_{\Omega}\tilde{F}_0\bar{L}e^{2u_{z_0}}\tilde{c}(-2G_{\Omega}(\cdot,z_0))=0.
	\end{equation}
	Note that
	\begin{displaymath}
		\begin{split}
			e^{2u_{z_0}}\tilde{c}(-2G_{\Omega}(\cdot,z_0))=&e^{2u_{z_0}+2k2G_{\Omega}(\cdot,z_0)}c(-2(k+1)G_{\Omega}(\cdot,z_0))\\
			=&|f_{u}f_{u_{z_0}}f^k_{z_0}|^{2}e^{-2u}c(-2(k+1)G_{\Omega}(\cdot,z_0)).
		\end{split}
	\end{displaymath}
	As $\chi_{u}\chi_{u_{z_0}}\chi^k_{z_0}=\chi_{u}(\chi_{z_0})^{k+1}=1$, then $p_{\star}(f_{u}f_{u_{z_0}}f^k_{z_0})$ is a holomorphic function on $\Omega$. For any holomorphic $(1,0)$ form $\tilde{L}$ on $\Omega$ satisfying $(\tilde{L},z_0)\in I^{k+1}\otimes\mathcal{O}(K_{\Omega})_{z_0}$, there exists $L$ such that $\tilde{L}=p_{\star}(f_{u}f_{u_{z_0}}f^k_{z_0})L$, thus equality \eqref{eq:210906a} becomes
	\begin{equation}
		\label{eq:210906b}\int_{\Omega}p_{\star}(f_{u}f_{u_{z_0}}f^k_{z_0})\tilde{F}_0\overline{\tilde{L}}e^{-2u}c(-2(k+1)G_{\Omega}(\cdot,z_0))=0.
	\end{equation}
	Note that $p_{\star}(f_{u}f_{u_{z_0}}f^k_{z_0})\tilde{F}_0=c_0p_{\star}(f_{u}f^k_{z_0}df_{z_0})$, then equality \eqref{eq:210906b}
	implies that $\int_{\Omega}|c_{k}p_{\star}(f_{u}f^k_{z_0}df_{z_0})|^2e^{-2u}c(-2(k+1)G_{\Omega}(\cdot,z_0))=\inf\{\int_{\Omega}|\tilde{F}|^2e^{-2u}c(-2(k+1)G_{\Omega}(\cdot,z_0)):\tilde{F}(z_0)=dw\,\&\,(\tilde{F}-f,z_0)\in I^{k+1}\otimes\mathcal{O}(K_{\Omega})\}$, where $c_k$ is the constant such that $(c_{k}p_{\star}(f_{u}f^k_{z_0}df_{z_0})-f,z_0)\in I^{k+1}\otimes\mathcal{O}(K_{\Omega}).$ As $f|_{V_{z_0}}=w^{k}dw$, then we have $|c_k|=e^{-u(z_0)}c_{\beta}^{-k-1}(z_0)$.
	Following from Lemma \ref{linear case of G} and the above discussion, we have
	\begin{equation}
		\label{eq:210906c}
		\begin{split}
			&\int_{\Omega}|c_{k}p_{\star}(f_{u}f^k_{z_0}df_{z_0})|^2e^{-2u}c(-2(k+1)G_{\Omega}(\cdot,z_0))\\
			=&\int_{\Omega}|\frac{c_{k}}{c_0}p_{\star}(f_{u}f_{u_{z_0}}f^k_{z_0})\tilde{F}_0|^2e^{-2u}c(-2(k+1)G_{\Omega}(\cdot,z_0))		\\
			=&\frac{e^{-2u(z_0)}c_{\beta}^{-2k-2}(z_0)}{e^{2u_{z_0}(z_0)}c_{\beta}^{-2}(z_0)}\int_{\Omega}|\tilde{F}_0|^2e^{2u_{z_0}}\tilde{c}(-2G_{\Omega}(\cdot,z_0))\\
			=&\frac{e^{-2u(z_0)}c_{\beta}^{-2k}(z_0)}{e^{2u_{z_0}(z_0)}}2\pi\frac{e^{2u_{z_0}(z_0)}}{(k+1)c_{\beta}^2(z_0)}\\
			=&2\pi\frac{e^{-2u(z_0)}}{(k+1)(c_{\beta}(z_0))^{2(k+1)}}.
	\end{split}\end{equation}
	Thus, the sufficiency holds. Note that when $g\not=1$, the minimal $L^2$  extension for jets is $c_{k}gp_{\star}(f_{u}f^k_{z_0}df_{z_0})$, where $c_k$ is a constant such that $(c_{k}gp_{\star}(f_{u}f^k_{z_0}df_{z_0})-f,z_0)\in I^{k+1}\otimes\mathcal{O}(K_{\Omega})$ and $|c_k|=\frac{1}{g(z_0)e^{-u(z_0)}c_{\beta}^{k+1}(z_0)}=e^{-\frac{\alpha}{2}}c_{\beta}^{-k-1}(z_0)$. Thus, Remark \ref{r:min} holds.

	\section{appendix}
	
	\subsection {Some results used in the proof of Lemma \ref{lemma2.1}}
	\begin{Lemma}(see \cite{Demailly00})
		Let Q be a Hermitian vector bundle on a K\"ahler manifold M of dimension $n$ with a
		K\"ahler metric $\omega$. Assume that $\eta , g >0$ are smooth functions on M. Then
		for every form $v\in D(M,\wedge^{n,q}T^*M \otimes Q)$ with compact support we have
		\begin{equation}
			\begin{split}
				&\int_M (\eta+g^{-1})|D^{''*}v|^2_QdV_M+\int_M \eta|D^{''}v|^2_QdV_M \\
				\ge  &\int_M \langle[\eta \sqrt{-1}\Theta_Q-\sqrt{-1}\partial \bar{\partial}
				\eta-\sqrt{-1}g
				\partial\eta \wedge\bar{\partial}\eta, \Lambda_{\omega}]v,v\rangle_QdV_M.
			\end{split}
		\end{equation}
		\label{lemma21}
	\end{Lemma}

	\begin{Lemma}(Lemma 4.2 in \cite{guan-zhou13ap})
		Let M and Q be as in the above lemma and $\theta$ be a continuous (1,0) form on M.
		Then we have
		\begin{equation}
			[\sqrt{-1}\theta \wedge
			\bar{\theta},\Lambda_\omega]\alpha=\bar{\theta}\wedge(\alpha\llcorner(\bar{\theta})^\sharp),
		\end{equation}
		for any (n,1) form $\alpha$ with value in Q. Moreover, for any positive (1,1) form
		$\beta$, we have $[\beta,\Lambda_\omega]$ is semipositive.
		\label{lemma22}
	\end{Lemma}
	
	\begin{Lemma}
		(Remark 3.2 in \cite{Demailly00})
		Let ($M,\omega$) be a complete K\"ahler manifold equipped with a (non-necessarily
		complete) K\"ahler metric $\omega$, and let Q be a Hermitian vector bundle over M.
		Assume that $\eta$ and g are smooth bounded positive functions on M and let
		$B:=[\eta \sqrt{-1}\Theta_Q-\sqrt{-1}\partial \bar{\partial} \eta-\sqrt{-1}g
		\partial\eta \wedge\bar{\partial}\eta, \Lambda_{\omega}]$.Assume that $\delta \ge
		0$ is a number such that $B+\delta I$ is semi-positive definite everywhere on
		$\wedge^{n,q}T^*M \otimes Q$ for some $q \ge 1$. Then given a form $v \in
		L^2(M,\wedge^{n,q}T^*M \otimes Q)$ such that $D^{''}v=0$ and $\int_M \langle
		(B+\delta I)^{-1}v,v\rangle_Q dV_M < +\infty$, there exists an approximate solution
		$u \in L^2(M,\wedge^{n,q-1}T^*M \otimes Q)$ and a correcting term $h\in
		L^2(M,\wedge^{n,q}T^*M \otimes Q)$ such that $D^{''}u+\sqrt{\delta}h=v$ and
		\begin{equation}
			\int_M(\eta+g^{-1})^{-1}|u|^2_QdV_M+\int_M|h|^2_QdV_M \le \int_M \langle (B+\delta
			I)^{-1}v,v\rangle_Q dV_M.
		\end{equation}
		\label{lemma23}
	\end{Lemma}
	
	\begin{Lemma}
		(Theorem 6.1 in \cite{DemaillyReg}, see also Theorem 2.2 in \cite{ZZ2019})
		Let ($M,\omega$) be a complex manifold equipped with a Hermitian metric
		$\omega$, and $\Omega \subset \subset M $ be an open set. Assume that
		$T=\widetilde{T}+\frac{\sqrt{-1}}{\pi}\partial\bar{\partial}\varphi$ is a closed
		(1,1)-current on M, where $\widetilde{T}$ is a smooth real (1,1)-form and
		$\varphi$ is a quasi-plurisubharmonic function. Let $\gamma$ be a continuous real
		(1,1)-form such that $T \ge \gamma$. Suppose that the Chern curvature tensor of
		$TM$ satisfies
		\begin{equation}
			\begin{split}
				(\sqrt{-1}&\Theta_{TM}+\varpi \otimes Id_{TM})(\kappa_1 \otimes \kappa_2,\kappa_1
				\otimes \kappa_2)\ge 0 \\
				&\forall \kappa_1,\kappa_2 \in TM \quad with \quad \langle \kappa_1,\kappa_2
				\rangle=0
			\end{split}
		\end{equation}
		for some continuous nonnegative (1,1)-form $\varpi$ on M. Then there is a family
		of closed (1,1)-current
		$T_{\zeta,\rho}=\widetilde{T}+\frac{\sqrt{-1}}{\pi}\partial\bar{\partial}
		\varphi_{\zeta,\rho}$ on M ($\zeta \in (0,+\infty)$ and $\rho \in (0,\rho_1)$ for
		some positive number $\rho_1$) independent of $\gamma$, such that
		\par
		$(i)\ \varphi_{\zeta,\rho}$ is quasi-plurisubharmonic on a neighborhood of
		$\bar{\Omega}$, smooth on $M\backslash E_{\zeta}(T)$,
		\\
		increasing with respect to
		$\zeta$ and $\rho$ on $\Omega $ and converges to $\varphi$ on $\Omega$ as $\rho
		\to 0$.
		\par
		$(ii)\ T_{\zeta,\rho}\ge\gamma-\delta\varpi-\delta_{\rho}\omega$ on $\Omega$.
		\par
		where $E_{\zeta}(T):=\{x\in M:v(T,x)\ge \zeta\}$ ($\zeta>0$) is the $\zeta$-upper level set of
		Lelong numbers and $\{\delta_{\rho}\}$ is an increasing family of positive numbers
		such that $\lim\limits_{\rho \to 0}\delta_{\rho}=0$.
		\label{lemma24}
	\end{Lemma}
	\begin{Remark}(see Remark 2.1 in \cite{ZZ2019})
		Lemma \ref{lemma24} is stated in \cite{DemaillyReg} in the case $M$ is a compact complex manifold. The similar proof as in \cite{DemaillyReg} shows that lemma \ref{lemma24} on noncompact complex manifold is still hold where the uniform estimate $(i)$ and $(ii)$ are obtained only on a relatively compact subset $\Omega$.
	\end{Remark}
	
	\begin{Lemma}
		(Theorem 1.5 in \cite{Demailly82})
		Let M be a K\"ahler manifold, and Z be an analytic subset of M. Assume that
		$\Omega$ is a relatively compact open subset of M possessing a complete K\"ahler
		metric. Then $\Omega\backslash Z $ carries a complete K\"ahler metric.
		\label{lemma25}
	\end{Lemma}
	
	\begin{Lemma}
		(Lemma 6.9 in \cite{Demailly82})
		Let $\Omega$ be an open subset of $C^n$ and Z be a complex analytic subset of
		$\Omega$. Assume that $v$ is a (p,q-1)-form with $L^2_{loc}$ coefficients and h is
		a (p,q)-form with $L^1_{loc}$ coefficients such that $\bar{\partial}v=h$ on
		$\Omega\backslash Z$ (in the sense of distribution theory). Then
		$\bar{\partial}v=h$ on $\Omega$.
		\label{lemma26}
	\end{Lemma}
	\subsection {Proof of Lemma \ref{lemma2.1}}
	\
	
	In this section, we prove Lemma \ref{lemma2.1}.
	
	Since $M$ is weakly pseudoconvex, there exists a smooth plurisubharmonic
	exhaustion function $P$ on M. Let $M_k:=\{P<k\}$ $(k=1,2,...,) $. We choose $P$ such that
	$M_1\ne \emptyset$.\par
	Then $M_k$ satisfies $M_1 \subset \subset M_2\subset \subset ...\subset \subset
	M_k\subset \subset M_{k+1}\subset \subset ...$ and $\cup_{k=1}^n M_k=M$. Each $M_k$ is complete weakly
	pseudoconvex K\"ahler manifold with exhaustion plurisubharmonic function
	$P_k=1/(k-P)$.
	\par
	\emph{We will fix $k$ during our discussion until step 9.}
	\\
	\par

	\emph{Step 1: Regularization of $\psi$,$\varphi$ and c(t)}
	\\
	\par
	We firstly introduce the regularization process of $\psi$ and $\varphi$.
	
	Take $a_i \in R$ ($i=1,...,n$) largely enough such that
	$\varpi=\sum\limits^{n}_{i=1}a_i dz_i\wedge d\bar{z_i}$ satisfies
	\begin{equation}\nonumber
		(\sqrt{-1}\Theta_{TM}+\varpi \otimes Id_{TM})(\kappa_1 \otimes \kappa_2,\kappa_1
		\otimes \kappa_2)\ge 0,
	\end{equation}
	for $\forall \kappa_1,\kappa_2 \in TM$ on $M_{k+1}$.
	
	\par
	Let $M=M_{k+1}$, $\Omega=M_{k}$, $T=\frac{\sqrt{-1}}{\pi}\partial\bar{\partial}\psi$
	, $\gamma =0$ in Lemma \ref{lemma24}, there exist a family of functions $\psi_{\zeta,\rho}$
	on $M_{k+1}$ such that
	\par
	$(i)\ \psi_{\zeta,\rho}$ is quasi-plurisubharmonic on a neighborhood of
	$\overline{M_k}$, smooth on $M_{k+1}\backslash E_{\zeta}(T)$, increasing with respect
	to $\zeta$ and $\rho$ on $M_k$ and converges to $\psi$ as $\rho \to 0$.
	\par
	$(ii)\
	\frac{\sqrt{-1}}{\pi}\partial\bar{\partial}\psi_{\zeta,\rho}>-\zeta\varpi-\delta_{\rho}\omega$
	on $M_k$,
	where $E_\zeta(T)=\{x\in M:v(T,x)\ge \zeta\}$ ($\zeta>0$) and $\delta_\rho$ is an increasing
	family of positive number such that $\lim\limits_{\rho \to
		0}\delta_{\rho}=0$.
	\par
	Set $\varphi_l=\max\{-l,\varphi\}$ and
	$T_l=\frac{\sqrt{-1}}{\pi}\partial\bar{\partial}\varphi_l$, where $l$ is a positive integer. Note that $v(T_l,x)=0$
	for all $x \in M_{k+1}$. Lemma \ref{lemma24} implies that there exists a family of functions
	$\varphi_{l,\zeta',\rho'}$ on $M_{k+1}$ such that \par
	$(i)\ \varphi_{l,\zeta',\rho'}$ is quasi-plurisubharmonic on a neighborhood of
	$\overline{M_k}$, smooth on $M_{k+1}$, increasing with respect to $\zeta'$ and $\rho'$ on
	$M_k$ and converges to $\varphi$ as $\rho' \to 0$.\par
	$(ii)\
	\frac{\sqrt{-1}}{\pi}\partial\bar{\partial}\varphi_{l,\zeta',\rho'}>-\zeta'\varpi-\delta_{l,\rho'}\omega$
	on $M_k$.
	Where $\delta'_{l,\rho'} $ is an increasing family of positive number such that
	$\lim\limits_{\rho' \to 0}\delta'_{l,\rho'}=0$.
	\par
	\emph{From now on, we will fix the positive integer $l$ during our discussion until step 7.}
	\par
	For fixed $l$, we can assume $\delta_x$ and $\delta'_{l,x}$ are the same function
	of variable $x$ (denoted by $\delta_x$), since we can replace them by
	$\max\{\delta_x,\delta'_{l,x}\}$.
	\par
	Since $M_k$ is relatively compact in M, there exists a positive number $b_k$ such
	that $b_k\omega \ge \varpi$ holds on $M_k$. Since we will fix $k$ until step 9, we will denote $b_k$ by $b$ for simplicity.
	\par
	Take $\zeta=\delta_\rho$ and denote $\psi_{\zeta,\delta_\rho}$ by $\psi_\rho$. For
	fixed $l$, take $\zeta'=\delta_{l,\rho'}$ and denote
	$\varphi_{l,\zeta',\delta_{\rho'}}$ by $\varphi_{l,\rho'}$.
	Take $\rho=\frac{1}{m}$ and $\rho'=\frac{1}{m'}$ (where $m,m'=1,2,3...$ ), we have
	two sequence of functions $\psi_m$(=$\psi_{\frac{1}{m}}$) and $\varphi_{l,m'}$
	(=$\varphi_{l,\frac{1}{m'}}$) satisfy the following:
	\par
	(1) $\psi_m$ is quasi-plurisubharmonic on $\overline{M_k}$, smooth on $M_{k+1} \backslash
	E_{\delta_m}(T)$, decreasing with respect to m on $M_k$, converges to $\psi$ as
	$m\to +\infty$ and
	\begin{equation}\nonumber
		T_m=\frac{\sqrt{-1}}{\pi}\partial\bar{\partial}\psi_m \ge -\delta_m b
		\omega-\delta_m \omega \ge -2b\delta_m \omega
	\end{equation}
	on $M_{k}$.
	\par
	(2) $\varphi_{l,m'}$ is quasi-plurisubharmonic on $\overline{M_k}$, smooth on $M_{k+1} $, decreasing with respect to $m'$ on $M_k$, converges to $\varphi_l$ as $m'\to +\infty$
	and \par
	\begin{equation}\nonumber
		T_{l,m'}=\frac{\sqrt{-1}}{\pi}\partial\bar{\partial}\varphi_{l,m'} \ge
		-\delta_{m'} b \omega-\delta_{m'} \omega \ge -2b\delta_{m'} \omega
	\end{equation}
	on $M_{k}$.\par
	where $\delta_n$ is an decreasing sequence of positive number such that
	$\lim\limits_{n \to +\infty}\delta_n=0$.
	\par
	As $\psi_m$ and $\varphi_{l,m'}$ decreasing convergent to $\psi$ and $\varphi_l$
	respectively.
	We may also assuming that $\sup\limits_m \sup\limits_{M_k}\psi_m<-S$ and $\sup\limits_{m'}
	\sup\limits_{M_k}\varphi_{l,m'}<+\infty$.
	
	We now introduce the regularization process of $c(t)$.
	
	Let $f(x)=2\mathbb{I}_{(-\frac{1}{2},\frac{1}{2})}\ast\rho(x)$ be a smooth function on $\mathbb{R}$, where $\rho$ is the kernel of convolution satisfying $supp(\rho)\subset (-\frac{1}{3},\frac{1}{3})$ and $\rho>0$.
	
	Let $g_n(x)=\left\{ \begin{array}{rcl}
		nf(nx) & \mbox{if}
		&x\le 0 \\ nf(n^2 x) & \mbox{if} & x>0
	\end{array}\right.$, then $\{g_n\}_{n\in \mathbb{N}^+}$ is a family of smooth functions on $\mathbb{R}$ satisfying:
	
	(1) $supp(g)\subset [-\frac{1}{n},\frac{1}{n}]$, $g_n(x)\ge 0$ for any $x\in\mathbb{R}$,
	
	(2) $\int_{-\frac{1}{n}}^0 g_n(x)dx=1$, $\int^{\frac{1}{n}}_0 g_n(x)dx\le\frac{1}{n}$ for any $n \in \mathbb{N}^+$.
	
	Set $c_n(t)=e^t\int_{\mathbb{R}}h(e^y(t-S)+S)g_n(y)dy,$ where $h(t)=c(t)e^{-t}$ and $c(t)\in \tilde{\mathcal{P}}_S$. It is easy to get
	$$c_n(t)-c(t)\ge e^t(\int_{-\frac{1}{n}}^0 (h(e^y(t-S)+S)-h(t))g_n(y))dy\ge 0.$$
	Set $\tilde{h}(t)=h(e^t+S)$ and $\tilde{g}(t)=g(-t)$, then $c_n(t)=e^t\tilde{h}\ast \tilde{g}_n(\ln(t-s))\in C^{\infty}(S,+\infty)$. Because $h(t)$ is decreasing with respect to $t$, so is $c_n(t)e^{-t}$. As
	\begin{equation}\nonumber
		\begin{split}
			\int_S^s c_n(t)e^{-t}dt&=\int_S^s \int_{\mathbb{R}}h(e^{y}(t-S)+S)g_n(y)dydt  \\
			&=\int_{\mathbb{R}}e^{-y}g_n(y)\int_S^{e^y(s-S)+S}h(t)dtdy\\
			&\le \int_{\mathbb{R}}e^{-y}g_n(y)dy\int_S^{e(s-S)+S}h(t)dtdy\\
			&<+\infty,
		\end{split}
	\end{equation}
	then $c_n(t)\in\tilde{\mathcal{P}}_S$ for any $n\in \mathbb{N}^+$.
	
	As $h(t)$ is decreasing with respect to $t$, then set $h^{-}(t)=\lim\limits_{s\to t-0}h(s)\ge h(t)$ and $c^{-}(t)=\lim\limits_{s\to t-0}c(s)\ge c(t)$, then we claim that $\lim\limits_{n\to +\infty}c_n(t)=c^{-}(t)$. In fact, we have
	\begin{equation}
		\begin{split}
			|c_n(t)-c^{-}(t)|&\le e^t\int_{-\frac{1}{n}}^0|h(e^y(t-S)+S)-h^{-}(t)|g_n(y)dy\\
			&+e^t\int_{0}^{\frac{1}{n}}h(e^y(t-S)+S)g_n(y)dy.
			\label{clt}
		\end{split}
	\end{equation}
	For any $\epsilon>0 $, there exists $\delta>0$ such that $|h(t-\delta)-h^{-}(t)|<\epsilon$. Then $\exists N>0$, such that for any $n>N$, $e^y(t-S)+S>t-\delta$ for all $y \in [-\frac{1}{n},0)$ and $\frac{1}{n}<\epsilon$. It follows from \eqref{clt} that
	$$|c_n(t)-c^{-}(t)|\le \epsilon e^t+\epsilon h(t)e^t,$$
	hence $\lim\limits_{n\to+\infty}c_n(t)=c^{-}(t)$ for any $t>S$.
	
	\
	
	\emph{Step 2: Recall some notations}
	\\
	\par
	To simplify our notation, we denote $b_{t_0,B}(t)$ by $b(t)$ and $v_{t_0,B}(t)$ by $v(t)$.
	
	Let $\epsilon \in (0,\frac{1}{8}B)$. Let $\{v_\epsilon\}_{\epsilon \in
		(0,\frac{1}{8}B)}$ be a family of smooth increasing convex functions on $\mathbb{R}$, such
	that:
	\par
	(1) $v_{\epsilon}(t)=t$ for $t\ge-t_0-\epsilon$, $v_{\epsilon}(t)=constant$ for
	$t<-t_0-B+\epsilon$;\par
	(2) $v_{\epsilon}{''}(t)$ are pointwisely convergent
	to $\frac{1}{B}\mathbb{I}_{(-t_0-B,-t_0)}$,when $\epsilon \to 0$, and $0\le
	v_{\epsilon}{''}(t) \le \frac{2}{B}\mathbb{I}_{(-t_0-B+\epsilon,-t_0-\epsilon)}$
	for ant $t \in \mathbb{R}$;\par
	(3) $v_{\epsilon}{'}(t)$ are pointwisely convergent to $b(t)$ which is a continuous
	function on R when $\epsilon \to 0$ and $0 \le v_{\epsilon}{'}(t) \le 1$ for any
	$t\in \mathbb{R}$.\par
	One can construct the family $\{v_\epsilon\}_{\epsilon \in (0,\frac{1}{8}B)}$  by
	the setting
	\begin{equation}\nonumber
		\begin{split}
			v_\epsilon(t):=&\int_{-\infty}^{t}(\int_{-\infty}^{t_1}(\frac{1}{B-4\epsilon}
			\mathbb{I}_{(-t_0-B+2\epsilon,-t_0-2\epsilon)}*\rho_{\frac{1}{4}\epsilon})(s)ds)dt_1\\
			&-\int_{-\infty}^{-t_0}(\int_{-\infty}^{t_1}(\frac{1}{B-4\epsilon}
			\mathbb{I}_{(-t_0-B+2\epsilon,-t_0-2\epsilon)}*\rho_{\frac{1}{4}\epsilon})(s)ds)dt_1-t_0,
		\end{split}
	\end{equation}
	where $\rho_{\frac{1}{4}\epsilon}$ is the kernel of convolution satisfying
	$supp(\rho_{\frac{1}{4}\epsilon})\subset
	(-\frac{1}{4}\epsilon,{\frac{1}{4}\epsilon})$.
	Then it follows that
	\begin{equation}\nonumber
		v_\epsilon{''}(t)=\frac{1}{B-4\epsilon}
		\mathbb{I}_{(-t_0-B+2\epsilon,-t_0-2\epsilon)}*\rho_{\frac{1}{4}\epsilon}(t),
	\end{equation}
	and
	\begin{equation}\nonumber
		v_\epsilon{'}(t)=\int_{-\infty}^{t}(\frac{1}{B-4\epsilon}
		\mathbb{I}_{(-t_0-B+2\epsilon,-t_0-2\epsilon)}*\rho_{\frac{1}{4}\epsilon})(s)ds.
	\end{equation}
	\par
	Let $\eta=s(-v_\epsilon(\psi_m))$ and $\phi=u(-v_\epsilon(\psi_m))$, where $s \in
	C^{\infty}((T,+\infty))$ satisfies $s\ge 0$ and $u\in C^{\infty}((T,+\infty))$, such that $u''s-s''>0$
	and $s'-u's=1$. It follows form $\sup\limits_m \sup\limits_{M_k}\psi_m<-T$  that
	$\phi=u(-v_\epsilon(\psi_m))$ are uniformly bounded on $M_k$ with respect to $m$
	and $\epsilon$, and $u(-v_\epsilon(\psi))$ are uniformly bounded on $M_k$ with
	respect to $\epsilon$.
	Let $\Phi=\phi+\varphi_{l,m'}$ and let $\widetilde{h}=e^{-\Phi}$.
	\\
	\par
	
	\emph{Step 3: Solving $\bar{\partial}$-equation with correcting term}
	\\
	\par
	
	Set $B=[\eta \sqrt{-1}\Theta_{\widetilde{h}}-\sqrt{-1}\partial \bar{\partial}
	\eta-\sqrt{-1}g\partial\eta \wedge\bar{\partial}\eta, \Lambda_{\omega}]$, where
	$g$
	is a positive continuous function on $M_k$. We will determine $g$ by calculations
	\begin{equation}\nonumber
		\begin{split}
			\partial\bar{\partial}\eta=&
			-s'(-v_{\epsilon}(\psi_m))\partial\bar{\partial}(v_{\epsilon}(\psi_m))
			+s''(-v_{\epsilon}(\psi_m))\partial(v_{\epsilon}(\psi_m))\wedge
			\bar{\partial}(v_{\epsilon}(\psi_m)),\\
			\eta\Theta_{\widetilde{h}}=&\eta\partial\bar{\partial}\phi+\eta\partial\bar{\partial}\varphi_{l,m'}\\
			=&su''(-v_{\epsilon}(\psi_m))\partial(v_{\epsilon}(\psi_m))\wedge
			\bar{\partial}(v_{\epsilon}(\psi_m))
			-su'(-v_{\epsilon}(\psi_m))\partial\bar{\partial}(v_{\epsilon}(\psi_m))\\
			+&s\partial\bar{\partial}\varphi_{l,m'}.
		\end{split}
	\end{equation}
	\par
	Hence
	\begin{equation}\nonumber
		\begin{split}
			&\eta \sqrt{-1}\Theta_{\widetilde{h}}-\sqrt{-1}\partial \bar{\partial}
			\eta-\sqrt{-1}g\partial\eta \wedge\bar{\partial}\eta\\
			=&s\sqrt{-1}\partial\bar{\partial}\varphi_{l,m'}\\
			+&(s'-su')(v'_{\epsilon}(\psi_m)\sqrt{-1}\partial\bar{\partial}(\psi_m)+
			v''_\epsilon(\psi_m)\sqrt{-1}\partial(\psi_m)\wedge\bar{\partial}(\psi_m))\\
			+&[(u''s-s'')-gs'^2]\partial(v_\epsilon(\psi_m))\wedge\bar{\partial}(v_\epsilon(\psi_m)).
		\end{split}
	\end{equation}
	\par
	Let $g=\frac{u''s-s''}{s'^2}(-v_\epsilon(\psi_m))$ and note that $s'-su'=1$,
	$v'_\epsilon(\psi_m)\ge 0$. Then
	\begin{equation}\nonumber
		\begin{split}
			&\eta \sqrt{-1}\Theta_{\widetilde{h}}-\sqrt{-1}\partial \bar{\partial}
			\eta-\sqrt{-1}g\partial\eta \wedge\bar{\partial}\eta\\
			=&s\sqrt{-1}\partial\bar{\partial}\varphi_{l,m'}+v'_{\epsilon}(\psi_m)\sqrt{-1}\partial\bar{\partial}(\psi_m)+
			v''_\epsilon(\psi_m)\sqrt{-1}\partial(\psi_m)\wedge\bar{\partial}(\psi_m)\\
			\ge&-s2b\delta_{m'}\omega-v'_\epsilon(\psi_m)2b\delta_m\omega+
			v''_\epsilon(\psi_m)\sqrt{-1}\partial(\psi_m)\wedge\bar{\partial}(\psi_m).
		\end{split}
	\end{equation}
	\par
	Note that $0 \le v'_\epsilon(\psi_m) \le 1$ is uniformly bounded on $M_k$ with respect to
	$m$ and $\epsilon$. By the construction of $v_\epsilon(t)$ and $\sup\limits_m
	\sup\limits_{M_k}\psi_m<-S$, $s(-v_\epsilon(\psi_m))$ is uniformly bounded above on $M_k$ with respect to $m$
	and $\epsilon$.
	\par
	Let $\tilde{M}$ be the common upper bound of $v'_\epsilon(\psi_m)$ and $s(-v_\epsilon(\psi_m))$, then on $M_k$
	\begin{equation}\nonumber
		\begin{split}
			&\eta \sqrt{-1}\Theta_{\widetilde{h}}-\sqrt{-1}\partial \bar{\partial}
			\eta-\sqrt{-1}g\partial\eta \wedge\bar{\partial}\eta\\
			\ge&-s2b\delta_{m'}\omega-v'_\epsilon(\psi_m)2b\delta_m\omega+
			v''_\epsilon(\psi_m)\sqrt{-1}\partial(\psi_m)\wedge\bar{\partial}(\psi_m)\\
			\ge&-2b\tilde{M}(\delta_m+\delta_{m'})\omega+
			v''_\epsilon(\psi_m)\sqrt{-1}\partial(\psi_m)\wedge\bar{\partial}(\psi_m).
		\end{split}
	\end{equation}
	Hence
	\begin{equation}
		\begin{split}
			&\langle(B+2b\tilde{M}(\delta_m+\delta_{m'})I)\alpha,\alpha\rangle_{\widetilde h}\\
			\ge&\langle[v''_\epsilon(\psi_m)\partial(\psi_m)\wedge\bar{\partial}(\psi_m),
			\Lambda_{\omega}]\alpha,\alpha\rangle_{\widetilde h}\\
			=&\langle(v''_\epsilon(\psi_m)\bar{\partial}(\psi_m)
			\wedge(\alpha\llcorner(\bar{\partial}\psi_m)^{\sharp})),\alpha\rangle_{\widetilde
				h}.
		\end{split}
		\label{320}
	\end{equation}
	\par
	By Lemma \ref{lemma22}, $B+2b\tilde{M}(\delta_m+\delta_{m'})I$ is semipositive.
	\par
	Using the definition of contraction, Cauchy-Schwarz inequality
	and the inequality \eqref{320}, we have
	\begin{equation}
		\begin{split}
			|\langle
			v''_\epsilon(\psi_m)\bar{\partial}\psi_m\wedge\gamma,\widetilde{\alpha}\rangle_
			{\widetilde h}|^2=
			&|\langle
			v''_\epsilon(\psi_m)\gamma,\widetilde{\alpha}\llcorner(\bar{\partial}\psi_m)^{\sharp}
			\rangle_{\widetilde h}|^2\\
			\le&\langle
			(v''_\epsilon(\psi_m)\gamma,\gamma)
			\rangle_{\widetilde h}
			(v''_\epsilon(\psi_m))|\widetilde{\alpha}\llcorner(\bar{\partial}\psi_m)^{\sharp}|^2_{\widetilde
				h}\\
			=&\langle
			(v''_\epsilon(\psi_m)\gamma,\gamma)
			\rangle_{\widetilde h}
			\langle
			(v''_\epsilon(\psi_m))\bar{\partial}\psi_m\wedge
			(\widetilde{\alpha}\llcorner(\bar{\partial}\psi_m)^{\sharp}),\widetilde{\alpha}
			\rangle_{\widetilde h}\\
			\le&\langle
			(v''_\epsilon(\psi_m)\gamma,\gamma)
			\rangle_{\widetilde h}
			\langle
			(B+2b\tilde{M}(\delta_m+\delta_{m'}I)\widetilde{\alpha},\widetilde{\alpha})
			\rangle_{\widetilde h}
			\label{421}
		\end{split}
	\end{equation}
	for any $(n,0)$ form $\gamma$ and (n,1)-form $\widetilde{\alpha}$.
	\par
	As $F$ is holomorphic on $\{\psi < -t_0\}\supset \supset Supp(1-v'_\epsilon(\psi_m))$
	, then $\lambda:= \bar{\partial}[(1-v'_{\epsilon}(\psi_m))F] $ is well-defined and smooth on
	$M_k$.
	\par
	Taking  $\gamma=F$, $\widetilde{\alpha}=
	(B+2b\tilde{M}(\delta_m+\delta_{m'}I))^{-1}(\bar{\partial}v'_{\epsilon}(\psi_m))\wedge
	F$. Note that $\widetilde h=e^{-\Phi}$, using inequality (\ref{421}), we have
	\begin{equation}\nonumber
		\langle
		(B+2b\tilde{M}(\delta_m+\delta_m')I)^{-1}\lambda,\lambda
		\rangle_{\widetilde h}
		\le
		v''_{\epsilon}(\psi_m)|F|^2e^{-\Phi}.
	\end{equation}
	\par
	Then it follows that
	\begin{equation}\nonumber
		\int_{M_k \backslash E_{\delta_m}(T)}\langle
		(B+2b\tilde{M}(\delta_m+\delta_m'))^{-1}\lambda,\lambda
		\rangle_{\widetilde h}
		\le
		\int_{M_k \backslash E_{\delta_m}(T)}
		v''_{\epsilon}(\psi_m)|F|^2e^{-\Phi}
		< +\infty.
		\label{323}
	\end{equation}
	\par
	By Lemma \ref{lemma25}, $M_k \backslash E_{\delta_m}(T)$ carries a complete K\"ahler metric.
	\par
	Using Lemma \ref{lemma23}, there exists $u_{k,m,m',\epsilon,l} \in L^2(M_k \backslash
	E_{\delta_m}(T),K_M)$ and $h_{k,m,m',\epsilon,l} \in L^2(M_k \backslash
	E_{\delta_m}(T),\wedge^{n,1}T^*M)$ such that
	\begin{equation}
		D''u_{k,m,m',\epsilon,l}+\sqrt{2b\tilde{M}(\delta_m+\delta_m')}h_{k,m,m',\epsilon,l}=\lambda,
		\label{324}
	\end{equation}
	and
	\begin{equation}
		\begin{split}
			&\int_{M_k \backslash E_{\delta_m}(T)}
			\frac{1}{\eta+g^{-1}}|u_{k,m,m',\epsilon,l}|^2 e^{-\Phi}+
			\int_{M_k \backslash E_{\delta_m}(T)}
			|h_{k,m,m',\epsilon,l}|^2 e^{-\Phi}\\
			\le&
			\int_{M_k \backslash E_{\delta_m}(T)}
			\langle
			(B+2b\tilde{M}(\delta_m+\delta_m')I)^{-1}\lambda,\lambda
			\rangle_{\widetilde h}\\
			\le&
			\int_{M_k \backslash E_{\delta_m}(T)}
			v''_{\epsilon}(\psi_m)|F|^2 e^{-\Phi}
			<+\infty.
			\label{325}
		\end{split}
	\end{equation}
	\par
	Note that $g=\frac{u''s-s''}{s'^2}(-v_\epsilon(\psi_m))$. Assume that we can choose
	$\eta$ and $\phi$ such that
	$(\eta+g^{-1})^{-1}=e^{v_\epsilon(\psi_m)}e^{\phi}c_n(-v_\epsilon(\psi_m))$. Now the solution $u$ and correcting term $h$ rely on the parameter $n$, hence we denote $u$ and $h$ by $u_{n,k,m,m',\epsilon,l}$ and $h_{n,k,m,m',\epsilon,l}$. Then
	inequality (\ref{325}) becomes
	\begin{equation}
		\begin{split}
			&\int_{M_k \backslash E_{\delta_m}(T)}
			|u_{n,k,m,m',\epsilon,l}|^2
			e^{v_\epsilon(\psi_m)-\varphi_{l,m'}}c_n(-v_\epsilon(\psi_m))+
			\int_{M_k \backslash E_{\delta_m}(T)}
			|h_{n,k,m,m',\epsilon,l}|^2 e^{-\phi-\varphi_{l,m'}}\\
			\le&
			\int_{M_k \backslash E_{\delta_m}(T)}
			v''_{\epsilon}(\psi_m)|F|^2 e^{-\phi-\varphi_{l,m'}}
			<+\infty.
		\end{split}
		\label{326}
	\end{equation}
	\par
	Note that on $M_k \subset \subset M_{k+1}$, for fixed $m',m,\epsilon,l$, functions
	$e^{-\varphi_{l,m'}}$is smooth hence bounded and by the construction of
	$v_\epsilon$ and $\sup\limits_m \sup\limits_{M_k}\psi_m<-S$,
	$e^{v_\epsilon(\psi_m)}c_n(-v_\epsilon(\psi_m))$ is also bounded. Hence from
	\eqref{326}, we know
	\begin{equation}\nonumber
		\begin{split}
			&u_{n,k,m,m',\epsilon,l} \in L^2(M_k,K_M),\\
			&h_{n,k,m,m',\epsilon,l} \in L^2(M_k,\wedge^{n,1}T^*M).
		\end{split}
	\end{equation}
	\par
	It follow from(\ref{324}), (\ref{326})and Lemma \ref{lemma26}, that
	\begin{equation}
		D''u_{n,k,m,m',\epsilon,l}+\sqrt{2b\tilde{M}(\delta_m+\delta_m')}h_{n,k,m,m',\epsilon,l}=\lambda \label {equation}
	\end{equation}
	holds on $M_k$ and
	\begin{equation}
		\begin{split}
			&\int_{M_k}
			|u_{n,k,m,m',\epsilon,l}|^2
			e^{v_\epsilon(\psi_m)-\varphi_{l,m'}}c_n(-v_\epsilon(\psi_m))+
			\int_{M_k}
			|h_{n,k,m,m',\epsilon,l}|^2 e^{-\phi-\varphi_{l,m'}}\\
			\le&
			\int_{M_k }
			v''_{\epsilon}(\psi_m)|F|^2 e^{-\phi-\varphi_{l,m'}}.
		\end{split}
		\label{329}
	\end{equation}
	\par
	\
	
	\emph{Step 4: when $m'$ goes to $+\infty$.}
	\\
	\par
	As $\varphi_{l,m'}$ decreasingly converge to $\varphi_l$ as $m' \to +\infty$, by
	Levi's theorem
	\begin{equation}
		\begin{split}
			&\lim\limits_{m' \to +\infty} \int_{M_k }
			v''_{\epsilon}(\psi_m)|F|^2 e^{-\phi-\varphi_{l,m'}}
			=\int_{M_k }
			v''_{\epsilon}(\psi_m)|F|^2 e^{-\phi-\varphi_{l}}\\
			\le&
			\sup\limits_{m}\sup\limits_{\epsilon}\sup\limits_{M_k}|e^{-\phi-\varphi_l}|
			\int_{M_k}
			\frac{2}{B}\mathbb{I}_{\{\psi<-t_0\}}|F|^2
			<+\infty.
		\end{split}
		\label{330}
	\end{equation}
	
	Note that
	$\inf\limits_{m'}\inf\limits_{M_k}
	e^{v_{\epsilon}(\psi_m)-\varphi_{l,m'}}c_n(-v_\epsilon(\psi_m))>0$, then it follows
	from (\ref{329}) and (\ref{330}) that $\sup\limits_{m'}\int_{M_k }
	|u_{n,k,m,m',\epsilon,l}|^2 < +\infty$.\par
	
	Therefore the solution $u_{k,m,m',\epsilon,l}$ are uniformly bounded in $L^2$ norm
	with respect to $m'$ on $M_k$. Since the closed unit ball of the Hilbert space is
	weakly compact, we can extract a subsequence $u_{n,k,m,m'',\epsilon,l}$ weakly
	converge to $u_{n,k,m,\epsilon,l}$ in $L^2(M_k,K_M)$.
	\par
	Now we want to show that for fixed $m'$,  $u_{n,k,m,m'',\epsilon,l}\sqrt{e^{v_\epsilon(\psi_m)-\varphi_{l,m'}}c(-v_\epsilon(\psi_m))}$ weakly converge to $u_{n,k,m,\epsilon,l}\sqrt{e^{v_\epsilon(\psi_m)-\varphi_{l,m'}}c(-v_\epsilon(\psi_m))}$
	in $L^2(M_k,K_M)$.
	\par
	Let $g \in L^2(M_k,K_M)$. As for fixed $m'$,  $|e^{v_\epsilon(\psi_m)-\varphi_{l,m'}}c_n(-v_\epsilon(\psi_m))|$ is bounded on
	$\overline{M_k}$, then $\sqrt{e^{v_\epsilon(\psi_m)-\varphi_{l,m'}}c_n(-v_\epsilon(\psi_m))}g \in L^2(M_k,K_M)$. \par
	As $u_{n,k,m,m'',\epsilon,l}$ weakly converge to $u_{n,k,m,\epsilon,l}$, we have
	\begin{equation}\nonumber
		\begin{split}
			&\lim_{m'' \to +\infty} \langle u_{n,k,m,m'',\epsilon,l},\sqrt{e^{v_\epsilon(\psi_m)-\varphi_{l,m'}}c_n(-v_\epsilon(\psi_m))}g \rangle
			\\
			=&
			\langle u_{n,k,m,\epsilon,l},\sqrt{e^{v_\epsilon(\psi_m)-\varphi_{l,m'}}c_n(-v_\epsilon(\psi_m))}g \rangle.
		\end{split}
	\end{equation}
	\par
	This means
	\begin{equation}\nonumber
		\begin{split}
			&\lim_{m'' \to +\infty} \langle u_{n,k,m,m'',\epsilon,l}\sqrt{e^{v_\epsilon(\psi_m)-\varphi_{l,m'}}c_n(-v_\epsilon(\psi_m))},g \rangle
			\\=&
			\langle u_{n,k,m,\epsilon,l}\sqrt{e^{v_\epsilon(\psi_m)-\varphi_{l,m'}}c_n(-v_\epsilon(\psi_m))},g \rangle.
		\end{split}
	\end{equation}
	\par
	Hence for fixed $m'$,  $u_{n,k,m,m'',\epsilon,l}\sqrt{e^{v_\epsilon(\psi_m)-\varphi_{l,m'}}c_n(-v_\epsilon(\psi_m))}$ weakly converge to \\ $u_{n,k,m,\epsilon,l}\sqrt{e^{v_\epsilon(\psi_m)-\varphi_{l,m'}}c_n(-v_\epsilon(\psi_m))}$
	in $L^2(M_k,K_M)$.
	\par
	Then
	\begin{flalign}
		&\int_{M_k}
		|u_{n,k,m,\epsilon,l}|^2 e^{v_\epsilon(\psi_m)-\varphi_{l,m'}}c_n(-v_\epsilon(\psi_m))\nonumber\\
		\le &
		\liminf\limits_{m'' \to +\infty}
		\int_{M_k}
		|u_{n,k,m,m'',\epsilon,l}|^2
		e^{v_\epsilon(\psi_m)-\varphi_{l,m'}}c_n(-v_\epsilon(\psi_m))\nonumber\\
		\le &
		\liminf\limits_{m'' \to +\infty}
		\int_{M_k}
		|u_{n,k,m,m'',\epsilon,l}|^2
		e^{v_\epsilon(\psi_m)-\varphi_{l,m''}}c_n(-v_\epsilon(\psi_m))\nonumber\\
		\le &
		\liminf\limits_{m'' \to +\infty}
		\int_{M_k}
		v''_{\epsilon}(\psi_m)|F|^2e^{-\phi-\varphi_{l,m''}}\nonumber\\
		=&
		\int_{M_k}
		v''_{\epsilon}(\psi_m)|F|^2e^{-\phi-\varphi_{l}}.\nonumber
	\end{flalign}
	\par
	By Levi Theorem
	\begin{equation}\nonumber
		\begin{split}
			&\lim_{m' \to +\infty}\int_{M_k}
			|u_{n,k,m,\epsilon,l}|^2 e^{v_\epsilon(\psi_m)-\varphi_{l,m'}}c_n(-v_\epsilon(\psi_m))
			\\=&
			\int_{M_k}
			|u_{n,k,m,\epsilon,l}|^2 e^{v_\epsilon(\psi_m)-\varphi_{l}}c_n(-v_\epsilon(\psi_m)).
		\end{split}
	\end{equation}
	\par
	Hence
	\begin{equation}\nonumber
		\begin{split}
			\int_{M_k}
			|u_{n,k,m,\epsilon,l}|^2 e^{v_\epsilon(\psi_m)-\varphi_{l}}c_n(-v_\epsilon(\psi_m))
			\le
			\int_{M_k}
			v''_{\epsilon}(\psi_m)|F|^2e^{-\phi-\varphi_{l}} .
			\label{s4}
		\end{split}
	\end{equation}
	\par
	As
	$\inf\limits_{m''}\inf\limits_{M_k}
	e^{-\phi-\varphi_{l,m''}}>0$, we also have  $\sup\limits_{m''}\int_{M_k }
	|h_{n,k,m,m'',\epsilon,l}|^2 < +\infty$.\par
	We can extract a subsequence $h_{n,k,m,m''',\epsilon,l}$ weakly
	convergence to $h_{n,k,m,\epsilon,l}$ in $L^2(M_k,\wedge^{n,1}T^*M)$ when $m'''\to +\infty$.
	\par
	Using  similar arguments as above, we know
	\begin{equation}
		\begin{split}
			\int_{M_k}
			|h_{n,k,m,\epsilon,l}|^2e^{-\phi-\varphi_l}
			\le
			\int_{M_k}
			v''_{\epsilon}(\psi_m)|F|^2e^{-\phi-\varphi_{l}}.
			\label{estimate of h}
		\end{split}
	\end{equation}
	Replace $m'$ by $m'''$ in (\ref{equation}) and let $m'''\to +\infty$,
	we have
	\begin{equation}
		D''u_{n,k,m,\epsilon,l}+\sqrt{2b\tilde{M}\delta_m}h_{n,k,m,\epsilon,l}=\lambda.
		\label{equation2}
	\end{equation}
	
	\
	\par
	\emph{Step 5: when $m$ goes to $+\infty$.}
	\\
	\par
	
	As $e^{-\varphi_l}< +\infty$  and
	$
	\sup\limits_{m,\epsilon}|\phi|=
	\sup\limits_{m,\epsilon}|u(-v_\epsilon(\psi_m))|< +\infty
	$ on $M_k$,
	then
	$\sup\limits_{m,\epsilon}e^{-\phi-\varphi_{l}}< +\infty$.
	Note that
	\begin{equation}\nonumber
		v''_{\epsilon}(\psi_m)|F|^2 e^{-\phi-\varphi_l} \le
		\frac{2}{B}\mathbb{I}_{\{\psi<-t_0\}}|F|^2
		\sup\limits_{m,\epsilon}e^{-\phi-\varphi_l}
	\end{equation}
	on $M_k$, then it follows from the dominated convergence theorem and
	\begin{equation}\nonumber
		\int_{\{\psi<-t_0\}\cap \overline{M_k}}|F|^2 < +\infty,
	\end{equation}
	we have
	\begin{equation}
		\begin{split}
			\lim\limits_{m \to +\infty} \int_{M_k }
			v''_{\epsilon}(\psi_m)|F|^2 e^{-\phi-\varphi_{l}}
			=\int_{M_k }
			v''_{\epsilon}(\psi)|F|^2 e^{-u(v_\epsilon(\psi))-\varphi_{l}}.
			\label{s5.1}
		\end{split}
	\end{equation}
	
	\par
	As $\inf\limits_{m}\inf\limits_{M_k}
	e^{v_{\epsilon}(\psi_m)-\varphi_{l}}c_n(-v_\epsilon(\psi_m))>0$, then it follows from
	(\ref{s4}) and (\ref{s5.1}) that $\sup\limits_{m}\int_{M_k } |u_{n,k,m,\epsilon,l}|^2
	< +\infty$.
	\par
	Therefore $u_{n,k,m,\epsilon,l}$ are uniformly bounded in $L^2$ norm
	with respect to $m$ on $M_k$. Since the closed unit ball of the Hilbert space is
	weakly compact, we can extract a subsequence $u_{n,k,m_{i},\epsilon,l}$ weakly
	converge to $u_{n,k,\epsilon,l}$ in $L^2(M_k,K_M)$. Note that
	$\sqrt{e^{v_\epsilon(\psi_{m})-\varphi_{l}}c_n(-v_\epsilon(\psi_{m}))}$ are
	pointwisely convergent to $\sqrt{e^{v_\epsilon(\psi)-\varphi_{l}}c_n(-v_\epsilon(\psi))}$, as $m\to +\infty$.
	\par
	We want to prove, as $m_{i}\to +\infty$, $u_{n,k,m_{i},\epsilon,l}\sqrt{e^{v_\epsilon(\psi_{m_i})-\varphi_{l}}c_n(-v_\epsilon(\psi_{m_i}))}$
	weakly converge to
	$u_{n,k,\epsilon,l}\sqrt{e^{v_\epsilon(\psi)-\varphi_{l}}c_n(-v_\epsilon(\psi))}$. Take arbitrary $h$ in $L^2(M_k,K_M)$.\par
	Consider
	\begin{equation}\nonumber
		\begin{split}
			I=&|\langle u_{n,k,m_{i},\epsilon,l}\sqrt{e^{v_\epsilon(\psi_{m_i})-\varphi_{l}}c_n(-v_\epsilon(\psi_{m_i}))},h\rangle-
			\langle
			u_{n,k,\epsilon,l}\sqrt{e^{v_\epsilon(\psi)-\varphi_{l}}c_n(-v_\epsilon(\psi))},h\rangle|\\
			=&
			|\int_{M_k}\sqrt{e^{v_\epsilon(\psi_{m_i})-\varphi_{l}}c_n(-v_\epsilon(\psi_{m_i}))}u_{n,k,m_{i},\epsilon,l}\bar{h}-
			\int_{M_k}\sqrt{e^{v_\epsilon(\psi)-\varphi_{l}}c_n(-v_\epsilon(\psi))}u_{n,k,\epsilon,l}\bar{h}|\\
			\le &
			|\int_{M_k}\sqrt{e^{v_\epsilon(\psi_{m_i})-\varphi_{l}}c_n(-v_\epsilon(\psi_{m_i}))}u_{n,k,m_{i},\epsilon,l}\bar{h}-
			\int_{M_k}\sqrt{e^{v_\epsilon(\psi_{m_i})-\varphi_{l}}c_n(-v_\epsilon(\psi_{m_i}))}u_{n,k,\epsilon,l}\bar{h}|\\
			+&
			|\int_{M_k}\sqrt{e^{v_\epsilon(\psi_{m_i})-\varphi_{l}}c_n(-v_\epsilon(\psi_{m_i}))}u_{n,k,\epsilon,l}\bar{h}-
			\int_{M_k}\sqrt{e^{v_\epsilon(\psi)-\varphi_{l}}c_n(-v_\epsilon(\psi))}u_{n,k,\epsilon,l}\bar{h}|\\
			=&
			I_1+I_2.
		\end{split}
	\end{equation}
	\par
	Note that $\sqrt{e^{v_\epsilon(\psi_{m_i})-\varphi_{l}}c_n(-v_\epsilon(\psi_{m_i}))}$ is uniformly bounded with respect to $m_i$ on $M_k$.
	There exists constant $M_{\epsilon,l,k}$ such that
	\begin{equation}\nonumber
		\begin{split}
			I_1\le & M_{\epsilon,l,k} |\int_{M_k}u_{n,k,m_{i},\epsilon,l}\bar{h}-
			\int_{M_k}u_{n,k,\epsilon,l}\bar{h}|\\
			=&
			M_{\epsilon,l,k} |\langle u_{n,k,m_{i},\epsilon,l},h\rangle-
			\langle u_{n,k,\epsilon,l},h\rangle|.
		\end{split}
	\end{equation}
	By the definition of weakly convergence, $\lim\limits_{m_i \to +\infty} I_1=0$.
	\par
	For $I_2$, it follows from dominated convergence theorem that
	\begin{equation}\nonumber
		\begin{split}
			\lim\limits_{m_i \to +\infty}
			&\int_{M_k}\sqrt{e^{v_\epsilon(\psi_{m_i})-\varphi_{l}}c_n(-v_\epsilon(\psi_{m_i}))}u_{n,k,\epsilon,l}\bar{h}\\
			=&
			\int_{M_k}\sqrt{e^{v_\epsilon(\psi)-\varphi_{l}}c_n(-v_\epsilon(\psi))}u_{n,k,\epsilon,l}\bar{h}.
		\end{split}
	\end{equation}
	\par
	Thus $\lim\limits_{m_i \to +\infty} I=0$, i.e. $u_{n,k,m_{i},\epsilon,l}\sqrt{e^{v_\epsilon(\psi_{m_i})-\varphi_{l}}c_n(-v_\epsilon(\psi_{m_i}))}$
	weakly converge to
	$u_{n,k,\epsilon,l}\sqrt{e^{v_\epsilon(\psi)-\varphi_{l}}c_n(-v_\epsilon(\psi))}$.
	\par
	By (\ref{s5.1}) and the weakly convergence property of
	$u_{n,k,m_{i},\epsilon,l}\sqrt{e^{v_\epsilon(\psi_{m_i})-\varphi_{l}}c_n(-v_\epsilon(\psi_{m_i}))}$
	we have
	\begin{equation}
		\begin{split}
			&\int_{M_k}|u_{n,k,\epsilon,l}|^2 e^{v_\epsilon(\psi)-\varphi_{l}}c_n(-v_\epsilon(\psi))\\
			\le&
			\liminf_{m_{i} \to +\infty}\int_{M_k}|u_{n,k,m_{i},\epsilon,l}|^2 e^{v_\epsilon(\psi_{m_{i}})-\varphi_{l}}c_n(-v_\epsilon(\psi_{m_{i}}))\\
			\le&
			\liminf_{m_{i} \to +\infty}\int_{M_k}v''_{\epsilon}(\psi_{m_i})|F|^2 e^{-u(-v_\epsilon(\psi_{m_i}))-\varphi_{l}}\\
			= &
			\int_{M_k}v''_{\epsilon}(\psi)|F|^2 e^{-u(-v_\epsilon(\psi))-\varphi_{l}}.
			\label{s5.3}
		\end{split}
	\end{equation}
	\par
	
	Note that
	$\inf\limits_{m_i}\inf\limits_{M_k}
	e^{-u(-v_\epsilon(\psi_{m_i}))-\varphi_{l}}>0$, it follows from \eqref{estimate of h}, we also have
	$\sup\limits_{m_i}\int_{M_k } |h_{n,k,m_i,\epsilon,l}|^2 < +\infty$.\par
	We can extract a subsequence (also denoted by $h_{n,k,m_i,\epsilon,l}$)
	such that $h_{n,k,m_i,\epsilon,l}$ weakly convergence to $h_{n,k,\epsilon,l}$ in
	$L^2(M_k,\wedge^{n,1}T^*M)$.\par
	As $\lim\limits_{ m_i \to +\infty}\delta_{m_i}=0$, we know $\sqrt{2b\tilde{M}\delta_{m_i}}h_{n,k,m_i,\epsilon,l}$ weakly converge to 0, as $m_i \to +\infty$.
	\par
	Replace $m$ by $m_i$ in (\ref {equation2}) and let $m_i \to +\infty$, by the weak continuity of differentiations, we get
	\begin{equation}\nonumber
		D''u_{n,k,\epsilon,l}=D''[(1-v'_{\epsilon}(\psi))F].
	\end{equation}
	\par
	Let $F_{n,k,\epsilon,l}=-u_{n,k,\epsilon,l}+(1-v'_\epsilon(\psi))F$, then $D''
	F_{n,k,\epsilon,l}=0$, i.e. $F_{n,k,\epsilon,l}$ is holomorphic $(n,0)$ form on $M_k$.
	Inequality (\ref {s5.3}) becomes
	\begin{equation}
		\begin{split}
			&\int_{M_k}
			|F_{n,k,\epsilon,l}-(1-v'_\epsilon(\psi))F|^2
			e^{v_\epsilon(\psi)-\varphi_{l}}c_n(-v_\epsilon(\psi))
			\le
			\int_{M_k}
			v''_{\epsilon}(\psi)|F|^2e^{-u(-v_\epsilon(\psi))-\varphi_{l}}.
		\end{split}
		\label{343}
	\end{equation}
	\par
	\
	
	\emph{Step 6: when $\epsilon$ goes to $0$.}
	\\
	\par
	Note that $\sup\limits_{\epsilon,M_k}e^{-u(-v_\epsilon(\psi))-\varphi_l} <
	+ \infty$ and
	\begin{equation}\nonumber
		v''_{\epsilon}(\psi)|F|^2 e^{-u(-v_\epsilon(\psi))-\varphi_l} \le
		\frac{2}{B}\mathbb{I}_{\{-t_0-B<\psi<-t_0\}}|F|^2
		\sup\limits_{\epsilon,M_k}e^{-u(-v_\epsilon(\psi))-\varphi_l},
	\end{equation}
	then it follows from
	\begin{equation}\nonumber
		\int_{\{\psi<-t_0\}\cap \overline{M_k}}|F|^2 < +\infty,
	\end{equation}
	and the dominated convergence theorem that
	\begin{equation}
		\begin{split}
			&\lim\limits_{\epsilon \to 0} \int_{M_k }
			v''_{\epsilon}(\psi)|F|^2 e^{-u(-v_\epsilon(\psi))-\varphi_{l}}
			=\int_{M_k }
			\frac{1}{B}\mathbb{I}_{\{-t_0-B<\psi<-t_0\}}|F|^2 e^{-u(-v(\psi))-\varphi_{l}}\\
			\le&
			(\sup\limits_{M_k}e^{-u(-v(\psi))})\int_{M_k }
			\frac{1}{B}\mathbb{I}_{\{-t_0-B<\psi<-t_0\}}|F|^2 e^{-\varphi_{l}}
			< + \infty.
		\end{split}
		\label{346}
	\end{equation}
	\par
	Note that $\inf\limits_{\epsilon}\inf\limits_{M_k}
	e^{v_{\epsilon}(\psi)-\varphi_{l}}c_n(-v_\epsilon(\psi))>0$, then it follows from
	(\ref{343}) and (\ref{346}) that $\sup\limits_{\epsilon}\int_{M_k }
	|F_{n,k,\epsilon,l}-(1-v'_{\epsilon}(\psi))F|^2 < +\infty$.
	Combining with
	\begin{equation}
		\sup\limits_{\epsilon}\int_{M_k }|(1-v'_{\epsilon}(\psi))F|^2
		\le
		\int_{M_k }\mathbb{I}_{\{\psi<-t_0\}}|F|^2< + \infty,
		\label{347}
	\end{equation}
	one can obtain that
	\begin{equation}\nonumber
		\sup\limits_{\epsilon}\int_{M_k }|F_{n,k,\epsilon,l}|^2
		< + \infty,
	\end{equation}
	which implies that there exists a subsequence of $\{F_{n,k,\epsilon,l}\}$ (also
	denoted by $\{F_{n,k,\epsilon,l}\}$) compactly convergent to a holomorphic
	$(n,0)$ form on $M_k$ denoted by $F_{n,k,l}$.
	\par
	Note that $\sup\limits_{\epsilon}\sup\limits_{M_k}
	e^{v_{\epsilon}(\psi)-\varphi_{l}}c_n(-v_\epsilon(\psi))< + \infty$ and
	$|F_{n,k,\epsilon,l}-(1-v'_\epsilon(\psi))F|^2
	\le 2(|F_{n,k,\epsilon,l}|^2+|\mathbb{I}_{\{\psi<-t_0\}}|F|^2)$, then it follows
	inequality (\ref{347}) and the dominated convergence theorem on any given $K\subset
	\subset M_k$ with dominant function
	\begin{equation}\nonumber
		2(\sup\limits_{\epsilon}\sup\limits_{K}(|F_{n,k,\epsilon,l}|^2)
		+\mathbb{I}_{\{\psi<-t_0\}}|F|^2)
		\sup\limits_{\epsilon}\sup\limits_{M_k}e^{v_\epsilon(\psi)-\varphi_l}c_n(-v_\epsilon(\psi))
	\end{equation}
	that
	\begin{equation}
		\begin{split}
			&\lim\limits_{\epsilon \to 0}
			\int_{K}
			|F_{n,k,\epsilon,l}-(1-v'_{\epsilon}(\psi))F|^2
			e^{v_\epsilon(\psi)-\varphi_{l}}c_n(-v_\epsilon(\psi))\\
			=&\int_{K}
			|F_{n,k,l}-(1-b(\psi))F|^2 e^{v(\psi)-\varphi_{l}}c_n(-v(\psi)).
		\end{split}
	\end{equation}
	\par
	Combining with inequality (\ref{343}) and (\ref{346}), one can obtain that
	\begin{equation}\nonumber
		\begin{split}
			&\int_{K}
			|F_{n,k,l}-(1-b(\psi))F|^2 e^{v{(\psi)}-\varphi_{l}}c_n(-v(\psi))\\
			\le&
			(\sup\limits_{M_k}e^{-u(-v(\psi))})\int_{M_k }
			\frac{1}{B}\mathbb{I}_{\{-t_0-B<\psi<-t_0\}}|F|^2 e^{-\varphi_{l}},
		\end{split}
	\end{equation}
	which implies
	\begin{equation}
		\begin{split}
			&\int_{M_k}
			|F_{n,k,l}-(1-b(\psi))F|^2 e^{v{(\psi)}-\varphi_{l}}c_n(-v(\psi))\\
			\le&
			(\sup\limits_{M_k}e^{-u(-v(\psi))})\int_{M_k }
			\frac{1}{B}\mathbb{I}_{\{-t_0-B<\psi<-t_0\}}|F|^2 e^{-\varphi_{l}}.
		\end{split}
		\label{352}
	\end{equation}
	\par
	\
	
	\emph{Step 7: when $l$ goes to $+\infty$.}
	\\
	\par
	Note that
	\begin{equation}\nonumber
		\begin{split}
			\int_{M_k }
			\frac{1}{B}\mathbb{I}_{\{-t_0-B<\psi<-t_0\}}|F|^2 e^{-\varphi_{l}}
			\le
			\int_{M_k }
			\frac{1}{B}\mathbb{I}_{\{-t_0-B<\psi<-t_0\}}|F|^2 e^{-\varphi}<+ \infty,
			\label{353}
		\end{split}
	\end{equation}
	and $\sup\limits_{M_k} e^{-u(-v(\psi))}< +\infty$,
	then it from (\ref{352}) that
	\begin{equation}\nonumber
		\sup\limits_l \int_{M_k}|F_{n,k,l}-(1-b(\psi))F|^2
		e^{v(\psi)-\varphi_{l}}c_n(-v(\psi)) < + \infty.
		\label{354}
	\end{equation}
	\par
	Combining with $\inf\limits_l \inf\limits_{M_k}
	e^{v(\psi)-\varphi_{l}}c_n(-v(\psi)) >0$, one can obtain that
	$\sup\limits_l \int_{M_k}|F_{n,k,l}-(1-b(\psi))F|^2 < +\infty$.
	Note that
	\begin{equation}
		\begin{split}
			\int_{M_k}|(1-b(\psi))F|^2\le \int_{M_k}\mathbb{I}_{\{\psi<-t_0\}}|F|^2< + \infty.
		\end{split}
		\label{355}
	\end{equation}
	\par
	Then $\sup\limits_l \int_{M_k}|F_{n,k,l}|^2 < +\infty$, which implies that
	there exists a compactly convergence subsequence of $\{F_{n,k,l}\}$ denoted by
	$\{F_{n,k,l'}\}$ which converge to a holomorphic $(n,0)$ form $F_{n,k}$ on $M_k$.
	\par
	Note that $\sup\limits_{M_k} e^{v(\psi)-\varphi_l}c_n(-v(\psi))<
	+\infty$, then it follows (\ref{355}) and the dominated convergence theorem on any given
	compact subset $K$ of $M_k$ with dominant function
	\begin{equation}\nonumber
		2(\sup\limits_{l'}\sup\limits_{K}(|F_{n,k,l'}|^2)
		+\mathbb{I}_{\{\psi<-t_0\}}|F|^2)
		\sup\limits_{M_k}e^{v(\psi)-\varphi_l}c_n(-v(\psi))
	\end{equation}
	that
	\begin{equation}
		\begin{split}
			&\lim\limits_{l' \to +\infty}
			\int_{K}
			|F_{n,k,l'}-(1-b(\psi))F|^2 e^{v(\psi)-\varphi_{l}}c_n(-v(\psi))
			\\
			=&\int_{K}
			|F_{n,k}-(1-b(\psi))F|^2 e^{v(\psi)-\varphi_{l}}c_n(-v(\psi)).
		\end{split}
		\label{357}
	\end{equation}
	\par
	Note that for $l'>l$, $\varphi_l' < \varphi_l $ holds, then it follows from (\ref{352})
	and (\ref{353}) that
	\begin{equation}
		\begin{split}
			&\lim\limits_{l' \to +\infty}
			\int_{K}
			|F_{n,k,l'}-(1-b(\psi))F|^2 e^{v(\psi)-\varphi_{l}}c_n(-v(\psi))
			\\
			\le&
			\limsup\limits_{l' \to +\infty}
			\int_{K}
			|F_{n,k,l'}-(1-b(\psi))F|^2 e^{v(\psi)-\varphi_{l'}}c_n(-v(\psi))\\
			\le&
			\limsup\limits_{l' \to +\infty}
			(\sup\limits_{M_k}e^{-u(-v(\psi))})\int_{M_k}
			\frac{1}{B}\mathbb{I}_{\{-t_0-B<\psi<-t_0\}}|F|^2e^{-\varphi_{l'}}\\
			\le&
			(\sup\limits_{M_k}e^{-u(-v(\psi))})\int_{M_k}
			\frac{1}{B}\mathbb{I}_{\{-t_0-B<\psi<-t_0\}}|F|^2e^{-\varphi}<+ \infty.
		\end{split}
	\end{equation}
	\par
	Combining with equality (\ref{357}), one can obtain that
	\begin{equation}\nonumber
		\begin{split}
			&\int_{K}
			|F_{n,k}-(1-b(\psi))F|^2 e^{v(\psi)-\varphi_{l}}c_n(-v(\psi))\\
			\le&
			(\sup\limits_{M_k}e^{-u(-v(\psi))})\int_{M_k}
			\frac{1}{B}\mathbb{I}_{\{-t_0-B<\psi<-t_0\}}|F|^2e^{-\varphi}<+ \infty,
		\end{split}
	\end{equation}
	for any compact subset $K$ of $M_k$, which implies
	\begin{equation}
		\begin{split}
			&\int_{M_k}
			|F_{n,k}-(1-b(\psi))F|^2 e^{v(\psi)-\varphi_{l}}c_n(-v(\psi))\\
			\le&
			(\sup\limits_{M_k}e^{-u(-v(\psi))})\int_{M_k}
			\frac{1}{B}\mathbb{I}_{\{-t_0-B<\psi<-t_0\}}|F|^2e^{-\varphi}<+ \infty.
		\end{split}
	\end{equation}
	\par
	When $l \to +\infty$, it follows from Levi's theorem that
	\begin{equation}
		\begin{split}
			&\int_{M_k}
			|F_{n,k}-(1-b(\psi))F|^2 e^{v(\psi)-\varphi}c_n(-v(\psi))\\
			\le&
			(\sup\limits_{M_k}e^{-u(-v(\psi))})\int_{M_k}
			\frac{1}{B}\mathbb{I}_{\{-t_0-B<\psi<-t_0\}}|F|^2e^{-\varphi}<+ \infty.
		\end{split}
		\label{464}
	\end{equation}
	\par
	\
	
	\emph{Step 8: ODE System.}
	\\
	\par
	Fix $n$, we want to find $\eta$ and $\phi$ such that
	$(\eta+g^{-1})=e^{-v_\epsilon(\psi_m)}e^{-\phi}\frac{1}{c_n(-v_{\epsilon}(\psi_m))}$.
	As $\eta=s(-v_{\epsilon}(\psi_m))$ and $\phi=u(-v_{\epsilon}(\psi_m))$, we have
	$(\eta+g^{-1})e^{v_\epsilon(\psi_m)}e^{\phi}=(s+\frac{s'^2}{u''s-s''})e^{-t}e^u\circ(-v_\epsilon(\psi_m))$.\\
	
	Summarizing the above discussion about $s$ and $u$, we are naturally led to a system of
	ODEs:
	\begin{equation}
		\begin{split}
			&1)(s+\frac{s'^2}{u''s-s''})e^{u-t}=\frac{1}{c_n(t)},\\
			&2)s'-su'=1,
		\end{split}
		\label{362}
	\end{equation}
	when $t\in(S,+\infty)$.\par
	We directly solve the ODE system (\ref{362}) and get
	$u(t)=-\log(\int^t_S c_n(t_1)e^{-t_1}dt_1)$ and
	$s(t)=\frac{\int^t_S(\int^{t_2}_S c_n(t_1)e^{-t_1}dt_1)dt_2}{\int^t_S
		c_n(t_1)e^{-t_1}dt_1)}$. It follows that $s\in C^{\infty}((S,+\infty))$ satisfies
	$s \ge 0$, $\lim\limits_{t \to +\infty}u(t)=-\log(\int^{+\infty}_S
	c_n(t_1)e^{-t_1}dt_1)$ exists and $u\in C^{\infty}((S,+\infty))$ satisfies
	$u''s-s''>0$.
	\par
	As $u(t)=-\log(\int^t_S c_n(t_1)e^{-t_1}dt_1)$ is decreasing with respect to t, then
	it follows from $-T \ge v(t) \ge \max\{t,-t_0-B_0\} \ge -t_0-B_0$, for any $t \le 0$
	that
	\begin{equation}
		\begin{split}
			\sup\limits_{M_k}e^{-u(-v(\psi))} \le
			\sup\limits_{M}e^{-u(-v(\psi))} \le
			\sup\limits_{t\in(S,t_0+B]}e^{-u(t)}=\int^{t_0+B}_S c_n(t_1)e^{-t_1}dt_1.
		\end{split}
	\end{equation}
	\par
	Hence on $M_k$, we have
	\begin{equation}
		\begin{split}
			&\int_{M_k}
			|F_{n,k}-(1-b(\psi))F|^2 e^{v(\psi)-\varphi}c_n(-v(\psi))\\
			\le&
			\int^{t_0+B}_S c_n(t_1)e^{-t_1}dt_1\int_{M_k}
			\frac{1}{B}\mathbb{I}_{\{-t_0-B<\psi<-t_0\}}|F|^2e^{-\varphi}\\
			\le&
			C\int^{t_0+B}_S c_n(t_1)e^{-t_1}dt_1.
			\label{566}
		\end{split}
	\end{equation}
	\par
	\
	
	\emph{Step 9: when $k$ goes to $+\infty$.}
	\\
	\par
	
	Note that for any given $k$, $e^{-\varphi+v(\psi)}c_n(-v(\psi))$ has a positive lower
	bound on $\overline{M_k}$, then it follows (\ref{566}) that for any given $k$ ,
	$\int_{M_k}
	|F_{n,k'}-(1-b(\psi))F|^2$ is bounded with respect to $k' \ge k$.
	Combining with
	\begin{equation}
		\int_{M_k}|1-b(\psi)F|^2 \le
		\int_{\overline{M_k}\cap\{\psi<-t_0\}}|F|^2<+ \infty,
		\label{567}
	\end{equation}
	One can obtain that $\int_{M_k}
	|F_{n,k'}|^2$ is bounded with respect to $k' \ge k$.
	\par
	By diagonal method, there exists a subsequence $F_{n,k''}$ uniformly converging on any
	$\overline{M_k}$ to a holomorphic $(n,0)$ form on M denoted by $F_{n}$. Then it
	follow from inequality (\ref{566}), (\ref{567}) and the dominated convergence theorem that
	\begin{equation}\nonumber
		\begin{split}
			\int_{M_k}
			|F_{n}-(1-b(\psi))F|^2 e^{-\max\{\varphi-v(\psi),-\alpha\}}c_n(-v(\psi))
			\le
			C\int^{t_0+B}_S c_n(t_1)e^{-t_1}dt_1,
		\end{split}
	\end{equation}
	for any $\alpha>0$, then Levi's theorem implies
	\begin{equation}\nonumber
		\begin{split}
			\int_{M_k}
			|F_{n}-(1-b(\psi))F|^2 e^{-(\varphi-v(\psi))}c_n(-v(\psi))
			\le
			C\int^{t_0+B}_S c_n(t_1)e^{-t_1}dt_1.
		\end{split}
	\end{equation}
	\par
	Let $k\to +\infty$, we get
	\begin{equation}
		\begin{split}
			\int_{M}
			|F_{n}-(1-b(\psi))F|^2 e^{-(\varphi-v(\psi))}c_n(-v(\psi))
			\le
			C\int^{t_0+B}_S c_n(t_1)e^{-t_1}dt_1.
			\label{estimate of Fn}
		\end{split}
	\end{equation}
	\par
	\
	
	\par
	\emph{Step 10: when $n$ goes to $+\infty$.}
	\\
	\par
	By construction of $c_n(t)$ in Step 1, we have
	\begin{equation}
		\begin{split}
			&\int_S^{t_0+B}c_n(t_1)dt_1\\
			=&\int_S^{t_0+B}\int_{\mathbb{R}}h(e^y(t_1-S)+S)g_n(y)dydt_1\\
			=&\int_{\mathbb{R}}e^{-y}g_n(y)\int_{S}^{(t_0+B-S)e^y+S}h(s)dsdy\\
			=&\int_{\mathbb{R}}e^{-y}g_n(y)dy\int_{S}^{t_0+B}h(s)ds+
			\int_{\mathbb{R}}e^{-y}g_n(y)\int_{t_0+B}^{(t_0+B-S)e^y+S}h(s)dsdy.
			\label{integral of cn}
		\end{split}
	\end{equation}
	As
	\begin{equation}\nonumber
		\begin{split}
			&\lim_{l\to+\infty}|\int_{\mathbb{R}}e^{-y}g_n(y)dy-1|\\
			\le&\lim_{l\to+\infty}|\int_{-\frac{1}{l}}^{0}(e^{-y}-1)g_l(y)dy|
			+\lim_{l\to+\infty}|\int^{\frac{1}{l}}_{0}e^{-y}g_l(y)dy|\\
			=&0
		\end{split}
	\end{equation}
	and
	\begin{equation}\nonumber
		\begin{split}
			&|\int_{\mathbb{R}}e^{-y}g_n(y)\int_{t_0+B}^{(t_0+B-S)e^y+S}h(s)dsdy|\\
			\le&e^{\frac{1}{l}}(1+(\frac{1}{l}))h((t_0+B-S)e^{-1}+S)(t_0+B-S)(e^{\frac{1}{l}}-e^{-\frac{1}{l}}),
		\end{split}
	\end{equation}
	then it follows from inequality \eqref{integral of cn} that
	\begin{equation}
		\begin{split}
			\lim_{n\to+\infty}\int_S^{t_0+B}c_n(t_1)dt_1=\int_{S}^{t_0+B}c(t_1)e^{-t_1}dt_1.
			\label{limitof integral of cn}
		\end{split}
	\end{equation}
	Combining with $\inf\limits_l\inf\limits_M e^{v(\psi)-\varphi}c_n(-v(\psi))\ge\inf\limits_M e^{v(\psi)-\varphi}c(-v(\psi))$, we have
	\begin{equation}\nonumber
		\begin{split}
			\sup\limits_n\int_M|F_n-(1-b(\psi))F|^2<+\infty.
		\end{split}
	\end{equation}
	Note that
	\begin{equation}\nonumber
		\begin{split}
			\int_M|(1-b(\psi))F|^2\le\int_M |\mathbb{I}_{\{\psi<-t_0\}}F|^2+\infty,
		\end{split}
	\end{equation}
	then $\sup\limits_l\int_M|F_n|^2<+\infty$, which implies that there exists a compactly convergent subsequence of $\{F_n\}$ (also denoted by $\{F_n\}$), which converges to a holomorphic $(n,0)$ form $\tilde{F}$ on $M$. Then it follows from inequality \eqref{estimate of Fn} and Fatou's Lemma that
	\begin{equation}\nonumber
		\begin{split}
			&\int_{M}
			|\tilde{F}-(1-b(\psi))F|^2 e^{-(\varphi-v(\psi))}c(-v(\psi))\\
			\le&\int_{M}|F-(1-b(\psi))F|^2 e^{-(\varphi-v(\psi))}c^{-}(-v(\psi))\\
			=&\int_{M}\liminf_{n\to+\infty}|F_{n}-(1-b(\psi))F|^2 e^{-(\varphi-v(\psi))}c_n(-v(\psi))\\
			\le&\liminf_{n\to+\infty}\int_{M}|F_{n}-(1-b(\psi))F|^2 e^{-(\varphi-v(\psi))}c_n(-v(\psi))\\
			\le&C\liminf_{n\to+\infty}\int^{t_0+B}_S c_n(t_1)e^{-t_1}dt_1\\
			=&C\int^{t_0+B}_S c(t_1)e^{-t_1}dt_1.
		\end{split}
	\end{equation}
	Lemma \ref{lemma2.1} is proved.

	
	\vspace{.1in} {\em Acknowledgements}.
The first author and the second author were supported by National Key R\&D Program of China 2021YFA1003100. The first author was supported by NSFC-12425101, NSFC-11825101, NSFC-11522101 and NSFC-11431013. The second author was supported by the Talent Fund of Beijing Jiaotong University.
The third author was supported by China Postdoctoral Science Foundation BX20230402 and 2023M743719.

	\bibliographystyle{references}\bibliography{xbib}

\end{document}